\documentclass[12pt]{amsart}
\usepackage{amsmath,amssymb,amsthm,amscd,stmaryrd,enumitem}
\usepackage{color}
\usepackage[mathscr]{euscript}
\usepackage[arrow,curve,frame,color]{xy}

\usepackage[usenames,dvipsnames]{xcolor}

\definecolor{gray}{gray}{0.7}

\newcommand{\slc}[1]{\langle #1\rangle} 
\newcommand{\cs}{\text{\rm c}}

\usepackage{hyperref}
\hypersetup{breaklinks=true}
\hypersetup{bookmarksdepth=3}

\usepackage[small,nohug,heads=vee]{diagrams}

\usepackage{graphicx}
\usepackage[labelformat=empty]{caption}
\usepackage{accents}

\usepackage[utf8]{inputenc} 
\usepackage[T1]{fontenc}
\usepackage{enumitem}

\numberwithin{equation}{section}
\theoremstyle{plain}

\makeatletter
\def\blfootnote{\gdef\@thefnmark{}\@footnotetext}
\makeatother

\newtheorem{theorem}[equation]{Theorem}
\newtheorem{thm}[equation]{Theorem}
\newtheorem{proposition}[equation]{Proposition}
\newtheorem{lemma}[equation]{Lemma}
\newtheorem{corollary}[equation]{Corollary}
\newtheorem{conjecture}[equation]{Conjecture}

\newtheorem{lem}[equation]{Lemma}

\newtheorem{prop}[equation]{Proposition}
\newtheorem{cor}[equation]{Corollary}

\newcommand{\sub}{\subset}
\newcommand{\bZ}{{\mathbf Z}}

\theoremstyle{remark}
\newtheorem{remark}[equation]{Remark}

\theoremstyle{definition}
\newtheorem{definition}[equation]{Definition}

\newtheorem*{question*}{Question}

\newcommand{\bb}[1]{\llbracket #1\rrbracket} 
\newcommand{\pp}[1]{\left\langle #1\right\rangle} 
\newcommand{\Lip}{\operatorname{Lip}}

\newcommand{\cD}{{\mathscr D}}

\newcommand{\vp}{\varphi}

\newcommand{\C}{{\mathcal C}}
\renewcommand{\d}{\mathcal{D}}

\newcommand{\ha}{\tilde{H}^{\text{AK}}}
\newcommand{\hl}{\tilde{H}^{\text{L}}}
\newcommand{\E}{\mathbb E}

\newcommand{\f}{\mathbb{F}}

\newcommand{\h}{{\mathcal H}}
\newcommand{\I}{{\textbf I}}

\renewcommand{\L}{{\mathcal L}}
\newcommand{\M}{{\mathbf M}}
\newcommand{\N}{\mathbb N}

\newcommand{\R}{\mathbb R}

\newcommand{\Z}{\mathbb Z}

\newcommand{\dd}{\mathfrak D}

\newcommand{\cP}{{\mathscr P}}
\newcommand{\bI}{{\mathbf I}}

\newcommand{\0}{\textbf{0}}

\newcommand{\al}{\alpha}
\newcommand{\be}{\beta}
\newcommand{\ben}{\begin{enumerate}}
\newcommand{\bit}{\begin{itemize}}

\newcommand{\cat}{\operatorname{CAT}}

\newcommand{\D}{\partial}
\newcommand{\de}{\delta}

\newcommand{\diam}{\operatorname{diam}}

\newcommand{\een}{\end{enumerate}}
\newcommand{\eit}{\end{itemize}}
\newcommand{\eps}{\epsilon}

\newcommand{\Fill}{\operatorname{Fill}}
\newcommand{\lip}{\operatorname{Lip}}

\newcommand{\ga}{\gamma}

\newcommand{\id}{\operatorname{id}}
\newcommand{\im}{\operatorname{Im}}

\newcommand{\isom}{\operatorname{Isom}}

\newcommand{\la}{\lambda}

\newcommand{\length}{\operatorname{length}}

\newcommand{\om}{\omega}

\newcommand{\on}{\:\mbox{\rule{0.1ex}{1.2ex}\rule{1.1ex}{0.1ex}}\:}

\newcommand{\ra}{\rightarrow}

\newcommand{\si}{\sigma}

\newcommand{\spt}{\operatorname{spt}}

\renewcommand{\th}{\theta}

\newcommand{\ul}{\underline}
\newcommand{\loc}{\text{\rm loc}}
\newcommand{\es}{\emptyset}

\newcommand{\Del}{\Delta}
\newcommand{\til}{\tilde}

\newcommand{\B}[2]{B_{#1}(#2)}   
\newcommand{\Sph}[2]{S_{#1}(#2)} 
\begin{document}

\title{Morse quasiflats I}
\author{Jingyin Huang}
\author{Bruce Kleiner}
\author{Stephan Stadler}
\thanks{The first author thanks the Max-Planck Institute for mathematics at Bonn, where part of this work was done. 
The second author was supported by NSF grants DMS-1711556
and DMS-2005553 and a Simons Collaboration grant. The third author was supported by  DFG grant SPP 2026.
}

\date{\today}
\maketitle

\begin{abstract}
This is the first in a series of papers concerned with Morse quasiflats, which are a generalization of Morse quasigeodesics to arbitrary dimension.  In this paper we introduce a number of alternative definitions, and under appropriate assumptions on the ambient space we show that they are equivalent and quasi-isometry invariant; we also give a variety of examples.  The second paper proves that Morse quasiflats are asymptotically conical and have canonically defined Tits boundaries; it also gives some first applications.

\end{abstract}

\tableofcontents

\section{Introduction}

\subsection{Background and overview}

Gromov-hyperbolicity has been a central concept in geometric group theory since it was first introduced in \cite{Gromov_hyp}.  Over the years, it has inspired  a large literature with variations on the original idea.   Roughly speaking two approaches have been used: the first is to identify hyperbolic features in settings which may be very far from hyperbolic, while the second is to relax Gromov-hyperbolicity to geometric conditions which still retain some weakly hyperbolic flavor.  Examples of the first include: 
\begin{itemize}
	\item Relative hyperbolicity \cite{Gromov_hyp,farb1998relatively,bowditch2012relatively,dructu2005tree,osin2006relatively}, 
	\item Notions of directional hyperbolicity, including rank 1/Morse/ contracting/sublinear geodesics and subsets, tree graded structure, hyperbolically embedded subgroups etc \cite{Ballmann_axial,bestvina2002bounded,dructu2010divergence,charney2014contracting,MR3737283,sisto2018contracting,sublinear_boundary,durham2015convex,cordes2017stability,kapovich1998quasi,dructu2005tree,dahmani2017hyperbolically};
	\item Generalizations of classical small cancellation theory \cite{ol1992periodic,delzant1996sous,champetier1994petite,gromov2003random,osin2007peripheral,delzant2008courbure,osin2010small,ol2009lacunary,coulon2011asphericity,wise2018structure,dahmani2017hyperbolically,coulon2019small};
	\item Statistical or probabilistic versions of hyperbolicity \cite{duchin2012statistical,chatterjee2021average};
	\item Cohomological aspects of negative
	curvature \cite{monod2006orbit,hamenstadt2008bounded,thom2009low};
	\item Analyzing certain non-hyperbolic solvable groups using elements of hyperbolic geometry \cite{farb1998rigidity,farb1999quasi,eskin2012coarse,eskin2013coarse};
	\item Acylindrical hyperbolicity \cite{bestvina2002bounded,bowditch2008tight,osin2016acylindrically,dahmani2017hyperbolically}, 
	\item The projection complex \cite{bestvina2015constructing},
	\item Hierarchically hyperbolicity \cite{masur2000geometry,kim2014geometry,behrstock2017hierarchically,behrstock2015hierarchically},
\end{itemize}
Examples of the second approach include:
\begin{itemize}
	\item Spaces with coning inequalities \cite{gromov1983filling,wenger2005isoperimetric}, 
	\item Combable, automatic and semi-hyperbolic groups \cite{MR1161694,gromov_asym,alonso1995semihyperbolic};
	\item Spaces with certain geodesic bicombing, injective metric spaces, $\cat(0)$ spaces and generalizations  \cite{busemann2012geometry,isbell1964six,dress1984trees,lang2013injective,descombes2015convex,kar2011asymptotically}, 
	\item Combinatorial non-positive curvature \cite{januszkiewicz2006simplicial,bandelt2008metric,osajda2013combinatorial,brevsar2013bucolic,chalopin2014weakly,hoda2017quadric,huang2019metric,joharinad2019topology},
	\item Ptolemy spaces \cite{foertsch2007nonpositive,foertsch2011hyperbolicity,foertsch2011group},
	\item Median graphs, median spaces and coarse median spaces \cite{bandelt1983median,verheul1993multimedians,gerasimov1997semi,gerasimov1998fixed,chepoi2000graphs,bandelt2008metric,chatterji2010kazhdan,bowditch2016some,bowditch2013coarse}, 
\end{itemize}

This is the first of a series of two papers where we introduce and study the notion of \emph{Morse quasiflats}, which are a higher dimensional generalization of Morse quasi-geodesics \cite{Ballmann_axial,kapovich1998quasi,dructu2010divergence}. (We caution the reader that our notion of a Morse quasiflat should not be confused with a quasiflat which is a ``Morse subset'' in the sense of \cite{genevois2017hyperbolicities}, see the footnote\footnote{To illustrate the difference, consider a copy of $\mathbb Z^2$ in $F_2\times \mathbb Z$. This is a Morse quasiflat in our sense, but not a Morse subset in the sense of \cite{genevois2017hyperbolicities}.  ``Morse subsets'' in \cite{genevois2017hyperbolicities} were referred to as ``strongly quasiconvex subsets'' in \cite{tran2019strongly}.  While we were aware of the potential confusion caused by our choice of terminology, we were not able to find a better alternative.} below.)   Like Morse quasigeodesics, the definition of Morse quasiflats axiomatizes the stability properties satisfied by quasigeodesics in Gromov hyperbolic spaces.    The following is an overview of our objectives. (1), (2) and (3) are done in this paper, which will be discussed in more detail in the remainder of the introduction. (4) and (5) are in the second part \cite{HKS}.  The second paper uses a few results from this paper, but it may be read independently if the reader is willing to take these results for granted.   
\ben
\item We introduce a number of potential definitions of Morse quasiflats. We find conditions which guarantee that the definitions are equivalent and quasi-isometry invariant.   
\item We establish stability properties for higher dimensional Morse quasiflats, including a generalization of the Morse lemma.
\item We provide criteria for quasiflats to be Morse, and give examples.  In $\cat(0)$ spaces, we give an easy-to-verify  flat half-space criterion which generalizes \cite{Ballmann_axial}.
\item In spaces with a convex geodesic bicombing, we prove that Morse quasiflats are asymptotically conical,  or, to put it in more analytical terms, they have unique tangent cones at infinity.  Consequently they have a well-defined Tits boundary.
\item We mention a few applications to quasi-isometric rigidity.
\een
We emphasize that although a few applications will be discussed in \cite{HKS}, our main objective here is to develop the basic theory of Morse quasiflats.  There are clearly many natural examples, and we expect the techniques here to be useful in proving further QI rigidity results.

We mention that geometric measure theory plays an important role in this paper, both as the language used to state results, and as a technical tool in the proofs.
In the process of developing our ideas, we were led to several questions about geometric measure theory in metric spaces which are not addressed by the existing literature.  We anticipate that further work in this topic will help to stimulate a fruitful and deeper interaction between geometric group theory and geometric measure theory.  See Subsection~\ref{subsec_GMT} for more discussion.

\subsection{Definitions of Morse quasiflats}
\label{subsec_def}

In this section we discuss a number of candidates for the definition of a Morse quasiflat. Several  of these are analogs of well-known definitions for Morse quasigeodesics, while others are not. In addition to the definitions mentioned here, there are other definitions appearing only in the body of the paper.

For the remainder of the introduction $X$ will denote a metric space, and $Q$ an $n$-quasiflat in $X$ -- the image of an $(L,A)$-quasi-isometric embedding $\Phi:\R^n\ra X$; for simplicity we will assume here that $\Phi$ is $L$-Lipschitz.  
More general cases where $Q$ may not be represented by a Lipschitz map will be discussed in the main body of the paper.   
We will denote by $F$ a  bilipschitz $n$-flat, i.e. the image of a bilipschitz embedding $\R^n\ra Z$, where $Z$ is a metric space.

The reader may find it helpful while reading this section to keep some simple but illustrative examples in mind, such as a copy of the vertex group $\mathbb Z^3$ in $\mathbb Z^3*_{\mathbb Z^2}\mathbb Z^4$, or a 2-dimensional flat in $\mathbb H^2\times\mathbb H^2$. More interesting examples are in Subsection~\ref{subsec_examples_morse_quasiflats_intro}.

We start by recalling several equivalent characterizations of Morse quasi-geodesics, which are axiomization of known properties of quasi-geodesics in hyperbolic spaces. We describe the main idea of each definition and give reference to the detailed formulation. See \cite{dructu2010divergence,behrstock2014divergence,charney2014contracting,MR3690269} for equivalence of these definitions.

Let $Q$ be a quasi-geodesic. 
\begin{enumerate}
	\item (Divergence) The cost of joining two points at distance $2r$ on $Q$, while avoiding the $r$-ball centered at their  (intrinsic) midpoint, is superlinear in $r$ 
	\cite{dructu2010divergence,behrstock2014divergence,charney2014contracting,MR3690269}.
	\item (Asymptotic cone) For every asymptotic cone $F:=Q_\om\subset X_\om=:Z$ the following holds (\cite{dructu2010divergence}). Every point $p\in F$ is a cut point of $Z$ which disconnects $F$ in $Z$.
	\item (Contracting property of projection) Let $\pi:X\to Q$ be a nearest point projection. Then the diameter of the $\pi$-image of any ball of radius $r$ disjoint from $Q$ is sublinear in $r$ \cite{MR3690269}.
	\item (Stability) For any $L,A>0$, the collection of all $(L,A)$-quasi-geodesics joining two points on $Q$ stays in a uniformly bounded neighborhood of $Q$ \cite{dructu2010divergence}. 
\end{enumerate}
We mention the following homological reformulation of the cut point condition in (2), in anticipation of the higher dimensional case, where homological language seems unavoidable.
\begin{enumerate}[label=(\alph*)]
	\item For any $p\in F$, the induced map in reduced homology 
	\begin{equation}
	\label{eq_homology}
	\tilde H_0(F\setminus\{p\})\ra \tilde H_0(Z\setminus \{p\})
	\end{equation}
	is injective. Equivalently, if $\al:I\ra F$ is a topological embedding and $\ga:I\ra Z$ is a path with the same endpoints, then $\ga(I)$ contains $\al(I)$. 
\end{enumerate}

Now we discuss higher dimensional versions of the above definitions, among which the formulation of the contraction properties is the least evident. The divergence condition  characterizes how hard it is to fill in a 0-cycle while avoiding a ball, hence naturally leads to the definition below in higher dimensions. Here we state the corresponding higher dimensional version for Lipschitz chains. However, more flexible notions are discussed in the text, cf. 
Definition~\ref{def_divergence1}, Definition~\ref{def_divergence2}.

\begin{definition}
	\label{def_divergence_intro}
	The $n$-quasiflat $Q=\Phi(\R^n)$ has ($\delta$-){\em super-Euclidean divergence}, 
	if for any $D>1$, there exists a function $\delta=\delta_{D}:[0,\infty)\to[0,\infty)$ with $\lim\limits_{r\to\infty}\delta(r)=+\infty$ such that the following property holds.
	
	Pick $r>0$ and $x\in \R^n$. Suppose that $\hat\si$ is an $(n-1)$-dimensional Lipschitz cycle carried by $\R^n\setminus B_x(r)$ which represents a nontrivial class in the reduced homology group $\tilde H_{n-1}(\R^n\setminus B_x(r))$, and assume that  the mass $\M(\hat\si)$ of $\hat\si$ satisfies $\M(\hat\si)\leq D\cdot r^{n-1}$. Then for any Lipschitz chain $\tau$ carried by  $X\setminus B_{\Phi(x)}(\frac{r}{D})$  such that $\partial\tau=\Phi_*(\hat\si)$, we have 
	$\M(\tau)\geq\delta(r)\cdot r^n$.
\end{definition}

The mass $\M(\sigma)$ of a Lipschitz $k$-chain $\sigma$ takes into account cancellation; in our present setting, one may think of $\M(\sigma)$ as the $k$-dimensional volume of the image, counted with multiplicity (cf. Section~\ref{sec_prelim}). 

\begin{remark}
	The notion of ``divergence'' has been studied in a more general context, where roughly speaking it measures how hard it is to fill in cycles outside larger and larger balls, see \cite{gersten1994quadratic,brady1998filling,wenger2006filling,abrams2013pushing,mcgdivergence}. One can either define divergence homotopically (filling spheres by disks) or homologically (filling cycles by chains), using cellular chains, or Lipschitz chains, or integral currents. 
\end{remark}

\bigskip 

We now consider the higher dimensional generalization of the asymptotic cone condition (2) above (compare (\ref{eq_homology}) above).  For this, we require that every asymptotic cone $Q_\om\subset X_\om$ of an $n$-quasiflat $Q\subset X$ satisfies the following higher dimensional analog of \eqref{eq_homology}:  

\begin{definition}\label{def_rigid_intro}(Compare  \cite[Definition 3.8]{kapovich1998quasi})
	A bilipschitz $n$-flat $F\subset Z$ is \emph{rigid}, if for every topological embedding $\al:D^n\ra F$ of the $n$-disk and every singular $n$-chain 
	$\tau$ with  $\D\tau=k\cdot\al_*(\D[D^n])$ for some $k\neq 0$, we have  $\im\al\subset\im \tau$.  Here $\im\tau$ refers to the union of images of corresponding singular simplices.  Equivalently, the inclusion map induces an injective mapping 
	\begin{equation}
	\label{eqn_reduced_homology_injective}
	\tilde H_{n-1}(F\setminus \{p\})\ra \tilde H_{n-1}(Z\setminus \{p\})
	\end{equation} 
	in reduced homology for any $p\in F$.
\end{definition}

\begin{remark}\label{def_full_support_intro}
	Note that when $H_n(Z)=\{0\}$ then injectivity of (\ref{eqn_reduced_homology_injective}) is equivalent to the injectivity of the induced homomorphism in relative homology 
	$$
	H_n(F,F\setminus\{p\},\mathbb Z)\to H_n(Z,Z\setminus\{p\},\mathbb Z)\,.
	$$ 
\end{remark}

\bigskip
We point out that Definitions~\ref{def_divergence_intro} and \ref{def_rigid_intro} are related to previous work (see \cite{mcgdivergence,{kapovich1998quasi}}), however those papers had different objectives, and did not define notions equivalent to Morse quasiflats.

We now turn to the contracting property.  To motivate our generalization to the higher dimensional case, rather than using the contracting property as formulated in (3) above, we use the following variation which is equivalent to (3) by \cite[Thoerem 1.4 (4)]{MR3690269}:

$(3)':$ For every $\rho\in (0,1)$ there is a function $f_\rho:[0,\infty)\ra[0,\infty)$ such that $\frac{f_\rho(R)}{R}\ra 0$ as $R\ra\infty$ with the following property.   For every $R\in (0,\infty)$, and any two points $x,y\in X$ which are joined by a path $\tau$ of length $\leq C\cdot R$ which lies outside the $\rho R$-neighborhood of $Q$, the distance between the projections satisfies $d(\pi(x),\pi(y))\leq f_\rho(R)$. Here $\pi$ denotes a nearest point projection $X\ra Q$.

We would like to reinterpret the above property homologically by treating $\{x,y\}$ as a reduced 0-cycle $\si:=x-y\in\tilde Z_0(X)$, and $\tau$ as a 1-chain with $\D \tau=\si$. We think of $\si'=\pi(x)-\pi(y)$ as the projection of the cycle $\si$. Then $(3)'$ implies $\si'$ can be filled by a chain with much smaller mass.

To formulate the higher dimensional version, let $\si$ be a cycle with $\dim(\si)=\dim(Q)-1$ such that:
\bit
\item $\M(\si)\le C \cdot R^{n-1}$, $\diam(\im(\si))\le R$ and $d(\im(\si),Q)\le R$;
\item $\si$ is filled by a Lipschitz chain $\tau$ with mass $\le C\cdot R^n$ such that $\tau$ lies outside the $\rho R$-neighborhood of $Q$ and $\im(\tau)\subset N_{R}(\im(\si))$.
\eit

Now we define the projection $\si'$ of $\si$ onto $Q$. Instead of considering the nearest point projection, which might not be well-behaved in the higher-dimensional case, we consider the following relative Plateau problem. Let $\mathcal{C}$ be the collection of all $n$-chains $\al$ such that $\D\al=\si+\si_Q$ where $\si_Q$ is a cycle living in the bilipschitz flat $Q$. Let $\al_0$ be an element with minimal mass in $\mathcal{C}$ and define $\si'=\D\al_0-\si$.

We say $Q$ has the \emph{cycle contracting property} if for each $\rho\in(0,1)$ there exists a sublinear $f_\rho$ as above such that $\si'$ can be filled by a chain with mass $\le f_\rho(R)\cdot R^{n-1}$.

In the 1-dimensional case, if $d(x,Q)$ is relatively small compared to $d(x,y)$, then $\si'=\pi(x')-\pi(y')$, which agrees with the nearest point projection. Otherwise we have $\si'=0$ and the above condition holds trivially. Thus even in the 1-dimensional case, our condition looks weaker than the usual contracting property. Nonetheless, it is strong enough to characterize ``Morse'' quasiflats.

Here we have omitted  some technical details such as the existence of the minimizer $\si'$. See Definition~\ref{def_ccp} for a precise formulation.

\bigskip

We turn to stability. A naive generalization of (4) would be  to require any quasidisk $D$ to be contained in a small neighborhood of $Q$ given that $\D D$ does so. However, this is not a useful definition for the following reason. To exploit condition (4) in the case of Morse quasigeodesics, one takes a pair of points $x,y\in Q$,  joins them by a geodesic segment $\ga$, and applies (4).  To imitate this in the higher dimensional case, one would want to fill in spheres by quasidisks; unfortunately, there is no good procedure for producing such fillings in general metric spaces.  Instead we consider the following variation, which is equivalent to (4) by \cite[Proposition 3.24]{dructu2010divergence}.

$(4)':$ For $x,y\in Q$ with distance $\leq R$, let $\om_{xy}$ be the sub-segment of $Q$ from $x$ to $y$ and let $\omega$ be a path from $x$ to $y$ of length $\le C\cdot R$. Then $\om_{xy}$ is contained in the $f(R)$-neighborhood of $\omega$ where $f$ is sublinear.

In the higher dimensional case, we replace $\{x,y\}$ by a codimension-1 cycle $\si$ of the bilipschitz flat $Q$, replace $\om_{xy}$ by the canonical filling $\tau$ of $\si$ inside $Q$, replace $\omega$ by any chain $\tau'$ with appropriate mass bound such that $\D\tau'=\si$, and require that $\im\tau'$ ``covers'' $\im\tau$. We refer the reader to the notion of $(\mu,b)$-rigid in Definition~\ref{def_rigid} for a precise formulation.

\bigskip

Now we discuss several definitions of Morse quasiflats not previously considered for quasi-geodesics. To motivate these, we go back to condition (a) on the asymptotic cone in the 1-dimensional case.

Suppose (a) holds, and consider a Lipschitz embedding $\al:I\ra F$ and a Lipschitz path $\ga:I\ra Z$ with the same endpoints.  Then the inverse image $\ga^{-1}(Z\setminus F)$ will be a disjoint union of a collection $\{I_j\}$ of at most countably many open intervals.  Restricting $\ga$ to the closure of $I_j$ we obtain a path $\ga_j:\bar I_j\ra Z$ which must actually be a loop in view of condition (a).  We may form a new path $\hat \ga:\hat I \ra Z$ by collapsing each of the closed intervals $\{\bar I_j\}$ to a point; note that $\hat \ga(\hat I)\subset F$.  Heuristically, we would like to assert that we have a decomposition
$$
\ga=\hat \ga+\sum_j\ga_j
$$
where $\sum_j\ga_j$ is an infinite sum of cycles, and the equality holds up to (possibly infinite) decomposition and reparametrization.  Also, since $\hat \ga$ and $\al$ are Lipschitz paths in $F\simeq \R$ with the same endpoints, then we should have (heuristically speaking) that $\hat\ga=\al$ modulo cancellation; here the reader may think of this as being analogous to the cancellation which occurs when adding simplicial chains.  Summing up, we arrive at a second condition which one may impose on asymptotic cones $F\subset Z$:
\begin{enumerate}[label=(\alph*)]
	\setcounter{enumi}{1}
	\item (Informal) For every Lipschitz embedding $\al:I\ra F$ and Lipschitz path $\ga:I\ra Z$ with the same endpoints,  we have 
	$$
	\length(\ga)=\length(\al)+\length(\ga-\al)\,,
	$$
	where length is measured after cancellation.
	\een
	
	To make Condition (b) rigorous and generalize it to higher dimensions  we use integral currents and mass rather than Lipschitz chains and length, so that infinite sums become valid and cancellation is taken into account.    For the purpose of this introduction, the space of \emph{integral $n$-currents} $\I_n(X)$ can be thought of as the closure of Lipschitz $n$-chains with respect to 
	an appropriate norm, which is derived from mass. See Subsection~\ref{subsect:currents} for precise definitions and basic properties. For now it is only important that integral currents form a chain complex whose homology is isomorphic to singular homology, at least for reasonable spaces, and that a chain $\tau\in\I_n(X)$ comes with a natural associated locally finite mass measure $\|\tau\|$.

	Our second condition to impose on asymptotic cones is as follows:
	\begin{definition}[Piece property]
		\label{def_piece_decomposition_intro}
		A bilipschitz $n$-flat $F$ has the {\em piece property} if for any integral $n$-current $\nu$ supported in $F$,   every filling $\tau$ of $\D\nu$ contains $\nu$ as a piece, i.e. the decomposition $\tau=\nu+(\tau-\nu)$ is additive with respect to mass:
		$$
		\M(\tau)=\M(\nu)+\M(\tau-\nu)\,;
		$$
		furthermore, the mass measure $\|\tau-\nu\|$ is concentrated on $Z\setminus F$.
		
	\end{definition}
	Here a {\em filling} of an integral cycle $\al$ is an integral current $\be$ with $\D\be=\al$.
	
	Returning again to the $1$-dimensional case, suppose $F\subset Z$ is a bilipschitz  $1$-flat satisfying Condition (a).  Let $\ga:I\ra Z$ be a Lipschitz path in the complement of $F$ whose endpoints lie in the $\rho$-neighborhood $N_\rho(F)$.  Then we may concatenate $\ga$ with two paths $\al_\pm$ of length $<\rho$ to obtain a new path $\ga'$ meeting $F$ only at its endpoints.  Condition (a) implies that the endpoints of $\ga'$ are the same, and hence the concatenation of $\al_\pm$ gives a path of length $< 2\rho$ joining the endpoints of $\ga$.   The higher dimensional generalization of this is as follows:

	\begin{definition}\label{def_lipschitz_neck}
		A bilipschitz $n$-flat $F$ in a metric space $Z$ has the \emph{neck property}, if there exists a constant $C>0$ such that the following holds for all
		$\rho>0$.  If $\si$ is an integral $(n-1)$-cycle supported in the $\rho$-neighborhood $N_\rho(F)$ of $F$ and $\si$ has a filling supported in $Z\setminus F$, then it has a filling $\tau$ with
		\[\M(\tau)\leq C\cdot\rho\cdot\M(\si).\]
	\end{definition}

	In addition to rigidity, the piece property, and the neck property, there are other conditions --  the full support property and the weak neck property -- which we will define in the body of the paper.  
	
	As mentioned above, each of the above definitions for bilipschitz flats gives rise to a notion for quasiflats: a quasiflat $Q$ in a metric space $X$ 
	is \emph{ asymptotically  rigid/has the asymptotic piece property/has the asymptotic neck property} if the corresponding property holds
	in every asymptotic cone $Q_\om\subset X_\om$.  
	
	One may produce still more conditions by converting Definitions \ref{def_piece_decomposition_intro}, \ref{def_lipschitz_neck} to quantitive conditions in the space $X$, without reference to asymptotic cones.   
	Instead of stating each counterpart of the notions above, we refer to the precise definitions in the later chapters: 
	Definition~\ref{def_neckp} (coarse neck property) and  Definition~\ref{def_tcp} (coarse piece property). 
	
	\subsection{An informal discussion on subsets of Morse degree $k$.}
	In the previous subsection we introduced a collection of properties for quasiflats. Some of these properties can be naturally formulated for more general subsets, leading to the notion of ``Morse degree'' for subsets. For example, based on the coarse neck property, we arrive at the following tentative definition.  
	
	\begin{definition}
		Let $k\ge 1$ be an integer. A subset $Y\subset X$ has \emph{Morse degree $k$} if there is a constant $C_0>0$ such that the following holds. 
		
		For given positive constants $C$ and $\rho$, 
		there exists $\ul R=\ul R(C,\rho,d(p,Y))$ such that for any $R\ge \ul R$ and any $\tau\in \I_k(B_p(CR)\setminus N_{\rho R}(Y))$ satisfying
		\bit
		\item $\M(\tau)\le C\cdot R^k$;
		\item $\sigma:=\partial\tau\in \I_{k-1}(X)$, $\spt(\si)\subset N_{2\rho R}(Y)$ and $\M(\si)\le C\cdot R^{k-1}$;
		\eit
		we have 
		\[\Fill(\si)\le C_0\cdot \rho R\cdot \M(\si).\] 
	\end{definition}
	
	Variations of this definition can be obtained by adjusting other properties from Section~\ref{subsec_def}. Such a definition incorporates several previously investigated notions/examples:
	\begin{itemize}[leftmargin=*]
		\item A Morse quasigeodesic has Morse degree $1$. 
		\item A ``Morse subset'' as defined in \cite{genevois2017hyperbolicities,tran2019strongly} has Morse degree $1$.
		\item The discussion in Section~\ref{subsec_def} is intended to characterize $n$-quasiflats which have Morse degree $n$.
		\item If $m\le n$, then a copy of $\mathbb Z^n$ in $\mathbb Z^n*_{\mathbb Z^{m-1}}\mathbb Z^n$ has Morse degree $m$.
		\item $\mathbb H^2\times \mathbb R$ has Morse degree $2$ in $\mathbb H^2\times \mathbb H^2$.
	\end{itemize}
	However, in this paper, we will focus on quasiflats.

	\subsection{Relations between different definitions of Morse quasiflats} 
	We now show that the definitions given in Section~\ref{subsec_def} become equivalent under appropriate assumptions on the ambient space $X$.  We refer the reader to Section~\ref{sec_prelim} for definitions.
	
	In the study of Morse quasi-geodesics, it is often assumed that the underlying metric space is geodesic or quasi-geodesic.  When working with $n$-dimensional quasiflats, the analogous assumption is that the metric space satisfies coning inequalities for integral currents up to dimension $(n-1)$, in fact we will sometimes need stronger assumptions.

	\begin{definition}
		\label{def_coning_ineq_intro}
		A complete metric space $X$ satisfies {\em coning inequalities up to dimension $n$ (or satisfies condition~(CI$_n$))}, if there exists a constant $c_0>0$, such that for all $0\leq k\leq n$ and every cycle 
		$S\in\I_{k}(X)$ with bounded support there exists a filling
		$T\in\I_{k+1}(X)$ with 
		\[\M(T)\leq c_0\cdot\diam(\spt (S))\cdot\M(S).\]
		If in addition, there exists a constant $c_1>0$ such that $T$ can always be chosen to fulfill
		\[\diam(\spt (T))\leq c_1\cdot \diam(\spt (S)),\]
		then $X$ is said to satisfy {\em strong coning inequalities} up to dimension $n$, abbreviated condition~(SCI$_n$).
	\end{definition}

	\begin{thm}[Theorem~\ref{thm:equivalence}]
		\label{thm:equivalence intro}
		Let $Q\subset X$ be an $(L,A)$-quasiflat in a proper metric space $X$. Consider the following conditions:
		\begin{enumerate}
			\item $Q$ is $(\mu,b)$-rigid (cf. Definition~\ref{def_rigid}).
			\item $Q$ has super-Euclidean divergence (cf. Definition~\ref{def_divergence2}).
			\item $Q$ has the cycle contracting property (cf. Definition~\ref{def_ccp}).
			\item $Q$ has the coarse neck property (cf. Definition~\ref{def_tcp}).
			\item $Q$ has the coarse piece property (cf. Definition~\ref{def_neckp}).
		\end{enumerate}
		Then the following hold:
		\begin{enumerate}[label=(\alph*)]
			\item $(1)$ and $(2)$ are equivalent if $X$ satisfies {\rm (CI$_{n-1}$)};
			\item $(1),(2),(3)$ and $(4)$ are equivalent if $X$ satisfies {\rm (SCI$_{n-1}$)};
			\item $(1),(2),(3),(4)$ and $(5)$ are equivalent if $X$ satisfies {\rm (SCI$_{n}$)}.
		\end{enumerate}
		Moreover, all of the above equivalences give uniform control on the parameters (e.g. if $X$ satisfies {\rm (SCI$_{n-1}$)} and $Q$ has super-Euclidean divergence, 
		then $Q$ has the coarse neck property (CNP) with parameters of CNP depending only on parameters of super-Euclidean divergence and $X$).
	\end{thm}

	\begin{thm}[Theorem~\ref{thm:conditions in asymptotic cones body}]
		\label{thm:conditions in asymptotic cones}
		Suppose $X$ is a proper metric space satisfying {\rm (SCI$_{n}$)}. Let $Q\subset X$ be an $n$-dimensional quasiflat. Suppose in addition that any asymptotic cone of $X$ has a Lipschitz combing (see Section~\ref{sec_prelim}). 
		Then the following conditions are equivalent and each of them is equivalent to the conditions in Theorem~\ref{thm:equivalence intro}.
		\begin{enumerate}
			\item  $Q\subset X$ has the asymptotic piece property (Definition~\ref{def_piece_decomposition}).
			\item  $Q\subset X$ has the asymptotic neck property (Definition~\ref{def_neck_decomposition}).
			\item  $Q\subset X$ has the asymptotic weak neck property (Definition~\ref{def_weak_neck_decomposition}).
			\item $Q\subset X$ has the asymptotic full support property (Definition~\ref{def_full_support}) with respect to reduced singular homology.
			\item $Q\subset X$ has the asymptotic full support property with respect to  reduced homology induced by  Ambrosio-Kirchheim currents.
		\end{enumerate}
		All asymptotic cones here are taken with base points inside $Q$. 
	\end{thm}
	
	Some of the equivalences are proved under weaker assumptions. Actually, we only need the asymptotic cone to satisfy certain coning inequalities, see Theorem~\ref{thm:conditions in asymptotic cones body} for the detailed set of assumptions. Also Theorem~\ref{thm:conditions in asymptotic cones body} is formulated more generally for a family of quasiflats rather than a single quasiflat.

	Despite the various equivalent characterizations given in the above theorems, we formally define a quasiflat in a metric space to be \emph{Morse} if it has super-Euclidean divergence in the sense of Definition~\ref{def_divergence2}.
	
	An immediate consequence of Theorem~\ref{thm:conditions in asymptotic cones} is:
	\begin{corollary} 
		A quasi-isometry $X\ra X'$ maps Morse $n$-quasiflats to Morse $n$-quasiflats if $X$ and $X'$ satisfy (SCI$_{n}$) and their asymptotic cones have Lipschitz combings. 
	\end{corollary} 
	See Proposition~\ref{prop:QI invariance} for quasi-isometry invariance of  the super-Euclidean divergence property of quasiflats under weaker assumptions.
	
	\begin{remark}
		\label{rmk:combable}
		Theorem~\ref{thm:equivalence intro} and Theorem~\ref{thm:conditions in asymptotic cones} apply to all groups which admit a quasi-geodesic combing \cite{MR1161694}, including all automatic groups. 
		More precisely, for any group $G$ with a quasi-geodesic combing and any integer $n>0$, we can always find a complex $X$ where $G$ acts properly and cocompactly such that $X$ satisfies (SCI$_{n}$), moreover, 
		the quasi-geodesic combing in $G$ passes to a Lipschitz combing on the asymptotic cone.
	\end{remark}

	Now we give more criteria for Morse quasiflats in $\cat(0)$ spaces.

	\begin{prop}[Proposition~\ref{lem_coefficient}]
		\label{lem_coefficient intro}
		Let $Q$ be an $n$-quasiflat in a $\cat(0)$ space $X$. Let $\f$ be a non-trivial abelian group. Consider the following conditions.
		\ben
		\item There exists an asymptotic cone $X_\om$ of $X$ and $p_\om\in Q_\om$ such that the map $H_n(Q_\om,Q_\om\setminus\{p_\om\},\f)\to H_n(X_\om,X_\om\setminus\{p_\om\},\f)$ is not injective.  
		\item There exists an asymptotic cone $X_\om$ of $X$ such that the limit $Q_\om$ of $Q$ is an $n$-flat which bounds a flat half-space in $X_\om$. 
		\item There exists an ultralimit $X_\om=\lim_{\om}(X,p_i)$ such that the limit $Q_\om$ of $Q$ is an $n$-flat which bounds a flat half-space in $X_\om$. 
		\een
		Then (1) and (2) are equivalent. If we assume in addition that $Q$ is a flat, then all three conditions are equivalent.
	\end{prop}
	
	\subsection{Stability of Morse quasiflats and quasidisks}

	In the following we use $d_H$ to denote Hausdorff distance.
	
	\begin{prop}[Morse lemma for Morse disks]
		\label{prop_Morse_lemma_for_quasidisks_intro}
		Suppose $X$ is a complete metric space satisfying condition ~{\rm (CI$_{n}$)}. Given $(\mu,b)$ as in Definition~\ref{def_rigid} and positive constants $L,A,A',n$, there exists $C$ depending only on $\mu,b,L,A,A',n$ and $X$ such that the following holds.
		
		Let $D$ and $D'$ be two $n$-dimensional $(L,A)$-quasi-disks in $X$ such that $d_H(\D D,\D D')<A'$ and $D$ is $(\mu,b)$-rigid. Then $d_H(D,D')<C$.
	\end{prop}
	
	In \cite[Section 5]{higherrank}, stronger versions of Morse lemmas for top rank quasi-minimizers are proved.
	
	\begin{prop}\label{prop_divergence_intro}
		Let $X$ be a complete metric space satisfying condition ~{\rm (CI$_{n}$)} and
		let $Q, Q'\subset X$ be $(L,A)$-quasiflats with $\dim Q=n$. Suppose that $Q$ is $(\mu,b)$-rigid. Then there exist $A'$ and $\eps$ depending only on $X,L,A,n,b$ and $\mu$ such that either $d_H(Q,Q')\le A'$, or 
		\begin{equation}
		\label{eq_linear_div_intro}
		\limsup_{r\to\infty}\frac{d_H(B_p(r)\cap Q, B_p(r)\cap Q')}{r}\ge \eps
		\end{equation}
		for some (hence any) $p\in Q$.
	\end{prop}

	\subsection{Examples of Morse quasiflats}
	\label{subsec_examples_morse_quasiflats_intro}
	Recall that we have defined a quasiflat in a metric space $X$ to be \emph{Morse} if it has super-Euclidean divergence in the sense of Definition~\ref{def_divergence1}. There are many examples of Morse quasiflats which are neither 1-dimensional nor of top rank. In the following list we assume $X$, $X_1$ and $X_2$ satisfy the hypotheses of Theorem~\ref{thm:conditions in asymptotic cones} (this applies to combable groups, spaces with Lipschitz combing etc.), where being Morse can be characterized by other conditions in Theorem~\ref{thm:conditions in asymptotic cones}.
	\begin{itemize}
		\item Suppose $X$ has asymptotic rank $\le n$ (\cite{higherrank}). Then every $n$-dimensional quasiflat in $X$ is Morse.
		\item Suppose $Q_i\subset X_i$ is a Morse quasiflat for $i=1,2$. Then $Q_1\times Q_2$ is a Morse quasiflat in $X_1\times X_2$ (see Corollary~\ref{cor:products}).
		\item One can obtain Morse quasiflats from more interesting operations, like amalgamation or taking branched covers under certain conditions. We refer to Section~\ref{subsec:example} for an example using branched covers.
	\end{itemize}
	
	Morse quasiflats also arise naturally in some group theoretic contexts. The following is a special case of Proposition~\ref{lem_coefficient}.
	
	\begin{cor}\label{cor_Morse_crit_intro}
		Suppose $X$ is a proper $\cat(0)$ space and $F\subset X$ is a flat such that the stabilizer of $F$ in $\isom(X)$ acts cocompactly on $F$. 
		Then $F$ is Morse if and only if $F$ does not bound an isometrically embedded half-flat.
	\end{cor}

	Let $G$ be a group. 
	Following \cite{wise2017cubical} we say that a finitely generated free abelian subgroup $H$ of $G$ is \emph{highest} if its commensurability class is maximal, i.e. $H$ does not have a finite index subgroup that is contained in a free abelian subgroup of higher rank.
	
	\begin{cor}[Theorem~\ref{thm_virtually_special}]
		\label{cor:special}
		Let $G$ be a group such that $G$ has a finite index subgroup which is the fundamental group of a compact special cube complex. Then any highest free abelian subgroup of $G$ is a Morse quasiflat.
	\end{cor}
	
	One cannot drop the word ``special'' from the above corollary because of certain irreducible lattices acting on products of trees \cite{RattaggiRobertson}. For groups acting on higher rank Euclidean buildings or symmetric spaces, it might happen quite often that highest abelian subgroups are not Morse quasiflats, and it would be interesting to identify which number-theoretical invariants lead to this as in \cite{RattaggiRobertson}.
	
	Morse quasiflats also arise naturally in mapping class groups:
	\begin{thm}
		\label{cor_highest_abelian_mcg}
		Suppose $G$ is a mapping class group of a surface. Then any highest abelian subgroup in $G$ is a Morse quasiflat.
	\end{thm}
	Behrstock and Drutu \cite{mcgdivergence} studied highest abelian subgroups in order to give a lower bound for divergence functions, and their argument is at the heart of the proof of Theorem~\ref{cor_highest_abelian_mcg}.
	
	\begin{proof}
		By \cite{birman1983abelian}, any highest abelian subgroup in $G$ satisfies the assumption of \cite[Proposition 3.5]{mcgdivergence}, hence has super-Euclidean divergence, where the notion of divergence in \cite[Proposition 3.5]{mcgdivergence} uses a cellular setting and considers filling spheres by disks.  To see that this is equivalent to the notion of super-Euclidean divergence from Definition~\ref{def_divergence_intro}, recall that mapping class groups are coarse median \cite{behrstock2011centroids}, hence their aymptotic cones are bilipschitz to $\cat(0)$ spaces \cite{bowditch2013coarse}. Now the equivalence follows from Corollary~\ref{cor:bilip}, Theorem~\ref{thm:conditions in asymptotic cones} and Remark~\ref{rmk:combable}.
	\end{proof}
	
We expect that the analog of Corollary \ref{cor:special} holds for more examples.
	\begin{conjecture}
		Suppose $G$ is either a Coxeter group, an Artin group, or the fundamental group of the complement of a complexified real simplicial arrangement of hyperplanes as in Deligne \cite{deligne1972immeubles}.  Then any highest abelian subgroup in $G$ is a Morse quasiflat.
	\end{conjecture}
We believe the case of Coxeter groups is accessible, and follows from known techniques.  On the other hand, it appears that substantial new ingredients would be required to treat the other two cases.

	\subsection{The role of geometric measure theory in this paper}\label{subsec_GMT}
	We now make a few remarks about analytical aspects of this paper.
	
	In the initial stages of this project, we had considered bypassing integral currents in favor of a technically simpler alternative such as simplicial chains or Lipschitz chains.  However, our attempts to implement such alternative approaches failed, either because they could only be used to treat a fraction of our results, or because they would necessitate the replacement of standard technicalities from geometric measure theory by new technical arguments of similar complexity, thereby eliminating the expected benefit of avoiding geometric measure theory.  Here are some of the properties that a geometric chain complex should have in order to facilitate the objectives of this paper:
	\bit
	\item There should be a notion of ``cancellation'', when one adds chains which overlap with opposite orientation.
	\item Slicing of chains by Lipschitz functions, or at least distance functions, is basic to many comparison arguments.
	\item Minimizers should be compatible with $\cat(0)$ structure which may not be polyhedral.
	\item Natural limit spaces (Tits boundaries, asymptotic cones) are not PL in any sense, so one is forced to consider chain complexes in metric spaces.
	\item Compactness is needed in a number of places.  
	\eit
	We were not able to find a viable alternative to integral currents which was any simpler, and still had the above properties.
	
	In addition to  providing the basic framework for our results,  integral currents helped to simplify our arguments in  several places.  In particular, the estimates we obtained by solving  Plateau problems relative to a Lipschitz submanifold (or quasiflat) were  very helpful in Sections \ref{sec_approx} and \ref{sec_cycles_close_to_Morse_quasiflats}, and may be of independent interest.  Also, following earlier authors \cite{ambrosio2000currents,wenger2011compactness}, we use the Ekeland variational principle in Section \ref{sec_approx}.

	\subsection{Structure of the paper and suggestions for the reader}
	We are mostly working under the assumption that the ambient space satisfies coning inequality. However, quasiflats in such spaces might not be represented by Lipschitz maps, which brings complications at several places in the proofs. The reader might find it helpful to assume quasiflats are represented by Lipschitz maps, to avoid some technicalities. 
	
	The reader might start with Section~\ref{sec_definition}, and consult earlier sections whenever needed.
	
	In Section~\ref{sec_proofs} we discuss ideas of proofs of some of the main results. 
	In Section~\ref{sec_prelim} we discuss some background on metric spaces and metric currents. In Section~\ref{sec_approx} we discuss a procedure of ``regularizing'' an integral current for later use. 
	In Section~\ref{sec:quasiflats in metric spaces} we prove some properties on quasiflats in metric spaces with cone inequalities.
	
	In Section~\ref{sec_definition}, we introduce various definitions of Morse quasiflats. Definitions using asymptotic cones are in Section~\ref{subsec_cone_conditions}, 
	and definitions without asymptotic cones are in Section~\ref{subsec_asymptotic condition}. We also prove quasi-isometry invariance of Morse quasiflats in Section~\ref{subsec_QI_invariance}.

	Section~\ref{sec:cone and space}, Section~\ref{sec_cycles_close_to_Morse_quasiflats} and Section~\ref{sec_neck_decomposition} are devoted to the proof of equivalence of various definitions. In Section~\ref{sec:cone and space} we discuss relations between definitions using asymptotic cones, and definitions without using asymptotic cones. In Section~\ref{sec_cycles_close_to_Morse_quasiflats} we introduce the cycle contraction property and study its relation with some conditions in Section~\ref{subsec_asymptotic condition}. In  Section~\ref{sec_neck_decomposition} we introduce the coarse piece decomposition and conclude the equivalence of all properties under suitable conditions.
	
	In Section~\ref{sec_Morse_Lemma} we prove stability properties of Morse quasiflats, including a version of the Morse lemma for Morse quasi-disks. 
	
	In Section~\ref{sec_example} we prove the half flat criterion (Proposition~\ref{lem_coefficient intro})
	for Morse quasiflats in $\cat(0)$ spaces, and discuss several examples/non-examples of Morse quasiflats.
	
	When the ambient spaces has a Lipschitz bicombing (e.g. $\cat(0)$ spaces or injective metric spaces), then some part of the proof can be simplified and short cuts are discussed in the appendix of this paper. 
	
	\subsection{Acknowledgment}
	We thank J. Behrstock for bringing \cite{mcgdivergence} to our attention and pointing out Theorem~\ref{cor_highest_abelian_mcg}. We also thank J. Russell for some helpful discussions. We warmly thank the referee for a careful reading and many helpful comments to improve the exposition.
	
	\section{Discussion of proofs}
	\label{sec_proofs}
	We give rough sketches of some of the key arguments in the paper.

	\noindent
	\subsection*{Proof of Theorem~\ref{thm:equivalence intro}}
	
	We only discuss $(1)\Rightarrow (3)$, which is the most interesting part. For simplicity, we assume $Q$ is a bilipschitz flat.

	Take a base point $p\in Q$. Let $\si$ be an $(n-1)$-cycle such that $\si=\D \tau$ with $\M(\tau)\le C R^n$, $\spt(\tau)\subset B_p(R)\setminus N_{\rho R}(Q)$, and $M(\si)\le C R^{n-1}$ 
	(here $R\gg 1$, $\rho\ll 1$). Let $\si'$ be a cycle in $Q$ which is the projection of $\si$ obtained from a relative Plateau problem as explained before. Let $\tau'$ be the canonical filling of $\si'$ inside the bilipschitz flat $Q$. Given $\delta>0$. We need to show there exists $\ul{R}$ such that $\M(\tau')\le \delta\cdot R^n$ whenever $R\ge \ul{R}$.
	
	Let $\al$ be the solution of the relative Plateau problem with $\D \al=\si-\si'$. Let $\al_h$ be the slice of $\al$ at distance $h$ from $Q$, and let $\al_{\le h}$ be the piece of $h$ which is at distance $\le h$ from $Q$. For the moment let us make the following simplifying assumption which need not to be true in general:
	\begin{center}
		$(*)$ $\al_h$ is contained in the $h$-neighborhood of $\si'$.
	\end{center}

	Then $\si'=\D\tau'=\D(\al+\tau)$. The $(\mu,b)$-rigid condition implies that $\spt(\tau')\subset N_{\eps R}(\spt (\al+\tau))$, where $\eps\ll\rho$ as long as $R$ is sufficiently large. 
	As $\spt(\tau)$ is further away from $Q$, we know $\spt(\tau')\subset N_{\eps R}(\spt(\al)))$. This together with condition $(*)$ implies  $\spt(\tau')\subset N_{2\eps R}(\spt(\si'))$. Hence $\si'$ can be filled in a small neighborhood of its support, which intuitively suggests that $\M(\si')$ is small. However, we can not argue in this way as there are examples of cycles with small filling 
	radius but large filling volume.
	
	The idea is that if $\si'$ has some extra geometric control, more precisely controlled $k$-dimensional Minkowski content with $k<n$, then we can deduce small filling volume from small filling radius 
	(see Lemma~\ref{lem_small_filling} for a precise formulation). As $\al$ is the solution of a relative Plateau problem, it satisfies a relative version of the monotonicity formula (cf. Lemma~\ref{lem_alpha_control}), which can be used to estimate the exponent of the Minkowski content (cf. Lemma~\ref{lem_si'} (3)).
	
	In the general case there are several further complications, including:
	\bit
	\item Condition $(*)$ may fail. In fact, as $h\to 0$, the control on the diameter of $\al_h$ gets worse. 
	\item The solution to the relative Plateau problem may not be realized by an integral current $\al$. 
	\item The quasiflat may not be represented by a bilipschitz (or even Lipschitz) quasi-isometric embedding.
	\eit
	These issues are handled in Section~\ref{sec_cycles_close_to_Morse_quasiflats}.

	\subsection*{Proof of Theorem~\ref{thm:conditions in asymptotic cones}}
	We only discuss the step where having full support in the asymptotic cone with respect to \emph{compactly supported} integral currents implies super-Euclidean divergence, see Lemma~\ref{lem_fullsup_implies_suplindiv}. The proof is another instance where solving a free boundary minimizing problem turns out to be helpful.
	
	The usual argument is to suppose super-Euclidean divergence fails. Then one obtains a sequence of integral currents $\si_k$, and argues that their rescaled limit in the asymptotic cone contradicts the full support with respect 
	to the reduced homology induced by compact supported integral currents.
	The Wenger compactness theorem guarantees a limit, however, the limit may not have compact support.
	In order to resolve this issue, we introduce a regularization procedure which we believe is of independent interest.
	(Another way to deal with this point might be to use a variation  on the ``thick-thin decomposition'' from \cite{gromov1983filling,wenger2011compactness}.)

	Our idea is to use a free boundary minimizing problem to ``regularize'' an integral current $\si$. More precisely, given a constant $\kappa>0$, we consider a functional 
	$\Theta:\I_{n}(X)\to\R$ by $\Theta(\tau)=\kappa\M(\tau)+\M(\si-\D\tau)$. We can think of $\tau$ as a ``cylinder'' with one end of the cylinder being $\si$, and the other end being $\si':=\si-\D\tau$. 
	The functional involves the mass of the cylinder and the mass of the ``free end'' of the cylinder. Note that if $X=\mathbb R^n$ and $\si$ has ``fingers'', then cutting off the fingers will decrease the functional $\Theta$. This also holds true for more general $X$ with appropriate isoperimetric assumption.
	We take a minimizer $\tau'$ of $\Theta$ and let $\si'=\si-\D\tau'$. A simple computation yields that $\si'$ has a lower  density bound. Moreover, by making $\kappa$ large, we can assume the flat distance between $\si'$ and $\si$ is small, 
	and $\spt(\si')$ is contained in a small neighborhood of $\spt(\si)$, see Proposition~\ref{prop_approx}. Back to the proof, we can apply this regularization procedure to each $\tau_k$ to obtain a new sequence with 
	uniformly bounded supports, moreover, the supports of elements in the new sequence are not too far away from the original elements, which is enough to conclude the proof.
	
	\section{Preliminaries}
	\label{sec_prelim}
	
	\subsection{Metric notions} \label{subsect:metric}
	\mbox{}
	Let $X = (X,d)$ be a metric space. We write 
	\[
	\B{p}{r} := \{x \in X : d(p,x) \le r\}, \quad
	\Sph{p}{r} := \{x \in X : d(p,x) = r\}
	\] 
	for the closed ball and sphere with radius $r \ge 0$ and center $p \in X$. We write $A_p(r_1,r_2):=\{x\in X: r_1\le d(p,x)\le r_2\}$ for closed annulus with inner radius $r_1$ and outer radius $r_2$.

	A map $f \colon X \to Y$ into another metric space $Y = (Y,d)$ is
	{\em $L$-Lipschitz}, for a constant $L \ge 0$, 
	if $d(f(x),f(x')) \le L\,d(x,x')$ for all $x,x' \in X$.
	The smallest such $L$ is the {\em Lipschitz constant\/} $\Lip(f)$ of~$f$. 
	The map $f \colon X \to Y$ is an {\em $L$-bilipschitz embedding}
	if $L^{-1}d(x,x') \le d(f(x),f(x')) \le L\,d(x,x')$ for all $x,x' \in X$.
	For an $L$-Lipschitz function $f \colon E \to \R$ defined on a set 
	$E \sub X$,
	\[
	\bar f(x) := \sup\{f(a) - L\,d(a,x): a \in E\} \quad \text{($x \in X$)}
	\]
	defines an $L$-Lipschitz extension $\bar f \colon X \to \R$ of $f$.
	Every $L$-Lipschitz map $f \colon E \to \R^n$, $E \sub X$, admits
	a $\sqrt{n}L$-Lipschitz extension $\bar f \colon X \to \R^n$.
	
	A map $f \colon X \to Y$ between two metric spaces is called an
	{\em $(L,A)$-quasi-isometric embedding}, for constants
	$L \ge 1$ and $A \ge 0$, if
	\[
	L^{-1} d(x,x') - A \le d(f(x),f(x')) \le L\,d(x,x') + A
	\]
	for all $x,x' \in X$.
	A {\em quasi-isometry} $f \colon X \to Y$ has the additional property
	that $Y$ is within finite distance of the image of $f$. An \emph{$(L,A)$-quasi-disk} $D$ in a metric space $X$ is the image of an $(L,A)$ quasi-isometric embedding $q$ from a closed metric ball $B$ in $\R^n$ to $X$. $B$ is called the \emph{domain} of $D$. The \emph{boundary} of $D$, denoted $\D D$, is defined to be $q(\D B)$. An $n$-dimensional {\em quasiflat\/} in $X$ is the image of a quasi-isometric
	embedding of $\R^n$.
	
	A curve $\rho \colon I \to X$ defined on some interval $I \sub \R$ is a
	{\em geodesic} if there is a constant $s \ge 0$, the {\em speed}
	of $\rho$, such that $d(\rho(t),\rho(t')) = s |t - t'|$ for all $t,t' \in I$.
	A geodesic defined on $I = \R_+ := [0,\infty)$ is called a {\em ray}.

	\begin{definition}
		\label{def:lipschitz bicombing}
		A metric space $X$ has an \emph{$L$-Lipschitz combing} if for each pair of points $x,y\in X$, there is an $L$-bilipschitz path $\si_{xy}:[0,a_{xy}]\to X$ from $x$ to $y$ such that $$d(\si_{xy}(t),\si_{xy'}(t))\le L\cdot d(y,y')$$ for any $x,y,y'\in X$ and any $t\ge 0$. Here $\si_{xy}$ is assumed to be extended
		to domain $\R$ by constant map. $X$ has an \emph{$L$-Lipschitz bicombing} if we instead require $$d(\si_{xy}(t),\si_{x'y'}(t))\le L\cdot \max\{d(x,x'),d(y,y')\}$$ for any $x,y,x',y'\in X$ and any $t\ge 0$. 
	\end{definition}

	Having $L$-Lipschitz (bi)combing is closed under taking ultralimits. If $X$ is a metric space with a bounded $(L,C)$–quasi-geodesic combing in the sense of \cite[Section 1]{alonso1995semihyperbolic}, then any asymptotic cone of $X$ has an $L$-Lipschitz combing. Thus the asymptotic cone of any combable group, or semi-hyperbolic group, has an $L$-Lipschitz combing.

	\subsection{Local currents in proper metric spaces} 
	\label{subsect:currents}
	\mbox{}
	Readers not familiar with the theory of current may think of currents as Lipschitz chains for the first reading (see Remark~\ref{rmk_Lipschitz_chain}). However, many useful features of currents, like slicing and various compactness results, are less clear for Lipschitz chains.

	Currents of finite mass in complete metric spaces were introduced by Ambrosio and Kirchheim in~\cite{ambrosio2000currents}. There is a localized variant of this theory for locally compact metric spaces, as described in~\cite{Lan3}. Here we only provide some background on the theory from \cite{Lan3}, and leave it to the reader to identify the corresponding statements in \cite{ambrosio2000currents} (such comparison is already carried out in detail in \cite{Lan3}). To avoid certain technicalities, we will assume that the underlying metric space $X$ is proper throughout Section~\ref{subsect:currents}, hence complete and separable.

	For every integer $n \ge 0$, let $\cD^n(X)$ denote the set of all 
	$(n+1)$-tuples $(\pi_0,\ldots,\pi_n)$ of real valued functions on $X$ such 
	that $\pi_0$ is Lipschitz with compact support $\spt(\pi_0)$ and 
	$\pi_1,\dots,\pi_n$ are locally Lipschitz. 
	(In the case that $X = \R^N$ and the entries of $(\pi_0,\ldots,\pi_n)$ are 
	smooth, this tuple should be thought of as representing the compactly supported 
	differential $n$-form $\pi_0\,d\pi_1 \wedge \ldots \wedge d\pi_n$.)
	An {\em $n$-dimensional current\/} $S$ in $X$ is a function 
	$S \colon \cD^n(X) \to \R$ satisfying the following three conditions:
	\ben
	\item
	$S$ is $(n+1)$-linear;
	\item 
	$S(\pi_{0,k},\ldots,\pi_{n,k}) \to S(\pi_0,\ldots,\pi_n)$
	whenever $\pi_{i,k} \to \pi_i$ pointwise on $X$ with 
	$\sup_k\Lip(\pi_{i,k}|_K) < \infty$ for every compact set $K \sub X$ 
	($i = 0,\dots,n$) and with $\bigcup_k\spt(\pi_{0,k}) \sub K$ for some such set; 
	\item
	$S(\pi_0,\ldots,\pi_n) = 0$ whenever one of the functions
	$\pi_1,\ldots,\pi_n$ is constant on a neighborhood of $\spt(\pi_0)$.
	\een
	We write $\cD_n(X)$ for the vector space of all $n$-dimensional currents 
	in $X$. The defining conditions already imply that every $S \in \cD_n(X)$ is 
	alternating in the last $n$ arguments and satisfies a product derivation 
	rule in each of these. The definition is further motivated by the fact that
	every function $w \in L^1_\loc(\R^n)$ induces a current 
	$\bb{w} \in \cD_n(\R^n)$ defined by
	\[
	\bb{w}(\pi_0,\dots,\pi_n) 
	:= \int w \pi_0\det\bigl[\d_j\pi_i\bigr]_{i,j = 1}^n \,dx
	\]
	for all $(\pi_0,\dots,\pi_n) \in \cD^n(\R^n)$, where the partial derivatives 
	$\d_j\pi_i$ exist almost every\-where according to Rademacher's theorem. 
	Note that this just corresponds to integration of the differential form
	$\pi_0\,d\pi_1 \wedge \ldots \wedge d\pi_n$ over $\R^n$, weighted by $w$.
	For the characteristic function $\chi_W$ of a Borel set $W \sub \R^n$,
	we put $\bb{W} := \bb{\chi_W}$.
	(See Section~2 in~\cite{Lan3} for details.) 
	
	\subsubsection{Support, push-forward,  and boundary}
	\mbox{}
	For every $S \in \cD_n(X)$ there exists a smallest closed subset of $X$,
	the {\em support\/} $\spt(S)$ of $S$, such that the value 
	$S(\pi_0,\ldots,\pi_n)$ depends only on the restrictions of 
	$\pi_0,\dots,\pi_n$ to this set.
	
	Let $A\subset X$ be a closed subset and let $S_A\in \cD_n(A)$. Then $S_A$ naturally defines $S\in \cD_n(X)$ by 
	$S(\pi_0,\ldots,\pi_n)=S_A(\pi_0|_{A},\ldots,\pi_n|_A)$. We have $\spt(S_A)=\spt(S)$. 
	Conversely, let $S\in \cD_n(X)$ and let $A\subset X$ be locally compact. If $\spt(S)\subset A$, then $S$ uniquely determines an 
	element $S_A\in \cD_n(A)$ (\cite[Proposition 3.3]{Lan3}).

	For a proper Lipschitz map $f \colon X \to Y$ into another proper
	metric space $Y$, the {\em push-forward\/} $f_\#S \in \cD_n(Y)$ is defined
	simply by
	\[
	(f_\#S)(\pi_0,\ldots,\pi_n) := S(\pi_0 \circ f,\ldots,\pi_n \circ f)
	\]
	for all $(\pi_0,\ldots,\pi_n) \in \cD^n(Y)$. This definition can be
	extended to proper Lipschitz maps $f \colon \spt(S) \to Y$ by the previous paragraph. 
	In either case, $\spt(f_\#S) \sub f(\spt(S))$.
	For $n \ge 1$, the {\em boundary} $\D S \in \cD_{n-1}(X)$ of 
	$S \in \cD_n(X)$ is defined by
	\[
	(\D S)(\pi_0,\dots,\pi_{n-1}) := S(\tau,\pi_0,\dots,\pi_{n-1})
	\]
	for all $(\pi_0,\ldots,\pi_{n-1}) \in \cD^{n-1}(X)$ and for any compactly 
	supported Lipschitz function $\tau$ that is identically $1$ on some 
	neighborhood of $\spt(\pi_0)$. If $\til\tau$ is another such function, 
	then $\pi_0$ vanishes on a neighborhood of $\spt(\tau - \til\tau)$ and 
	$\D S$ is thus well-defined by~(1) and~(3).
	Similarly one can check that $\D \circ \D = 0$. The inclusion
	$\spt(\D S) \sub \spt(S)$ holds, and $f_\#(\D S) = \D(f_\# S)$ for
	$f \colon \spt(S) \to Y$ as above.
	(See Section~3 in \cite{Lan3}.) If $\D S=0$, then we will call $S$ an \emph{$n$-cycle}.
	
	\subsubsection{Mass} \label{subsect:mass}
	\mbox{}
	Let $S \in \cD_n(X)$. A tuple $(\pi_0,\ldots,\pi_n) \in \cD^n(X)$
	will be called {\em normalized\/} if the restrictions of
	$\pi_1,\ldots,\pi_n$ to the compact set $\spt(\pi_0)$ are $1$-Lipschitz.
	For an open set $U \sub X$, 
	the {\em mass} $\|S\|(U) \in [0,\infty]$ of $S$ in $U$ is then defined as
	the supremum of $\sum_j S(\pi_{0,j},\ldots,\pi_{n,j})$
	over all finite families of normalized tuples
	$(\pi_{0,j},\ldots,\pi_{n,j}) \in \cD^n(X)$ such that
	$\bigcup_j \spt(\pi_{0,j}) \sub U$ and $\sum_j |\pi_{0,j}| \le 1$.
	Note that $\|S\|(U) > 0$ if and only if $U \cap \spt(S) \ne \es$.
	This induces a regular Borel measure $\|S\|$ on $X$, whose
	{\em total mass} $\|S\|(X)$ is denoted by $\M(S)$. For Borel sets
	$W,A \sub \R^n$, $\|\bb{W}\|(A)$ equals the Lebesgue measure of $W \cap A$.
	If $T \in \cD_n(X)$ is another $n$-current in $X$, then clearly
	\[
	\|S + T\| \le \|S\| + \|T\|.
	\]
	We will now assume that the measure $\|S\|$ is locally finite (and hence
	finite on bounded sets, as $X$ is proper). Then it can be shown that
	\[
	|S(\pi_0,\ldots,\pi_n)| \le \int_X |\pi_0|\,d\|S\|
	\]
	for every normalized tuple $(\pi_0,\ldots,\pi_n) \in \cD^n(X)$.
	This inequality allows us to extend the functional $S$ such that the first entry $\pi_0$ is allowed to be a compact-supported bounded Borel function \cite[Theorem 4.4]{Lan3}. We define the {\em restriction} $S \on A \in \cD_n(X)$
	of $S$ to a Borel set $A \sub X$ by
	\[
	(S \on A)(\pi_0,\dots,\pi_n) := \lim_{k \to \infty} S(\tau_k,\pi_1,\ldots,\pi_n)
	\]
	for any sequence of compactly supported Lipschitz functions $\tau_k$ converging
	in $L^1(\|S\|)$ to $\chi_A \pi_0$. The measure $\|S \on A\|$ equals
	the restriction $\|S\| \on A$ of $\|S\|$ (\cite[Lemma 4.7]{Lan3}).
	
	A \emph{piece decomposition} of $S$ is a sum $S=\sum_i S_i$ such that for any Borel set $A$ holds $\|S\|(A)=\sum_i\|S_i\|(A)$. Each $S_i$ is a \emph{piece} of $S$. Suppose $X=A\sqcup B$ with $A$ and $B$ Borel. Then the previous paragraph implies that $S=S\on A+S\on B$ is a piece decomposition.
	
	Recall a more general restriction operation as follows (\cite[Definition 2.3]{Lan3}). For $S\in \cD_m(X)$ and $(u,v) \in \Lip_{\loc}(X)\times [\Lip_{\loc}(X)]^k$ where $m\ge k\ge 0$, define the current $S\on(u,v)\in\cD_{m-k}(X)$ by $(S\on(u,v))(f,g):= S(uf,v,g) = S(uf,v_1,\ldots,v_k,g_1,\ldots,g_{m-k})$ for $(f,g)\in\cD^{m-k}(X)$. If $k=0$, then the definition simplifies to $$(S\on u)(f,g):= S(uf,g) = S(uf,g_1,\ldots,g_{m}).$$
	
	If $f \colon \spt(S) \to Y$ is a proper $L$-Lipschitz map into a proper 
	metric space $Y$ and $B \sub Y$ is a Borel set, then
	$(f_\#S) \on B = f_\#(S \on f^{-1}(B))$ and
	\[
	\|f_\#S\|(B) \le L^n\,\|S\|(f^{-1}(B)).
	\]
	(See Section~4 in~\cite{Lan3}.)
	
	We say $\|S\|$ is \emph{concentrated} on a Borel subset $A\subset X$ if $\|S\|(X\setminus A)=0$. It is possible that $A$ is much smaller than $\spt(S)$. 
	\begin{remark}
		\label{rmk:concentration}
		We summarize some useful properties which follow directly from \cite[Lemma 4.7]{Lan3}:
		\begin{enumerate}
			\item if $S_i$ converges to $S$ weakly and each $S_i$ is concentrated on a Borel set $A$, then $S$ is concentrated on $A$;
			\item if $S$ is concentrated on $A$; then $S\on(1,v)\in \cD_{m-k}(X)$ for $v\in[\Lip(X)]^k$ is also concentrated on $A$.
		\end{enumerate}
	\end{remark}

	\subsubsection{Slicing} \label{subsect:slicing}
	\mbox{}
	Let $S \in \cD_n(X)$ be such that both $\|S\|$ and $\|\D S\|$ are locally finite
	(that is, $S$ is {\em locally normal}, see Section~5 in~\cite{Lan3}).
	Let $\rho \colon X \to \R$ be a Lipschitz function.
	The corresponding {\em slice} of $S$ is the $(n-1)$-dimensional current
	\begin{align*}
	\slc{S,\rho,s+} &:= \D(S \on  \{\rho \le s\}) - (\D S) \on  \{\rho \le s\}\\
	&= (\D S)\on\{\rho>s\} - \D(S \on  \{\rho > s\})
	\end{align*}
	with support in $\{\rho = s\} \cap \spt(S)$. 
	Similarly, we define
	\begin{align*}
	\slc{S,\rho,s-} &:= (\D S)\on\{\rho\ge s\} - \D(S \on  \{\rho \ge s\})\\
	&= \D(S \on  \{\rho < s\}) - (\D S) \on  \{\rho < s\}\ .
	\end{align*}
	When $\spt(S)$ is separable (this is always true if $X$ is proper), $(\|S\|+\|\D S\|)(\pi^{-1}(s))=0$ for a.e. (almost every) $s$. Thus $\slc{S,\rho,s+}=\slc{S,\rho,s-}$ holds for a.e. $s$, in this case, we define $\slc{S,\rho,s}$ to be either one of them.
	
	For almost all $s$, $S\on\{\rho\le s\}$ is the maximal piece of $S$ supported in $\{\rho\le s\}$ and $S$ has a piece decomposition $S=S\on\{\rho\le s\}+S\on\{\rho> s\}$ (we can also write the piece decomposition as $S=S\on\{\rho< s\}+S\on\{\rho> s\}$ as $S\on\{\rho\le s\}=S\on\{\rho<s\}$ for a.e. $s$).
	
	For $s\in\R$ and $\delta>0$, let now $\ga_{s,\de}:\R\to\R$ be the piecewise affine $\frac{1}{\de}$-Lipschitz function with $\ga_{s,\de}|_{(-\infty,s]}= 0$ and$\ga_{s,\de}|_{[s+\de,\infty)}= 1$. Then, for $y = (y_1,\ldots,y_k) \in \R^k$ and $\de>0$, define $\ga_{y,\de}:\R^k\to\R^k$ such that $\ga_{y,\de}(z)= (\ga_{y_1,\de}(z_1),\ldots,\ga_{y_k,\de}(z_k))$ for all $z = (z_1,\ldots,z_k)\in\R^k$.
	
	\begin{lem}
		\label{lem_slice}
		For each $s$, we have
		\begin{enumerate}
			\item $\slc{S,\rho,s+}$ is equal to the weak limit $\lim_{\de\to 0+}S\on(1,\ga_{s,\de}\circ\rho)$;
			\item if $S$ is concentrated on a Borel subset $A\subset X$, then $\slc{S,\rho,s+}$ is concentrated on $A\cap\rho^{-1}(s)$.
		\end{enumerate}
	\end{lem}
	
	\begin{proof}
		The first assertion is \cite[Equation (6.4)]{Lan3}. (2) follows from (1) and Remark~\ref{rmk:concentration}.
	\end{proof}
	
	For a Lipschitz function $\pi:X\to \R^k$, we define the slice of $S$ at $y\in \R^k$ with respect to $\pi$ as the weak limit $\slc{S,\pi,y}:=\lim_{\de\to 0+} S\on (1,\ga_{y,\de}\circ \pi)$ whenever it exists and defines an element of $\cD^{(m-k)}(X)$. By \cite[Theorem 6.4]{Lan3}, the weak limit $\lim_{\de\to 0+} S\on (1,\ga_{y,\de}\circ \pi)$ exists for a.e. $y\in \R^n$.

	For a Borel subset $B\subset\R^k$, the coarea inequality
	\[
	\int_B \M(\slc{S,\pi,s}) \,ds \le \Lip(\pi)\,\|S\|(\pi^{-1}(B))
	\]
	holds by \cite[Thoerem 6.4 (3)]{Lan3}.

	\subsubsection{Integral currents}
	\mbox{}
	A current $S \in \cD_n(X)$ is called {\em locally integer rectifiable\/} 
	if all the following properties hold:
	\begin{enumerate}
		\item $\|S\|$ is locally finite and is concentrated on the union of
		countably many Lipschitz images of compact subsets of $\R^n$;
		\item for every
		Borel set $A \sub X$ with compact closure and every Lipschitz map
		$\phi \colon X \to \R^n$, the current $\phi_\#(S \on A) \in \cD_n(\R^n)$
		is of the form $\bb{w}$ for some {\em integer valued\/}
		$w = w_{A,\phi} \in L^1(\R^n)$.
	\end{enumerate}
	Then $\|S\|$ turns out to be absolutely continuous with 
	respect to $n$-dimensional Hausdorff measure. 
	Furthermore, push-forwards and restrictions to Borel sets
	of locally integer rectifiable currents are again locally integer rectifiable.
	
	A current $S \in \cD_n(X)$ is called a {\em locally integral current\/} 
	if $S$ is locally integer rectifiable and, for $n \ge 1$, $\D S$ satisfies 
	the same condition.
	(Remarkably, this is the case already when $\|\D S\|$ is locally finite,
	provided $S$ is locally integer rectifiable; see Theorem~8.7 in~\cite{Lan3}.)
	This yields a chain complex of abelian groups $\bI_{n,\loc}(X)$.
	We write $\bI_{n,\cs}(X)$ (resp. $\bI_n(X)$) for the respective subgroups of
	{\em integral currents\/} with compact support (resp. with finite mass).
	
	By \cite[Thoerem 8.5]{Lan3}, if $\pi:X\to \R^k$ is Lipschitz, then for a.e. $y\in \R^k$, $\slc{S,\pi,y} \in \bI_{n-k,\loc}(X)$. In particular, if $k=1$, $S \on  \{\pi \le y\} \in \bI_{n,\loc}(X)$ for a.e. $y\in \R$.

	\begin{remark}
		\label{rmk_Lipschitz_chain}
		If $\Del \sub \R^n$ is an $n$-simplex and $f \colon \Del \to X$
		is a Lipschitz map, then $f_\#\bb{\Del} \in \bI_{n,\cs}(X)$.
		Thus every singular Lipschitz chain in $X$ with integer coefficients
		defines an element of $\bI_{n,\cs}(X)$.
	\end{remark}
	There is a canonical chain isomorphism from $\bI_{*,\cs}(\R^N)$
	to the chain complex of ``classical'' integral currents in $\R^N$
	originating from~\cite{FedF}.
	
	If $T\in \bI_{N,\loc}(\R^N)$, then $T=\bb{u}$ for some function $u$ of locally bounded variation, moreover $u$ is integer-valued almost everywhere \cite[Theorem 7.2]{Lan3}. The element $\bb{\R^N}\in \bI_{N,\loc}(\R^N)$ is the \emph{fundamental class} of $\R^N$. 
	
	For $n \ge 1$, we let $\bZ_{n,\loc}(X) \sub \bI_{n,\loc}(X)$ and 
	$\bZ_{n,\cs}(X) \sub \bI_{n,\cs}(X)$ denote the subgroups of currents
	with boundary zero. An element of $\bI_{0,\cs}(X)$ is an integral
	linear combination of currents of the form $\bb{x}$, where
	$\bb{x}(\pi_0) = \pi_0(x)$ for all $\pi_0 \in \cD^0(X)$. We let
	$\bZ_{0,\cs}(X) \sub \bI_{0,\cs}(X)$ denote the subgroup of linear combinations
	whose coefficients sum up to zero. The boundary of a current in $\bI_{1,\cs}(X)$
	belongs to $\bZ_{0,\cs}(X)$.
	Given $Z \in \bZ_{n,\cs}(X)$, for $n \ge 0$, we will call
	$V \in \bI_{n+1,\cs}(X)$ a {\em filling} of $Z$ if $\D V = Z$.
	
	For each $S\in \bZ_{N-1,c}(\R^N)$, there exists a unique $T\in\bI_{N,c}(\R^N)$ such that $\D T=S$. In this case, $T$ is called the \emph{canonical filling of $S$}. By  the  discussion above, $T=\bb{u}$ for a compactly supported integer valued function of bounded variation. Similarly, by an approximation argument, one can define a canonical filling $T\in \bI_N(\R^N)$ for each $S\in\bZ_{N-1}(\R^{N})$. If $T=\bb{u}$, then $\M(T)$ is the $L^1$-norm of $u$ (\cite[Equation 4.5]{Lan3}) and $\|\D T\|=|D u|$ (\cite[Theorem 7.2]{Lan3}).

	\subsection{Ambrosio-Kirchheim currents in general metric spaces}
	We will also need metric currents in a complete metric space $X$ which is not necessarily locally compact \cite{ambrosio2000currents}.  These currents are defined using slightly different axioms, however, all the features discussed in the previous section (e.g. support, boundary, mass, slicing etc) are available, see \cite{ambrosio2000currents} for more details. We refer to \cite[Section 4]{Lan3} for comparison between currents in the sense of Ambrosio and Kirchheim and local currents in the previous section.
	
	Let $X$ be a complete metric space. We use $\I_n(X)$ to denote the abelian group of $n$-dimensional integral currents in the sense of Ambrosio and Kirchheim. In the special case when $X$ is proper, each element in $\I_{n,\loc}(X)$ defined in Section~\ref{subsect:currents} with finite mass gives a unique element in $\I_n(X)$, and each element of $\I_n(X)$ in the sense of Ambrosio and Kirchheim gives a unique element in $\I_{n,\loc}$ with finite mass (\cite[Section 4]{Lan3}). So when $X$ is proper, we can treat $\I_n(X)$ as the abelian group of local integral currents with finite mass. We use $\I_{n,c}(X)$ to denote integral currents with compact support. When $X$ is proper, these currents can be identified with local integral currents with compact support.
	
	\subsection{Homotopies}
	\label{subsec_homotopy}
	Let $X$ be a complete metric space.
	Let $\bb{0,1} \in \bI_{1,\cs}([0,1])$ denote the current defined by
	\[
	\bb{0,1}(\pi_0,\pi_1) := \int_0^1 \pi_0(t) \pi'_1(t) \,dt.
	\]
	Note that $\D\bb{0,1} = \bb{1} - \bb{0}$.
	We endow $[0,1] \times X$ with the usual $l_2$~product metric.
	There exists a canonical product construction
	\[
	T \in \bI_n(X) \leadsto 
	\bb{0,1} \times T \in \bI_{n+1}([0,1] \times X)
	\]
	for all $n \ge 0$ (see \cite[Definition 2.8]{wenger2005isoperimetric}).
	Suppose now that $Y$ is another complete metric space,
	$h \colon [0,1] \times X \to Y$ is a Lipschitz homotopy from 
	$f = h(0,\cdot)$ to $g = h(1,\cdot)$, and $T \in \bI_n(X)$. 
	Then $h_\#(\bb{0,1} \times T)$ is an element of $\bI_{n+1}(Y)$ 
	with boundary
	\[
	\D\,h_\#(\bb{0,1} \times T) = g_\# T - f_\# T - h_\#(\bb{0,1} \times \D T)
	\]
	(for $n = 0$ the last term is zero.) See \cite[Theorem 2.9]{wenger2005isoperimetric}.
	If $h(t,\cdot)$ is $L$-Lipschitz for every $t$, and $h(\cdot,x)$ is a
	geodesic of length at most $D$ for every $x \in \spt(T)$, then
	\[
	\M(h_\#(\bb{0,1} \times T)) \le (n+1) L^n D\,\M(T).  
	\]
	(see \cite[Proposition 2.10]{wenger2005isoperimetric}). A similar formula holds if instead $h(\cdot,x)$ is a bi-Lipschitz path of length $\le D$.
	
	An important special case of this is
	when $S \in \bZ_n(X)$ and $h(\cdot,x) = \si_{px}$ is a geodesic from
	some fixed point $p \in X$ to $x$ for every $x \in \spt(S)$.   
	Then $h_\#(\bb{0,1} \times S) \in \bI_{n+1}(X)$ is the 
	{\em cone from $p$ over $S$} determined by this family of geodesics,
	whose boundary is $S$. 
	If the family of geodesics satisfies the convexity condition
	\[
	d(h(t,x),h(t,x')) = d(\si_{px}(t),\si_{px'}(t)) \le t\, d(x,x')
	\]
	for all $x,x' \in \spt(S)$ and $t \in [0,1]$, and if 
	$\spt(S) \sub \B{p}{r}$, then
	\[
	\M(h_\#(\bb{0,1} \times S)) \le r\cdot \M(S).  
	\]
	
	\subsection{Coning inequalities and isoperimetric inequalities}
	
	\begin{definition}
		\label{def_coning_ineq}
		A complete metric space $X$ satisfies {\em coning inequalities} up to dimension $n$ for currents, abbreviated condition~(CI$_n$), if there exists a constant $c_0>0$, such that for all $0\leq k\leq n$ and every cycle $S\in\I_{k}(X)$ with bounded support there exists a filling
		$T\in\I_{k+1}(X)$ with 
		\[\M(T)\leq c_0\cdot\diam(\spt (S))\cdot\M(S).\]
		If in addition, there exists a constant $c_1>0$ such that $T$ can always be chosen to fulfill
		\[\diam(\spt (T))\leq c_1\cdot \diam(\spt (S)),\]
		then $X$ is said to satisfy {\em strong coning inequalities} up to dimension $n$, abbreviated condition~(SCI$_n$).
	\end{definition}

	\begin{remark}
		\label{rmk_combing}
		By Section~\ref{subsec_homotopy}, condition (CI$_n$) and (SCI$_n$) are satisfied for complete metric space with an $L$-Lipschitz combing for any $n$.
	\end{remark}
	
	\begin{definition}
		\label{def_coning_ineq1}
		Let $\mathcal{C}$ be a class of chains, e.g. $\mathcal{C}$ could be the class of singular chains, or the class of Lipschitz chains, or the class of compact supported integral currents. We can reformulate Definition~\ref{def_coning_ineq} by requiring both $S$ and $T$ being elements in $\mathcal{C}$, leading to the notion of having \emph{(strong) coning inequalities up to dimension $n$ for $\mathcal{C}$}. 
	\end{definition}

	If $X$ is proper, then (CI$_n$) is equivalent to having coning inequalities up to dimension $n$ for compact supported integral currents. If $X$ is bilipschitz homeomorphic
	to a finite-dimensional simplicial complex with standard metrics on the simplices, then (CI$_n$) (by a variant of the Federer–Fleming deformation
	theorem \cite{FedF}) is equivalent to having coning inequalities up to dimension $n$ for simplicial chains or singular
	Lipschitz chains (with integer coefficients). See \cite[Section 2]{abrams2013pushing} for more explanation. Similar statements hold for (SCI$_n$).
	
	\begin{remark}
		Any $n$-connected simplicial complex with a properly discontinuous
		and cocompact simplicial action of a combable group satisfies (CI$_n$) and (SCI$_n$); see
		\cite[Section 10.2]{MR1161694}. Every combable group, in particular every automatic group, admits such an action.
	\end{remark} 
	
	\begin{theorem}[isoperimetric inequality] \label{thm:isop-ineq}
		Let $n \ge 2$, and let $X$ be a complete metric space satisfying
		condition~{\rm (CI$_{n-1}$)}. Then every cycle 
		$S \in \bZ_{n-1}(X)$ possesses a filling $T \in \bI_n(X)$  such that
		\begin{enumerate}
			\item 	$\M(T) \le b_1\cdot\M(S)^{n/(n-1)}$;
			\item $\spt(T)\subset N_c(\spt(\D T))$ where $c=b_2\M(S)^{1/(n-1)}$,
		\end{enumerate}
		for some constants $b_1,b_2 > 0$ depending only on the constants 
		of condition~{\rm (CI$_{n-1}$)}. Moreover, if $S$ has compact support, then we can also require $T$ to have compact support.
	\end{theorem}
	
	The first item is proved in \cite{wenger2005isoperimetric}, see the comment after \cite[Theorem 1.2]{wenger2005isoperimetric} regarding compact supports. The second item follows from \cite[Proposition 4.3 and Corollary 4.4]{wenger2011compactness}.

	We also consider the following condition which is weaker than Definition~\ref{def_coning_ineq} (by Theorem~\ref{thm:isop-ineq}.)
	\begin{definition}
		\label{def_eii}
		A complete metric space $X$ satisfies {\em Euclidean isoperimetric inequality} up to dimension $n$, abbreviated condition~(EII$_n$), 
		if there exists a constant $c>0$, such that for all $0\leq k\leq n$ and every cycle $S\in\bZ_{k,c}(X)$ there exists a filling
		$T\in\I_{k+1,c}(X)$ with 
		\[\M(T)\leq c\cdot\M(S)^{k+1/k}.\]
	\end{definition}
	
	\subsection{Quasi-minimizer and minimizer}
	
	\begin{definition}[quasi-minimizer] \label{def:qmin}
		Suppose that $X$ is a complete metric space, $n \ge 1$, and 
		$\Lambda \ge 1$, $a \ge 0$ are constants. For a closed set $Y \sub X$, a cycle 
		\[
		S \in \bZ_{n}(X,Y) := \{Z \in \bI_{n}(X): \spt(\D Z) \sub Y\}
		\]
		relative to $Y$ will be called {\em $(\Lambda,a)$-quasi-minimizing mod\/ $Y$} if, 
		for all $x \in \spt(S)$ and almost all $r > a$ such that 
		$\B{x}{r} \cap Y = \es$, the inequality
		\[
		\M(S \on \B{x}{r}) \le \Lambda\,\M(T)
		\] 
		holds whenever $T \in \bI_{n}(X)$ and $\D T = \D(S \on \B{x}{r})$
		(recall that $S \on \B{x}{r} \in \bI_{n}(X)$ for almost all $r > 0$). A current $S \in \bI_{n}(X)$ 
		is {\em $(\Lambda,a)$-quasi-minimizing} or a {\em $(\Lambda,a)$-quasi-minimizer} 
		if $S$ is $(\Lambda,a)$-quasi-minimizing mod $\spt(\D S)$, and we say that $S$ is 
		{\em quasi-minimizing} or a {\em quasi-minimizer} if this holds for some
		$\Lambda \ge 1$ and $a \ge 0$.
		
		When $X$ is proper, then we can define $(\Lambda,a)$-quasi-minimizing similarly for local cycle $S\in \bZ_{n,\loc}(X,Y)$ (in this case $T\in \I_{n,\cs}(X)$). 	
	\end{definition}
	
	\begin{definition}
		\label{def:minimizing}
		Suppose $X$ is a complete metric space. We say an element $S\in \I_n(X)$ is \emph{minimizing}, or $S$ is a \emph{minimizer}, if $\M(S)\le \M(T)$ for any $T\in\I_n(X)$ with $\D T=\D S$. For a constant $M\ge 1$, we say $S$ is \emph{$M$-minimizing}, if for each piece $S'$ of $S$, we have $\M(S')\le M\cdot \M(T)$ for any $T\in\I_n(X)$ with $\D T=\D S'$. Note that $S$ is minimizing if and only if $S$ is 1-minimizing. A local current $S\in \I_{n,\loc}(X)$ with $X$ being proper is \emph{minimizing}, if each compact supported piece of $S$ is minimizing. We define $M$-minimizing for local currents in a similar way.
	\end{definition}
	
	Obviously every minimizing $S \in \bI_{n}(X)$ is $(1,0)$-quasi-minimizing. 		For $S\in \bZ_n(X)$, we define $\Fill(S):=\inf \{\M(T):T\in \I_{n+1}(X), \D T=S\}$. When $S$ is compactly supported and $X$ satisfies EII$_n$, we can define $\Fill(S)$ by requiring $T\in \I_{n+1,c}(X)$ -- this gives rise to the same number (see \cite[Proposition 4.3]{wenger2011compactness}). The next theorem guarantees the existence of a minimal filling under extra assumptions.
	\begin{theorem} \label{thm:plateau}
		\cite[Theorem 2.4]{higherrank}
		Let $n \ge 1$, and let $X$ be a proper metric space satisfying 
		condition~{\rm (CI$_{n-1}$)}. Then for every $R \in \bZ_{n-1,\cs}(X)$
		there exists a filling $S \in \bI_{n,\cs}(X)$ of $R$ with mass $\M(S)=\Fill(R)$.
		Furthermore, $\spt(S)$ is within distance at most $(\M(S)/\de)^{1/n}$
		from $\spt(R)$ for some constant $\de > 0$ depending only on $n$ and constants in condition~{\rm (CI$_{n-1}$)}.
	\end{theorem}

When $X$ is not proper, minimizers might not exist. However, the following lemmas gives useful control on the geometry of ``almost minimizers'' in two different senses.

\begin{lemma}
	\label{lem:density0}
Let $n\ge 1$ and let $X$ be a complete metric space satisfying {\rm (EII$_{n-1}$)} (by Theorem~\ref{thm:isop-ineq}, this holds when $X$ satisfies {\rm (CI$_{n-1}$)}). For every $S\in \bZ_{n-1}(X)$ and every $\eps>0$, there exists $T\in \bI_n(X)$ such that 
\begin{enumerate}
	\item $\M(T)\le (1+\eps)\Fill(S)$;
	\item for every $x\in \spt(T)$ and every $0\le r\le d(x,\spt(S))$, we have $$
	\frac{1}{r^n} \|T\|(\B{x}{r}) \ge D$$
\end{enumerate}
where $D>0$ depending only on $n$ and the constants of {\rm (EII$_{n-1}$)}.
\end{lemma}
Lemma~\ref{lem:density0} is a special case of \cite[Proposition 4.3]{wenger2011compactness}.
	
	\begin{lemma}[density] \label{lem:density}
		Let $n \ge 1$, let $X$ be a complete metric space satisfying 
		condition~{\rm (CI$_{n-1}$)}, and let $Y \sub X$ be a closed set. 
		If $S \in \bZ_{n}(X,Y)$ (or $\bZ_{n,\loc}(X,Y)$ when $X$ is proper) is $(\Lambda,a)$-quasi-minimizing mod~$Y$,
		and if $x \in \spt(S)$ and $r > 2a$ are such that $\B{x}{r} \cap Y = \es$,
		then
		\[
		\frac{1}{r^n} \|S\|(\B{x}{r}) \ge D
		\] 
		for some constant $D > 0$ depending only on $n$, the constants
		$c_0$ from Definition~\ref{def_coning_ineq}, and $\Lambda$. 
	\end{lemma}
Lemma~\ref{lem:density} is \cite[Lemma 3.3]{higherrank}. Only the local current case was proved there, however, the same proof works for Ambrosio-Kirchheim currents in complete metric spaces, using Theorem~\ref{thm:isop-ineq}.

	\begin{lemma}\label{lem:fill-density}
		\cite[Lemma 3.4]{higherrank}
		Let $n \ge 1$, let $X$ be a complete metric space satisfying
		condition~{\rm (CI$_{n-1}$)}, and let $Y \sub X$ be a closed set. 
		If $S \in \bZ_n(X,Y)$ (or $S\in \bZ_{n,\loc}(X,Y)$ when $X$ is proper) is $(\Lambda,a)$-quasi-minimizing mod $Y$,
		and if $x \in \spt(S)$ and $r > 4a$ are such that $\B{x}{r} \cap Y = \es$, 
		then 
		\[
		\frac{1}{r^{n+1}}
		\inf\{\M(V) : V \in \bI_{n+1,\cs}(X),\,\spt(S-\D V) \cap \B{x}{r} = \es\} \ge c
		\]
		for some constant $c > 0$ depending only on $n$, the constant $D$ from 
		Lemma~\ref{lem:density}, and $\Lambda$.
	\end{lemma}

\section{Approximating currents by currents with uniform density}
\label{sec_approx}

In this section we describe an approximating procedure which improves the density properties of the support of currents. This will be used in Section~\ref{sec_definition} and Section~\ref{sec:cone and space}.
\subsection{The approximation}
Let $X$ be a complete metric space. For $\kappa>0$ we define a distance function $d_\kappa$ on the set of integral currents $\I_n(X)$ by
$$
d_\kappa(\tau_1,\tau_2)=\kappa\M(\tau_1-\tau_2)+\M(\D\tau_1-\D\tau_2)\,.
$$
It follows from \cite{ambrosio2000currents} that:
\begin{lemma}
	$(\I_n(X),d_\kappa)$ is a Banach space.
\end{lemma}

Suppose $\si\in\I_{n-1}(X)$  satisfying $\M(\si)\le C$ (we allow $\D\si\neq 0$).
We define $\Theta:\I_n(X)\ra \R$ by $\Theta(\tau)=\kappa\M(\tau)+\M(\si-\D\tau)$. Note that $\Theta$ is lower semicontinuous with respect to $d_\kappa$, since convergence in $\I_n(X)$ implies weak convergence of currents, and the mass is lower semicontinuous with respect to weak convergence.

Note that $\Theta(\0)=\M(\si)$. It follows from Ekeland's variational principle that there exists $\tau$ such that
\ben[label=(\alph*)]
\item $\Theta(\tau)\le \Theta(\0)=\M(\si)$;
\item   The functional $\tau'\to\tilde \Theta(\tau'):= \Theta(\tau')+\frac{1}{2}d_\kappa(\tau,\tau')$ defined on $\I_n(X)$ attains its minimum at $\tau'=\tau$.
\een

\begin{prop}
	\label{prop_approx}
	Suppose $X$ is a complete metric space satisfying condition~{\rm (EII$_{n-1}$)} (cf. Definition~\ref{def_eii}) with constant $D$. Let $\si\in\I_{n-1}(X)$. Suppose $\kappa>1$. Let $\tau$ be a minimum of $\tilde \Theta$ as above. Define $\si'=\si-\D\tau$. 
	There exist constants $\lambda$ and $\lambda'$ depending only on $X$, and $\bar\lambda$ depending only on $D$ and $C$ such that the following holds.
	\begin{enumerate}
		\item $\Fill(\si-\si')\le \frac{\M(\si)}{\kappa}$ and $\M(\si')\le \M(\si)$.
		\item For each point $x\in \spt(\si')$, we have $\M(\si' \on B_x(r))\ge \lambda r^{n-1}$ for $0\le r\le\min\{ \frac{\lambda'}{\kappa},d(x,\spt(\D\si'))\}$. 
		In particular, if $\si$ is cycle, or more generally if $\spt(\D \si)$ is compact, then $\spt(\si')$ is compact.
		\item $\spt(\si')\subset N_{a}(\spt(\si))$ where $a\le\frac{\bar\lambda\ln(\kappa)}{\kappa}$. 
		\item $\spt(\tau)\subset N_b(\spt(\si))$ where $b\le  \left( \frac{C}{\lambda\kappa}\right)^{\frac{1}{n}}+\frac{\bar\lambda\ln(\kappa)}{\kappa}$.
	\end{enumerate}
\end{prop}

\begin{proof}
	(1) follows from (a). Now we prove (2).

	Let $x$ be as in (2). Let $f(r)=\M(\si' \on B_x(r)))$. Let $\alpha$ be a filling of $\slc{\si',d_x,r}$ with $\M(\al)\le D (\M(\slc{\si',d_x,r}))^{\frac{n-1}{n-2}}$. Let $W$ be a filling of $\alpha-\si'\on B_x(r)$ such that
	$$
	\M(W)\le D(\M(\al)+(\M(\si'\on B_x(r)))^{\frac{n}{n-1}}.
	$$

	Define $\tau'=\tau-W$. Since $\Theta(\tau')+\frac{1}{2}d_\kappa(\tau,\tau')\ge \Theta(\tau)$, we have
	\begin{align*}
	&0\le (\Theta(\tau')-\Theta(\tau))+\frac{1}{2}d_\kappa(\tau,\tau')\\
	&\leq(-\M(\si' \on  B_x(r)))+\M(\al)+\kappa\M(W))
	+\frac{1}{2}(\kappa\M(W)+\M(-\si'\on B_x(r)+\alpha))\\
	&\le-\frac{1}{2}\M(\si' \on  B_x(r)))+\frac{3}{2}\M(\al)+\frac{3\kappa}{2}\M(W).
	\end{align*}  
	By coarea inequality, we have $\M(\slc{\si',d_x,r})\le f'(r)$ for a.e. $r$. It follows that $\M(\al)\le D(f'(r))^{\frac{n-1}{n-2}}$ and $\M(W)\le D(f(r)+\M(\al))^{\frac{n}{n-1}}$. Thus
	\begin{equation}
	\label{eqn_f_differential_inequality1}
	-f(r)+3D(f'(r))^{\frac{n-1}{n-2}}+3\kappa D[f(r)+D(f'(r))^{\frac{n-1}{n-2}}]^{\frac{n}{n-1}}\ge 0.
	\end{equation}
	
	Now we estimate the solution of \eqref{eqn_f_differential_inequality1} by comparison.
	Consider the following equation:
	\begin{equation}
	\label{eqn_mu1}
	-y+3D\mu+3\kappa D(y+D\mu)^\frac{n}{n-1}=0\,.
	\end{equation}
		If we consider the left hand side of Equation~(\ref{eqn_mu1}) as a function $\varphi(y,\mu)$, 
		then we have $\D_\mu\varphi(0,0)=3D$. 
		By the implicit function theorem, there exists $\hat\lambda_0$ depending only on $\kappa$ and $D$ such that 
		locally around $(0,0)$ the solution of Equation~(\ref{eqn_mu1}) is given by the graph of a function $\mu(y)$
		over the interval $[0,\hat\lambda_0]$. Since $\frac{\D\mu}{\D y}(0)=-\frac{\D_y\varphi}{\D_\mu \varphi}(0,0)=\frac{1}{3D}$,
		we may additionally assume $\frac{y}{6D}\leq \mu$ for $y\in [0,\hat\lambda_0]$. By rescaling Equation~(\ref{eqn_mu1}), we can make the dependence of $\hat \lambda_0$ on $\kappa$
		explicit, namely $\hat\lambda_0=(\frac{1}{ \lambda_0\kappa})^{n-1}$ for some constant $\lambda_0$, depending only on $D$.

	Thus for a.e. $r$, either $f(r)\ge (\frac{1}{ \lambda_0\kappa})^{n-1}$; or 
	\begin{equation}
	\label{eqn_classical_terms1}
	-f(r)+6D(f'(r))^{\frac{n-1}{n-2}}\ge 0\,.
	\end{equation}
	Indeed, in the latter case, we set $y=f(r)$, compare  \eqref{eqn_mu1} and \eqref{eqn_f_differential_inequality1} to find 
	$ (f'(r))^{\frac{n-1}{n-2}}\ge \mu\ge \frac{y}{6D}$ (the second inequality follows from the claim).
	As $f(r)$ is nondecreasing and $f(0)=0$, by solving \eqref{eqn_classical_terms1}, we know there exists $\lambda$ depending only on $D$ such that either $f(r)\ge (\frac{1}{ \lambda_0\kappa})^{n-1}$, or $f(r)\ge \lambda r^{n-1}$. Now the first assertion of (2) follows if we put $\lambda'={\lambda}^{\frac{-1}{n-1}}\cdot{\lambda_0}^{-1}$. Indeed, for $0\le r\le\min\{ \frac{\lambda'}{\kappa},d(x,\spt(\D\si'))\}$, if $f(r)\ge (\frac{1}{ \lambda_0\kappa})^{n-1}$, then $f(r)\ge (\frac{1}{ \lambda_0\kappa})^{n-1}=\lambda\cdot (\frac{\lambda'}{\kappa})^{n-1}\ge \lambda r^{n-1}$.
	
	Now we prove the ``in particular'' assertion of (2). Note that $\spt(\D\si')$ is compact as $\D \si'=\D\si$. As $\M(\si')$ is finite, we can use the first part of (2) to bound the number of disjoint $\eps$-balls in $\spt(\si')$ which are away from $\spt(\D\si')$. This together with the compactness of $\spt(\D\si')$ imply that $\spt(\si')$ is totally bounded. Thus $\spt(\si')$ is compact as $\spt(\si')$ is closed and $X$ is complete.

	Now we prove (3). Let $x\in \spt(\si')$. Let $f(r)=\M(\tau \on B_x(r))$. Then
	\begin{equation}
	\label{eq_Ekeland}
	\Theta(\tau- \tau \on B_x(r))+\frac{1}{2}d_\kappa(\tau,\tau-\tau \on B_x(r)) \ge \Theta(\tau)\,.
	\end{equation}
	
	for $r< d(x,\spt(\si))$. Note that 
	for $r< d(x,\spt(\si))$,
	$$
	\Theta(\tau- \tau \on B_x(r))-\Theta(\tau)=	-\kappa\M(\tau \on B_x(r)))-\M(\si'\on B_x(r))+\M(\slc{\tau,d_x,r}).
	$$
	This together with \eqref{eq_Ekeland} imply that
	\begin{align*}
	&0\le (\Theta(\tau- \tau \on B_x(r))-\Theta(\tau))+\frac{1}{2}d_\kappa(\tau,\tau-\tau \on B_x(r)) \\
	&\le	(-\kappa\M(\tau \on B_x(r)))-\M(\si' \on B_x(r)))+\M(\slc{\tau,d_x,r}))\\
	&+\frac{1}{2}(\kappa\M(\tau \on B_x(r)))
	+\M(\si'\on B_x(r)-\slc{\tau,d_x,r}))\\
	&\le -\frac{\kappa}{2}\M(\tau \on B_x(r)))-\frac{1}{2}\M(\si' \on B_x(r)))+\frac{3}{2}\M(\slc{\tau,d_x,r}).
	\end{align*}
	
	In particular, 
	\begin{align*}
	& -\frac{\kappa}{2}\M(\tau \on B_x(r))+\frac{3}{2}\M(\slc{\tau,d_x,r})\ge 0,\ \mathrm{for}\ 0<r< d(x,\spt(\si)),\ \mathrm{and}\\
	& -\frac{\kappa}{2}\M(\tau \on B_x(r))-\frac{1}{2}\lambda r^{n-1}+\frac{3}{2}\M(\slc{\tau,d_x,r})\ge 0,\ \mathrm{for}\ 0<r\le \frac{\lambda'}{\kappa}.
	\end{align*}
	The coarea inequality implies that $\M(\slc{\tau,d_x,r})\le f'(r)$ for a.e. $r$. Thus
	\begin{align}
	\label{eq_ode1} & -\kappa f(r)+3f'(r)\ge 0,\ \mathrm{for}\ 0<r< d(x,\spt(\si)),\ \mathrm{and}\\
	\label{eq_ode2} & -\kappa f(r)-\lambda r^{n-1}+3f'(r)\ge 0,\ \mathrm{for}\ 0<r\le \frac{\lambda'}{\kappa}.
	\end{align}
	Since $\ln(f(r))$ is nondecreasing, by \eqref{eq_ode1} we know for $r_2>r_1$,
	$$
	\ln(f(r_2))-\ln(f(r_1))\ge \int_{r_1}^{r_2}\frac{f'(t)}{f(t)}dt\ge\int_{r_1}^{r_2}\frac{\kappa}{3} dt=\frac{\kappa}{3}(r_2-r_1).
	$$
	Thus $e^{-\frac{\kappa}{3} r}f(r)$ is nondecreasing. This together with \eqref{eq_ode2} imply
	$$
	e^{-\frac{\kappa}{3} r}f(r)-e^{0}f(0)\ge \int_{0}^{r}(e^{-\frac{\kappa}{3} t}f(t))'dt\ge \frac{1}{3}\int_0^{r}\lambda e^{-\frac{\kappa}{3} t}t^{n-1}dt.
	$$
	Thus
	$$
	f(\frac{\lambda'}{\kappa})\ge \frac{\lambda}{3}\int_{0}^{\frac{\lambda'}{\kappa}} e^{-\frac{\kappa}{3} t}t^{n-1}dt=\frac{3^{n-1}\lambda}{\kappa^n}\int_{0}^{\frac{\lambda'}{3}}e^{-u}u^{n-1}du.
	$$
	Define $\lambda''=\lambda\int_{0}^{\frac{\lambda'}{3}}e^{-u}u^{n-1}du$. Then for $\frac{\lambda'}{\kappa}<r<d(x,\spt(\si))$,
	$$
	e^{-\frac{\kappa}{3} r}f(r)\ge e^{-\frac{\kappa}{3}\cdot \frac{\lambda'}{\kappa}} f(\frac{\lambda'}{\kappa})\ge e^{-\frac{\lambda'}{3}}\cdot \frac{3^{n-1}\lambda''}{\kappa^n}.
	$$
	Thus $f(r)\ge e^{\frac{\kappa}{3} r-\lambda'}\cdot \frac{3^{n-1}\lambda''}{\kappa^n}$. On the other hand, $f(r)\le \M(\tau)\le \frac{C}{\kappa}$. As $\kappa>1$ and $n-1\ge 1$, we conclude that $r\le \frac{\bar\lambda\ln(\kappa)}{\kappa}$ for $\bar\lambda$ depending only on $C$ and $D$.
	
	(4) is similar to \cite[Theorem 10.6]{ambrosio2000currents}. We replicate the proof for the convenience of the reader. Let $x\in \spt(\tau)\setminus \spt(\D\tau)$ and let $r_x=d(x,\spt(\D\tau))$. Let $f(r)=\M(\tau\on B_x(r))$. For $0<r<r_x$, let $\alpha$ be a filling of $\slc{\tau,d_x,r}$ such that $\M(\alpha)\le D(\M(\slc{\tau,d_x,r}))^{\frac{n}{n-1}}$. Let $\tau'=\tau-\tau\on B_x(r)+\alpha$. We deduce from $(\Theta(\tau')-\Theta(\tau))+\frac{1}{2}d_\kappa(\tau,\tau')\ge 0$ that
	\begin{equation*}
	\kappa(-\M(\tau\on B_x(r))+\M(\alpha))+\frac{\kappa}{2}(\M(\tau\on B_x(r))+\M(\alpha))\ge 0.
	\end{equation*}
	Thus 
	\begin{equation*}
	-\M(\tau\on B_x(r))+3\M(\alpha)\ge 0.
	\end{equation*}
	For a.e. $0<r<r_x$, $\M(\slc{\tau,d_x,r})\le f'(r)$. Thus $\M(\alpha)\le D(f'(r))^{\frac{n}{n-1}}$. Hence
	$$
	-f(r)+3D(f'(r))^{\frac{n}{n-1}}\ge 0\,.
	$$
	Thus $f(r)\ge \lambda r^n$ for $0<r<r_x$. Since $f(r)\le \M(\tau)\le \frac{C}{\kappa}$ by (1), we know $r_x\le \left( \frac{C}{\lambda\kappa}\right)^{\frac{1}{n}}$. Now (4) follows from (3).
\end{proof}

\subsection{Isomorphisms of homologies}
Let $X$ be a complete metric space. Let $\tilde H_*(X)$, $\hl_*(X)$, $\ha_{*,c}(X)$, $\ha_*(X)$ be the reduced homology groups induced by the chain complex of singular chains, Lipschitz chains, compactly supported Ambrosio-Kirchheim (AK) integral currents and finite mass AK integral currents. For a subset $Y\subset X$, the collection of elements in $\I_*(X)$ with support contained in $Y$ is stable under taking boundary. Thus we can define the relative homology groups $H^{AK}_{*,c}(X,Y)$ and $H^{AK}_{*}(X,Y)$ in the usual way.

There are natural maps $\hl_*(X)\to \tilde H_*(X)$, $\hl_*(X)\to \ha_{*,c}(X)$ and $\ha_{*,c}(X)\to \ha_*(X)$ (see \cite{riedweg2009singular,mitsuishi2013coincidence}).

\begin{proposition}\label{prop_equal_homology}
	Suppose $X$ is a complete metric space. Let $K\subset X$ be a compact set. Then
	\begin{enumerate}
		\item Suppose $X$ satisfies coning inequalities up to dimension $n$ for singular chains, Lipschitz chains and compactly supported currents. Then $H^{\textrm{L}}_n(X,X-K)\to H_n(X,X-K)$, $H^{\textrm{L}}(X,X-K)\to H^{AK}_{n,c}(X,X-K)$, $\hl_n(X)\to \tilde H_n(X)$ and $\hl_n(X)\to \ha_{n,c}(X)$ are isomorphisms.
		\item Suppose $X$ satisfies (EII$_{n+1}$), then $H^{AK}_{n,c}(X,X-K)\to H^{AK}_{n}(X,X-K)$ and $\ha_{n,c}(X)\to \ha_n(X)$ are isomorphisms.
	\end{enumerate}
\end{proposition}

\begin{proof}
	(1) follows from \cite{riedweg2009singular} and \cite{mitsuishi2013coincidence} (note that the notion of coning inequalities in these two references are different from our notion, however, our definitions imply their definitions).
	
	We only prove the first half of (2), as the second half is similar. 
	First we prove surjectivity. Take $\tau\in \I_{n}(X)$ representing as class in $H^{AK}_{*}(X,X-K)$. As $d(\spt(\D\tau),K)>0$, we can approximate $\D\tau$ by a $\si'\in\I_{n-1,c}(X)$ using Proposition~\ref{prop_approx} such that $\si'-\D\tau=\D\tau'$ for $\tau'$ supported in an $\eps$-neighborhood of $\spt(\D\tau)$ (note that $\spt(\si')$ is indeed compact by Proposition~\ref{prop_approx} (2) as $\D\tau$ is cycle). Thus we can assume $\spt(\tau')\cap K=\emptyset$. Thus $[\tau_0=\tau+\tau']=[\tau]$ in $H^{AK}_{*}(X,X-K)$, and $\spt(\D\tau_0)=\spt(\si')$ is compact. Now we use Proposition~\ref{prop_approx} again to approximate $\tau_0$ by $\tau'_0\in\I_n(X)$ such that $\D\tau_0=\D\tau'_0$ and $\spt(\tau'_0)$ is compact (this is possible as $\spt(\D\tau'_0)$ is compact, cf. Proposition~\ref{prop_approx} (2)). Thus $[\tau'_0]=[\tau_0]$ in $H^{AK}_{*}(X,X-K)$, which implies surjectivity.
	
	Now we verify injectivity. Let $\tau\in \I_{n,c}(X)$ represent a class in $H^{AK}_{n,c}(X,X-K)$. Suppose $\tau-\tau'=\D \alpha$ for $\alpha\in \I_{n+1}(X)$ and $\tau'\in\I_{n}(X)$ with $\spt(\tau')\cap K=\emptyset$. As $\spt(\tau')=\spt(\tau)$ is compact, by Proposition~\ref{prop_approx} we can approximate $\tau'$ by $\tau''\in \I_n(X)$ with compact support and $\D\tau''=\D\tau'$ such that $\tau'-\tau''=\D\alpha'$ for $\alpha'\in \I_{n+1}(X)$. Let $\alpha_0=\alpha+\alpha'$. Then $\D\al_0=\tau-\tau''$, which is compactly supported. We use Proposition~\ref{prop_approx} again to approximate $\alpha_0$ by an element in $\I_{n+1,c}(X)$ with the same boundary as $\alpha_0$. This implies $[\tau]=0$ in $H^{AK}_{n,c}(X,X-K)$. Hence injectivity follows.
\end{proof}

\section{Quasiflats in metric spaces}
\label{sec:quasiflats in metric spaces}
The main goal of this section is to establish  some auxiliary results which provide controlled chains between cycles and their ``projections'' on a quasiflat. 
Readers who are mainly interested in quasiflats in metric spaces with geodesic bicombings only need to read Section~\ref{subsec:Lipschitz quasifalts}. Section~\ref{subsec_cubulated_quasiflats} and Section~\ref{subsec:homology retract} concern the more general case where a quasiflat might not be represented by a continuous map.

\subsection{Cubulated quasiflats and quasidisks}
\label{subsec_cubulated_quasiflats}
We consider a quasiflat $Q$  in a complete metric space satisfying condition (CI$_{n-1}$) represented by the map $\Phi:\R^n\to X$. The goal is to define a ``push-forward map'' sending a current in $\R^n$ 
to a current in $X$ and to define a ``projection map'' sending a current in $X$ to a current close to $Q$.

Here we say that $\mathcal{C}_R$ is a {\em regular cubulation at scale} $R$ of $\R^n$ if its zero-skeleton $\mathcal{C}^{(0)}_R$ is given by $v+R\cdot\Z^n$ for some $v\in \R^n$.
If $\Phi:\R^n\to X$ is an $(L,A)$-quasi-isometric embedding, then we call a regular cubulation $\mathcal{C}_{R_0}$ {\em admissible (for $\Phi$)}, 
if its scale is $R_0=2LA$.

\begin{definition}
	\label{def_retraction}
	Let $\Phi:\R^n\to X$ be an $(L,A)$-quasi-isometric embedding with image $Q$. Let 
	$\mathcal{C}$ be an admissible regular cubulation of $\R^n$.
	A {\em $\lambda$-quasi-retraction} (associated to $\Phi$ and $\mathcal{C}$) is a map $\pi:X\to\R^n$ such that
	
	\begin{enumerate}
		\item $\pi$ is $\lambda$-Lipschitz;
		\item $(\pi\circ \Phi)|_{\mathcal{C}^{(0)}}=\id$ and  $d(\pi\circ \Phi(x),x)<\lambda$ for any $x\in\R^n$;
		\item $d(\Phi\circ \pi(x),x)<\la$ for any $x\in Q$. 
	\end{enumerate}
\end{definition}

Note that  since $\mathcal{C}$ is admissible, $\Phi|_{\mathcal{C}^{(0)}}$ is $2L$-bilipschitz. Hence any
Lipschitz extension  $\pi:X\to\R^n$ of $(\Phi|_{\mathcal{C}_R^{(0)}})^{-1}$
is a $\lambda$-quasi-retraction associated to $\Phi$ and $\mathcal{C}$. By McShane's extension result, we can arrange  $\lambda$ to depend only on $L$ and $n$.

\begin{proposition}[cubulated quasiflats] \label{prop_chain_map}
	Let $n \ge 2$, and let $X$ be a complete metric space satisfying 
	condition~{\rm (CI$_{n-1}$)}. Then for all 
	$L,A$ there exist $\Lambda,a$ depending only on $L,A,n$ and $X$ such that the following holds.
	Suppose that $\mathcal{C}_R$ is a regular cubulation of $\R^n$ at scale $R=2LA$.
	Let $\cP_*(\R^n)$ denote the collection of integral currents represented by cubical chains with respect to $\mathcal{C}_R$. 
	If $\Phi \colon \R^n \to X$ is an $(L,A)$-quasi-isometric embedding, then there 
	exists a chain map $\iota \colon \cP_*(\R^n) \to \bI_{*,\cs}(X)$ such that 
	\ben
	\item
	$\iota$ maps every vertex $\bb{x_0} \in \cP_0(\R^n)$ to $\bb{\Phi(x_0)}$ and, for 
	$1 \le k \le n$, every oriented cube $B \in \cP_k(\R^n)$ 
	to a current with support in $N_a(\Phi(B))$; 
	\item 
	$\M(\iota T)\leq a\cdot \M(T)$ for all $T\in \cP_*(\R^n)$;
	\item 
	
	For every top-dimensional chain $W\in\cP_n(\R^n)$, we have
	$\iota\bb{W} \in \bI_{n}(X)$ is $(\Lambda,a)$-quasi-minimizing mod
	$N_a(\Phi((\spt (\D W))^{(0)}))$;
	\item 
	For every top-dimensional chain $W\in\cP_n(\R^n)$, we have
	$$d(\Phi(x),\spt(\iota\bb{W})) \le a$$ for all $x \in \spt(W)$ with $d(x,\spt (\D W)) \ge a$. 
	\item For every chain $P\in\cP_k(\R^n)$, there exists a homology $h\in \I_{k+1}(\R^n)$
	such that 
	\begin{itemize}
		\item $\D h=\pi_\#\iota P-P$;
		\item $\M(h)\leq a\cdot\M(P)$;
		\item $\spt(h)\subset N_a(\spt (P))$.
	\end{itemize}
	
	\een
\end{proposition}

This proposition is essentially \cite[Proposition 3.7]{higherrank}. The chain map $\iota$ is constructed skeleton by skeleton, using the condition (CI$_{n-1}$) and Theorem~\ref{thm:isop-ineq}. 
We refer the reader to \cite{higherrank} for more details. \cite[Proposition 3.7]{higherrank} requires $X$ to be proper, however, 
the same proof works for Ambrosio-Kirchheim currents in complete metric spaces (images of $\iota$ having compact supports follows from Theorem~\ref{thm:isop-ineq}). 
Assertions (2) and (5) are not in \cite[Proposition 3.7]{higherrank}, however, they are direct consequences of the construction in \cite{higherrank} and can be readily justified by induction on dimension.

\begin{lem}
	\label{lem_top_dimensional}
	Let $n \ge 2$, and let $X$ be a complete metric space satisfying 
	condition~{\rm (CI$_{n-1}$)}. Let $\Phi:\R^n\to X$ be an $(L,A)$-quasiflat. Let $\pi:X\to\R^n$ be as in the beginning of Section~\ref{subsec_cubulated_quasiflats}. Let $\iota$ be as in Proposition~\ref{prop_chain_map}. 
	Then there exists $a'$ depending only on $L,A,n$ and $X$ such that for every top-dimensional chain $W\in\cP_n(\R^n)$, we have
	\begin{itemize}
		\item $x\in\spt(\pi_\#\circ\iota(W))$ for all $x \in \spt (W)$ with $d(x,\spt (\D W)) \ge a'$;
		\item $\spt(\pi_\#\circ\iota(W))\subset N_{a'}(\spt(W))$.
	\end{itemize}
	
\end{lem}
The lemma follows by applying Proposition~\ref{prop_chain_map} (5) to $P=\D W$.

\begin{definition}[Chain projection]
	\label{def_chain_projection}
	Let $X$ be a complete metric space satisfying condition (CI$_{n-1}$). Let $Q$ be an $(L,A)$-quasiflat represented by $\Phi:\R^n\to X$. 
	Let $\iota$ be as in Proposition~\ref{prop_chain_map} and let $\pi: X\to \R^n$ be as in Definition~\ref{def_retraction}. We now define a map sending each 
	$\si\in \I_{n,\cs}(X)$ to a current supported in a neighborhood of $Q$. 
	First applying the Federer-Fleming deformation to $\pi_\# \si$ to obtain a cubical chain $\si'$, and then define $\si_Q=\iota(\si')$. 
	Note that there exists $\lambda$ depending only on $L,A, n$ and $X$ such that $\M(\si_Q)\le \lambda\cdot \M(\si)$ and 
	$\spt(\si_Q)\subset N_{\lambda}(\Phi\circ \pi (\spt(\si)))$.
\end{definition}

\begin{remark}
	\label{rmk:quasidisk}
	Definition~\ref{def_retraction}, Proposition~\ref{prop_chain_map} and Definition~\ref{def_chain_projection} also apply to quasidisks. We can take the domain of a quasidisk to be single cube with a suitable cubulation 
	and repeat the previous discussion.
\end{remark}

\subsection{Homology retract}
\label{subsec:homology retract}
We now describe a way of projecting cycles to quasiflats. 

\begin{proposition}
	\label{lem_homology}
	There exists a constant $C=C(L,A,n,m,c)$ such that the following holds.
	Suppose that $X$ is a complete  metric space satisfying condition {\rm (SCI$_{m}$)} with constant $c$ and let $\Phi:\R^n\to X$ be an  
	$(L,A)$-quasiflat with image $Q$.
	Let $\mathcal{C}_{R_0}$ be a regular cubulation  of $\R^n$ at scale $R_0=2LA$.
	Denote by $\varphi:X\to\R^n$ a quasi-retraction induced by $\Phi$ and $\mathcal{C}_{R_0}$ as in Definition~\ref{def_retraction}. 
	Let  $\iota:\cP_*(\R^n)\to \I_{*,c}(X)$ be a chain map induced by $\Phi$ and $\mathcal{C}_{R_0}$ as in Proposition~\ref{prop_chain_map}.
	Then  for all $R\geq R_0$ the following holds true.

	If $S\in \bZ_{m}(X)$ is a cycle with
	$\spt (S)\subset N_{R}(Q)$, then there exists a cubical cycle $P\in\cP_m(\R^n)$, $\D P=0$, and homologies  $H\in \I_{m+1}(X)$ and $h\in\I_{m+1}(\R^n)$ with the following properties. 
	\begin{enumerate}
		\item $\D H=S-\iota P$\quad and \quad$\D  h=\varphi_\# S-P$;
		\item $\spt (H)\subset N_{CR}(\spt (S))$\quad and \quad$\spt (h)\subset N_C(\varphi(\spt (S)))$;
		\item $\M(H)\leq CR\cdot\M(S)$\quad and \quad$\M(h)\leq C\cdot \M(S)$;
		\item $\M(P)\leq C\cdot\M(S)$.
	\end{enumerate}
\end{proposition}

This proposition is proved in Section~\ref{subsec_proof}.

\begin{corollary}\label{cor_small_fill_close_to_Q}
	Let $X$ be a complete metric space satisfying {\rm (SCI$_n$)} with constant $c$ and let $\Phi:\R^n\to X$ be an $(L,A)$-quasiflat with image $Q$. 
	Then there exists $C=C(L,A,c,n)$ such that the following holds.
	If $S\in \bZ_{n}(X)$ is a cycle, with
	$\spt (S)\subset N_{R}(Q)$ for some $R\geq 2LA$, then $\Fill(S)\leq CR\cdot\M(S)$.
\end{corollary}

\begin{proof}
	Proposition~\ref{lem_homology} provides a controlled homology $H$ between $S$ and $\iota P$ for some cubical cycle $P$ in $\R^n$. 
	However, $P$ is top-dimensional and therefore trivial.
	It follows that $H$ is a filling of $S$ as required.
\end{proof}

If we make the stronger assumption that $A=0$, namely that $Q$ is a bilipschitz flat, then instead of
using the chain map $\iota$ we can use an actual Lipschitz retraction $\pi:X\to Q$ and push the cycle $S$ to a cycle $S'=\pi_\# S$. 
In this case Proposition~\ref{lem_homology} simplifies to:

\begin{lemma}
	\label{lem_homology_bilip}
	There exists a constant $C=C(L,n,m,c)$ such that the following holds.
	Suppose that $X$ is a metric space satisfying condition {\rm (SCI$_m$)} with constant $c$ and let $Q\subset X$ be an $n$-dimensional $L$-bilipschitz flat.
	Then for all $R\geq 0$ the following holds true. 
	If $S\in \bZ_{m}(X)$ is a cycle with
	$\spt (S)\subset N_{R}(Q)$, then there exists a cycle $S'\in\bZ_m(Q)$ and a homology  $H\in \I_{m+1}(X)$ such that 
	\begin{enumerate}
		\item $\D H=S-S'$;
		\item $\spt (H)\subset N_{CR}(\spt (S))$;
		\item $\M(H)\leq CR\cdot\M(S)$;
		\item $\M(S')\leq C\cdot\M(S)$.
	\end{enumerate}
	
\end{lemma}

Motivated by the above results, we formulate the following properties for quasiflats in metric spaces.

\begin{definition}
	\label{def_tame}
	Let $\Phi:\mathbb R^n\to X$ be an $n$-dimensional $(L,A)$-quasiflat with image $Q\subset X$ and let $\iota$ be the chain mapping in 
	Proposition~\ref{prop_chain_map}. For $a>0$, $Q$ is an $a$-\emph{homology retract} 
	if for any $S\in \bZ_{n-1}(X)$ such that $\spt(S)\subset N_R(Q)$ with $R>a$, there exist a cubical chain $P\in \I_{n-1}(\mathbb R^n)$ and 
	$H\in\I_{n}(X)$ such that
	\begin{enumerate}
		\item $\partial H=S-\iota(P)$ and $\spt(\iota(P))\subset N_a(Q)$;
		\item $\spt(H)\subset N_{aR}(\spt(S))$ and $\spt(P)\subset N_a(\pi(\spt(S)))$;
		\item $\M(H)\le aR\cdot\M(S)$ and $\Fill(P-\pi_\#(S))\le a\cdot\M(S)$;
		\item $\M(P)\le a\cdot\M(S)$ and $\M(\iota(P))\le a\cdot \M(S)$.
	\end{enumerate}
\end{definition}

It follows from Proposition~\ref{lem_homology} and Proposition~\ref{prop_chain_map} that an $n$-dimensional $(L,A)$-quasiflat in a metric space $X$ 
with condition~(SCI$_{n-1}$) is an $a$-homology retract for $a=a(L,A,n,X)$.

\section{Morse quasiflats}
\label{sec_definition}

In this section we introduce our main objective -- {\em Morse quasiflats}. More precisely, we provide several potential definitions, each expressing a higher dimensional hyperbolic feature  ``transversal'' to a quasiflat. We name one of them Morse and put off the discussion of the relation between  the different notions until the next section,
where they are all shown to be equivalent under appropriate assumptions. The conditions are organized in two groups. First, conditions on ultralimits of quasiflats in asymptotic cones. Second, asymptotic conditions on quasiflats within the space. In the last subsection we treat quasi-isometry invariance.

\subsection{Conditions in asymptotic cones}
\label{subsec_cone_conditions}
\mbox{}

Let $X$ be a complete metric space and let $Q\subset X$ be an $n$-dimensional $L$-bilipschitz flat.

\begin{definition}[Piece property]
	\label{def_piece_decomposition}
	We say $Q$ has the {\em piece property}, if the following holds. Let $\si\in \bZ_{n-1}(Q)$ be a cycle with canonical filling $\nu\in \I_{n}(Q)$.
	Then any alternate filling $\tau\in \I_{n}(X)$ of $\si$ contains $\nu$ as a piece, i.e. $\|\tau\|=\|\tau-\nu\|+\|\nu\|$ with $\|\tau-\nu\|$ concentrated on $X\setminus Q$.
\end{definition}

\begin{remark}
	The piece property implies that $Q$ is mass minimizing in the following sense. If $T$ is a piece of $Q$, then $T$ is a minimal filling of $\D T$. 
\end{remark}

\begin{lemma}\label{lem_fillarea_decreases}
	Suppose that $Q$ has the piece property and $\tau\in \I_{n}(X)$ is a filling of a cycle $\si\in \bZ_{n-1}(Q)$.
	Then the filling area of slices $\pp{\tau,d_Q,t}$ close to $Q$ becomes small,
	$\lim\limits_{t\to 0}\Fill(\pp{\tau,d_Q,t})=0$.
\end{lemma}

\begin{proof}
	Denote by $\nu\in \I_{n}(Q)$ the canonical filling of $\si$. Consider the filling $(\tau-\nu)\on\{d_Q\leq t\}$ 
	of the slice $\pp{\tau,d_Q,t}$. Note that $\tau-\nu$ is a cycle. Since $Q$ has the piece property, $\|\tau-\nu\|$ is concentrated on $X\setminus Q$ and the claim follows. 
\end{proof}

\begin{definition}[Neck property]
	\label{def_neck_decomposition}
	
	We say $Q$ has the {\em neck property}, if there exists a constant $C>0$ such that the following holds for all $\rho>0$. 
	Let $\si\in \bZ_{n-1}(X)$ be a cycle with $\spt(\si)\subset N_\rho(Q)$ and let $\tau\in \I_{n}(X)$ be a filling, $\D\tau=\si$, with $\spt(\tau)\subset X\setminus Q$. 
	Then 
	\[\Fill(\si)\leq C\cdot\rho\cdot\M(\si).\]
\end{definition}

The name derives from the fact that any chain filling a cycle in $Q$ has to have small necks near $Q$ in the following sense.

\begin{definition}[Weak neck property]
	\label{def_weak_neck_decomposition}
	
	We say $Q$ has the {\em weak neck property}, if the following holds. Let $\si\in\bZ_{n-1}(Q)$ be a cycle with a filling $\tau\in\I_{n}(X)$.
	Then for every $\eps>0$ there exists $\rho_0\in(0,\eps)$ such that the slice $\pp{\tau,d_Q,\rho_0}$
	has
	\[\Fill(\pp{\tau,d_Q,\rho_0})<\eps.\]
\end{definition}

\begin{lemma}\label{lem_neck_impl_weak_neck}
	The neck property implies the weak neck property.
\end{lemma}

\begin{proof}
	If $\tau$ is supported in $Q$, then there is nothing to show. Otherwise, we set $\tau'=\tau\on(X\setminus Q)$.
	For every $\delta>0$ we choose $\rho>0$ such that $\M(\tau'\on\{d_Q\leq\rho\})\leq \delta$. By the pigeonhole principle there exists $k\in\N$ such that 
	$\M(\tau_k)\leq\frac{\delta}{2^{k+1}}$ where $\tau_k=\tau'\on\{\frac{\rho}{2^{k+1}}\leq d_Q<\frac{\rho}{2^{k}}\}$. The coarea inequality implies $\M(\pp{\tau_k,d_Q,\rho_0})< \frac{\delta}{\rho}$ 
	for some $\rho_0\in(\frac{\rho}{2^{k+1}},\frac{\rho}{2^{k}})$. Since $\pp{\tau_k,d_Q,\rho_0}=\pp{\tau,d_Q,\rho_0}$ we conclude from the neck property
	$\Fill(\pp{\tau,d_Q,\rho_0})\leq C\cdot \rho_0 \cdot\frac{\delta}{\rho}\leq C\cdot\delta$ where $C$ is independent of $\delta$.
\end{proof}

\begin{lemma}\label{lem_pp_impl_neck}
	Suppose that $X$ satisfies condition {\rm (SCI$_{n-1}$)} where $n=\dim Q$. Then the piece property implies the neck property with constant $C$ in Definition~\ref{def_neck_decomposition} depending only on $n$, the constants of condition {\rm (SCI$_{n-1}$)} and the Lipschitz constant of $Q$.
\end{lemma}

\begin{proof}
	Suppose that $Q$ has the piece property.
	Let $\si\in \bZ_{n-1}(X)$ be a cycle with $\spt(\si)\subset N_\rho(Q)$ and let $\tau\in \I_{n}(X)$ be a filling with $\spt(\tau)\subset X\setminus Q$.
	By Lemma~\ref{lem_homology_bilip} there exists a cycle $\si'\in\bZ_{n-1}(Q)$ and a homology $H\in\I_{n}(X)$ with $\D H=\si-\si'$ and $\M(H)\leq C\cdot\rho\cdot\M(\si)$.
	Denote by $\nu$ the canonical filling of $\si'$ inside $Q$. Then $H+\nu$ is a filling of $\si$.
	The piece property implies that $\tau+H=(\tau+H+\nu)-\nu$ is a piece decomposition. Since the support of $\tau$ is disjoint from $Q$, we see that $H=(H+\nu)-\nu$
	is a piece decomposition as well. In particular, $\M(H+\nu)\leq\M(H)$ and the claim follows.
\end{proof}

\begin{definition}[Full support]
	\label{def_full_support}
	Let $n=\dim Q$.	
	We say $Q\subset X$ has {\em full support with respect to a homology theory $h_{*}$}, if the map
	\[h_n(Q,Q\setminus\{q\},\mathbb Z)\to h_n(X,X\setminus\{q\},\mathbb Z)\] 
	is injective for each $q\in Q$. More generally, we define a bilipschitz disk or half flat $K$ in $X$ has full support with respect to some homology theory if the above injectivity condition holds for all $q$ in the interior of $K$.
\end{definition}

\begin{lemma}\label{lem_neck_impl_full_spt}
	Let $n=\dim Q$. Suppose that $X$ satisfies condition {\rm (SCI$_{n}$)}. If $Q$ has the weak neck property, then $Q$ has full support with respect to $\ha_n$.
\end{lemma}

\begin{proof}
	If $Q$ does not have full support, then we find an embedded top-dimensional ball $B\subset Q$ around a point $x\in Q$ such that a nontrivial multiple of $\D\bb{B}$ can be filled by $\tau\in \I_n(X)$ with $x\notin \spt(\tau)$.
	We may even assume that it avoids the ball $B_x(1)$. Let $\nu$ be the canonical filling of $\D \tau$ in $Q$.
	Suppose that the weak neck property holds. 
	Then there exists a sequence $\rho_k\to 0$ such that the  slices $\pp{\tau,d_Q,\rho_k}$  fulfill $\Fill(\pp{\tau,d_Q,\rho_k})\to 0$. Choose fillings $\mu_k$
	of $\pp{\tau,d_Q,\rho_k}$ as provided by Theorem~\ref{thm:isop-ineq}. Set $h_k=\tau\on\{d_Q\leq\rho_k\}$. From Corollary~\ref{cor_small_fill_close_to_Q} and 
	Theorem~\ref{thm:isop-ineq} (2), we see that 
	$\Fill(\mu_k-h_k+\nu)\to 0$ and therefore $h_k$ converges to $\nu$ weakly \cite[Theorem 1.4]{wenger2007flat}.
	However, this is impossible since the support of the $h_k$ is disjoint from $B_x(1)$.
\end{proof}

At last we will close the cycle.

\begin{lemma}\label{lem_full_spt_impl_pp}
	Let $n=\dim Q$. Suppose that $X$ satisfies condition {\rm (CI$_{n-1}$)}. Suppose that each element in $\bZ_{n}(X)$ can be filled by an element in $\I_{n+1}(X)$. If $Q$ has full support with respect to $\ha_n$,
	then $Q$ has the piece property.
\end{lemma}

\begin{proof}
	Let $\si\in \I_{n-1}(Q)$ be  a nontrivial cycle with the additional property that $\h^n(\spt(\si))=0$. 
	Denote by $\nu\in \I_n(Q)$ the canonical filling of $\si$ in $Q$ and let $\tau\in \I_n(X)$ be an arbitrary filling. Consider the piece decomposition 
	$\tau=\tau\on Q+\tau\on (X\setminus Q)$. Let $\tau_Q:=\tau\on Q$ and $\tau_{Q^c}:=\tau\on (X\setminus Q)$. Then $\tau_Q$ and $\tau_{Q^c}$ are integer rectifiable currents (\cite[Definition 4.2]{ambrosio2000currents}). Moreover, $\|\tau_{Q^c}\|$ is concentrated on $X\setminus Q$. 
	
	We claim that $\tau_{Q}=\nu$. By the proof of \cite[Theorem 4.6]{ambrosio2000currents}, $\Theta_{n}(\|\tau_{Q^c}\|,x):=\lim_{r\to 0}\frac{\|\tau_{Q^c}\|(B_x(r))}{\om_nr^n}=0$
	for $\h^n$-a.e. $x\in Q$. We write  $\|\tau_{Q}-\nu\|=f\cdot\bb{Q}$  with $f\in L^1(Q,\Z)$. If our claim fails, then we find a Lebesgue point $p\in Q\setminus\spt(\si)$ of $f$
	with 
	\begin{align*}
	\Theta_n(\|\tau-\nu\|,p)&=\Theta_n(\|\tau_Q-\nu\|,p)+\Theta_{n}(\|\tau_{Q^c}\|,p)\\
	&=\Theta_n(\|\tau_Q-\nu\|,p)=f(p)=k\neq 0.
	\end{align*} 
	Then $T=k\cdot\bb{Q}-\tau+\nu$ has density zero at $p$. Hence, for every $\eps>0$ we can find a slice $\pp{T,d_p,r}$ with $\M(\pp{T,d_p,r})\leq\eps\cdot r^{n-1}$.
	Let $W_r$ be an almost minimal filling of $\pp{T,d_p,r}$ as provided by Theorem~\ref{thm:isop-ineq}. Then the isoperimetric inequality implies $\M(W_r)\lesssim\eps^{\frac{n}{n-1}}\cdot r^n$. By Theorem~\ref{thm:isop-ineq} (2), $p$ cannot lie in the support of $W_r$ for $\eps$ small enough.
	Since $W_r$ is homologous to $T\on\{d_p\leq r\}$ we see that $[T]=0$ in $\ha_n(X,X\setminus\{p\})$.
	On the other hand, we have $[T]=k\cdot\bb{Q}\in \ha_n(X,X\setminus\{p\})$ since $\tau-\nu$ can be filled by an element in $\I_{n+1}(X)$. In particular, $[T]\neq 0\in \ha_n(X,X\setminus\{p\})$
	since $Q$ has full support. Contradiction. Thus the claim follows and the piece property follows from the claim.
	
	In the general case we choose a homology $\alpha\in \I_n(Q)$ such that $\si'=\D\al+\si$ is nontrivial and fulfills $\h^n(\spt(\si'))=0$.
	Then $\tau'=\tau+\al$ is a filling of $\si'$. Note that $\nu'=\nu+\al$ is the canonical filling $\si'$ in $Q$. By the case above, we know that
	$\tau'=(\tau'-\nu')+\nu'$ is a piece decomposition. Since the support of $\al$ is contained in $Q$ we obtain the required piece decomposition for $\tau$.
\end{proof}

Let us summarize the results of this subsection.

\begin{proposition}
	\label{prop_cone_conditions}
	Let $X$ be a metric space and let $Q$ be an $n$-dimensional bilipschitz flat in $X$. Suppose $X$ satisfies condition {\rm (SCI$_{n}$)}. Then the following conditions on $Q$ are all equivalent:
	\begin{enumerate}
		\item Piece property (cf. Definition \ref{def_piece_decomposition});
		\item Neck property (cf. Definition \ref{def_neck_decomposition});
		\item Weak neck property (cf. Definition \ref{def_weak_neck_decomposition});
		\item Full support with respect to $\ha_n$ (cf. Definition \ref{def_full_support}).
	\end{enumerate}
	If in addition $X$ satisfies {\rm (EII$_{n+1}$)} and coning inequalities up to dimension $n$ for singular chains, Lipschitz chains and compact supported integral currents, then each of the above conditions is equivalent to $Q$ having full support with respect to reduced singular homology.
\end{proposition}
The last statement follows from Proposition~\ref{prop_equal_homology}.

\subsection{Asymptotic conditions for Morse quasiflats}
\label{subsec_asymptotic condition}
\begin{definition}[Rigid quasiflat]
	\label{def_rigid}
	Suppose $X$ is a metric space satisfying condition {\rm (CI$_{n-1}$)}. Let $\Phi:\R^n\to X$ be an $(L,A)$-quasiflat. Let $\mathcal{C}_R$ and $\iota$ be as in Proposition~\ref{prop_chain_map}, 
	and let $b>0$.
	
	We define $\Phi$ to be ($\mu,b$)-{\em rigid}, if for every constant $M>0$ there exists a sublinear function $\mu=\mu_M:[0,\infty)\to[0,\infty)$ with the following property.
	Let $x\in\R^n$ and $\Phi(x)=p$. Suppose that $\varphi\in \bZ_{n-1,c}(\B{x}{r})$ is a cubical cycle (with respect to $\mathcal{C}_R$) with $\M(\varphi)\leq M\cdot r^{n-1}$.
	Suppose $\tau\in \I_{n,c}(\B{p}{Mr})$ satisfies  $\partial\tau=\iota(\varphi)$ and
	$\M(\tau)\leq M\cdot r^n$. Let $\nu\in \I_{n,c}(\R^n)$ be the canonical filling of $\varphi$ with $\spt(\nu)=W$. Let $W_b=\{y\in W\mid d(y,\D W)>b\}$. Then 
	$\Phi(W_b)\subset N_{\mu(r)}(\spt(\tau))$. A quasiflat is \emph{rigid} if it is $(\mu,b)$-rigid for some choice of $\mu$ and $b$.
	
	We define $\Phi$ to be {\em pointed ($\mu,b$)-\em rigid}, if the previous paragraph holds only for a particular base point $x$.
\end{definition}

\begin{remark}
	In the definition of ($\mu,b$)-rigid, the parameters $\mu$ and $b$ are independent of $x\in \R^n$. However, in the pointed version, different $x\in \R^n$ give rise to different $\mu$ and $b$.
\end{remark}

\begin{remark}
	\label{rmk_constants}
	One can set up the above definition slightly differently by using three constants $M_1,M_2,M_3$ and requiring $\M(\vp)\le M_1\cdot r^{n-1}$, $\tau\in\I_{n,c}(\B{p}{M_2r})$ and $\M(\tau)\le M_3\cdot r^n$. 
	However, one readily sees that this leads to an equivalent definition.
\end{remark}
\begin{remark}
	\label{rmk_scale}
	Note that the definition of rigid quasiflat depends on the choice of  the chain map $\iota$ and its underlying cubulation and therefore also on $X$. 
	However, for a different choice of $\iota$, $\Phi$ will be $(\mu',b')$-rigid for a different  choice of $\mu'$ and $b'$.
\end{remark}

\begin{remark}
	\label{rmk_Lispchitz}
	Suppose in addition that $X$ has a Lipschitz combing. Then we can approximate $\Phi$ such that it is $L'$-Lipschitz and an $(L',A')$-quasi-isometric embedding, 
	with $L',A'$ depending only on $L,A$ and $n$. In this case, we can define $\mu$-rigid by taking $\varphi$ to be any element in $\bZ_{n-1,c}(\B{x}{r})$ (not necessarily cubical), 
	and using $\Phi_\#$ instead of $\iota$. This gives an equivalent definition, with possibly different $\mu$.
\end{remark}

\begin{remark}
	\label{rmk_rigid quasidisks}
	We can define $(\mu,b)$-rigid for quasidisks in the same way (see Remark~\ref{rmk:quasidisk}). Of course, every quasidisk is $(\mu,b)$-rigid for some choice of $\mu$ and $b$, 
	however, if one fixes $\mu$ and $b$, then it places a non-trivial geometric condition on large quasidisks.
\end{remark}

Now we introduce a related property of quasiflats called {\em super-Euclidean divergence}. 
We will supply two equivalent versions of super-Euclidean divergence, one is technically more convenient (Definition~\ref{def_divergence1}), and one is closer to the definition
of super-Euclidean divergence for quasi-geodesics in the literature (Definition~\ref{def_divergence2}).

\begin{definition}[Super-Euclidean divergence I]
	\label{def_divergence1}
	Let $X,\Phi$ and $\iota$ be as in Definition~\ref{def_rigid}.	$\Phi$ has ($\delta$-){\em super-Euclidean divergence}, 
	if for any $D>1$, there exists a function $\delta=\delta_{D}:[0,\infty)\to[0,\infty)$ with $\lim\limits_{r\to\infty}\delta(r)=+\infty$ such that the following property holds.
	Let $r>r'$ where $r'$ is a positive constant depending only on $L,A,n$ and $D$ \footnote{We can choose $r'=2LD(A+a)$ with $a$ being the constant in Proposition~\ref{prop_chain_map}.}. Let $x\in\R^n$ be arbitrary and let $\Phi(x)=p$. Suppose that $\varphi\in\bZ_{n-1,c}(A_x(\frac{r}{D},Dr))$ is a cubical chain representing a nontrivial homology class in $\tilde{H}_{n-1}(\R^n\setminus\B{x}{\frac{r}{D}})$. 
	Suppose $\M(\varphi)\leq D\cdot r^{n-1}$. Then for any $\tau\in\I_{n,c}(A_p(\frac{r}{2LD},2LDr))$ such that $\partial\tau=\iota(\varphi)$, we have $\M(\tau)\geq\delta(r)\cdot r^n$.
\end{definition}

\begin{remark}
	Similar to Remark~\ref{rmk_constants}, one can set up the above definition using several different constants $D_1,D_2,\cdots$ at different places instead of $D$, which will lead to an equivalent definition. Moreover, we can repeat Remark~\ref{rmk_scale}, Remark~\ref{rmk_Lispchitz} and Remark~\ref{rmk_rigid quasidisks} in the context of super-Euclidean divergence.
\end{remark}

\begin{definition}[Super-Euclidean divergence II]
	\label{def_divergence2}
	Let $X,\Phi$ and $\iota$ be as in Definition~\ref{def_rigid}.	$\Phi$ has ($\delta$-){\em super-Euclidean divergence}, 
	if for any $D>1$, there exists a function $\delta=\delta_{D}:[0,\infty)\to[0,\infty)$ with $\lim\limits_{r\to\infty}\delta(r)=+\infty$ such that the following property holds.
	Let $r>0$. Let $x\in\R^n$ be arbitrary and let $\Phi(x)=p$. Suppose that $\varphi\in\bZ_{n-1,c}(\R^n\setminus\B{x}{r})$ be a cubical chain representing a nontrivial homology class in $\tilde{H}_{n-1}(\R^n\setminus\B{x}{r})$. 
	Suppose $\M(\varphi)\leq D\cdot r^{n-1}$. Then for any $\tau\in\I_{n,c}(X\setminus\B{p}{\frac{r}{2LD}})$ such that $\partial\tau=\iota(\varphi)$, we have $\M(\tau)\geq\delta(r)\cdot r^n$ (if no such $\varphi$ or $\tau$ exists, then it is understood that this inequality automatically holds).
\end{definition}

\begin{lem}
	\label{lem:rigid vs divergence}
	Suppose $X$ satisfies condition~{\rm (CI$_{n-1}$)}. Then an $(L,A)$-quasiflat $\Phi:\R^n\to X$ satisfies Definition~\ref{def_divergence1} if and only if it satisfies Definition~\ref{def_divergence2} (for a possibly different $\delta$).
\end{lem}

\begin{proof}
	It is clear that Definition~\ref{def_divergence2} implies Definition~\ref{def_divergence1}. Now suppose Definition~\ref{def_divergence2} fails for $\Phi$. Then we find a constant $D_0>1$, a sequence $(x_k)$ in $\R^n$ and a sequence $r_k\to\infty$ such that the following holds.
	There exist cubical cycles $\varphi_k\in\bZ_{n-1,c}(\R^n\setminus\B{x_k}{r_k})$ with $[\varphi_k]\neq 0\in \tilde{H}_{n-1}(\R^n\setminus\B{x_k}{r_k})$, $\M(\varphi_k)\leq D_0\cdot r_k^{n-1}$ such that $\iota(\varphi_k)$ can be filled by $\tau_k\in\I_{n,c}(X\setminus \B{p_k}{\frac{r_k}{D_0}})$ with $\M(\tau_k)\leq D_0\cdot r_k^{n}$ (define $p_k=\Phi(x_k)$). 
	
	We claim we can assume in addition that $\spt(\vp_k)\subset A_{x_k}(r_k,2r_k)$. Otherwise for each $\vp_k$, we find a cubical chain $\vp'_k\in \bZ_{n-1,c}(A_{x_k}(r_k,2r_k))$ and $H_k\in\I_{n,c}(\R^n\setminus \B{x_k}{r_k})$ such that $\D H_k=\vp_k-\vp'_k$, $\M(\vp'_k)\le D_0\cdot r^{n-1}_k$ and $\M(H_k)\le D_1\cdot r^{n-1}_k$ (here $D_1$ is a constant independent of $k$), and then we replace $\vp_k$ by $\vp'_k$ and $\tau_k$ by $\tau_k+\iota(H_k)$.
	
	By the previous claim, we can assume that $\spt(\D\tau_k)=\spt(\iota(\vp_k))\subset A_{p_k}(\frac{r_k}{D_0},r_kD_0)$ up to choosing a larger $D_0$. Take $b>1$ whose value will be determined later. We use coarea inequality to find a slice $\si_k=\slc{\tau_k,d_{p_k},r'_k}$ with $10br_kD_0<r'_k<20br_kD_0$ such that $\M(\si_k)\le \frac{1}{10b}\cdot r^{n-1}_k$. Let $\tau'_k$ be a minimal filling of $\si_k$. It follows from Theorem~\ref{thm:isop-ineq} that 
	$\spt(\tau_k')\subset N_{b_1r_k}(\spt(\si_k))$ and $\M(\tau'_k)\le b_2\cdot r^{n}_k$ where $b_1$ and $b_2$ depend only on $b$ and $X$. By choosing $b$ sufficiently large, $b_1$ will be small enough so that $\spt(\tau_k')\cap \B{p_k}{\frac{r_k}{D_0}}=\emptyset$. Now we replace $\tau_k$ by $\tau''_k=\tau_k\on \{d_{p_k}\le r'_k\}+\tau'_k$. Note that $\D\tau''_k=\D\tau_k$. The existence of $\{\tau''_k\}$ and $\{\vp_k\}$ contradicts Definition~\ref{def_divergence1}.
\end{proof}

\begin{lemma}\label{lem_rig_iff_supdiv}
	Let $\Phi:\R^n\to X$ be an $n$-dimensional $(L,A)$-quasiflat in a metric space $X$ satisfying condition~{\rm (CI$_{n-1}$)}. Then $\Phi$ is rigid if and only if it has super-Euclidean divergence.
\end{lemma}

\begin{proof}
	We will use Definition~\ref{def_divergence1} in the proof.
	
	If super-Euclidean divergence fails, then we find a constant $D_0>1$, a sequence $(x_k)$ in $\R^n$ and a sequence $r_k\to\infty$ such that the following holds.
	There exist cubical cycles $\varphi_k\in\bZ_{n-1,c}(A_{x_k}(\frac{r_k}{D_0},D_0r_k))$ with $[\varphi_k]\neq 0\in \tilde{H}_{n-1}(\R^n\setminus \B{x_k}{\frac{r_k}{D_0}})$, $\M(\varphi_k)\leq D_0\cdot r_k^{n-1}$ such that $\iota(\varphi_k)$ can be filled by $\tau_k\in\I_{n,c}(A_{p_k}(\frac{r_k}{D_0},D_0r_k))$ with $\M(\tau_k)\leq D_0\cdot r_k^{n}$ (define $p_k=\Phi(x_k)$). Let $\alpha_k\in \I_{n,c}(\R^n)$ be the canonical filling of $\varphi_k$. Then $\M(\alpha_k)\le D'\cdot r^n_k$ with $D'$ depending only on $D_0$ and $n$. Moreover, $\B{x_k}{\frac{r_k}{D_0}}\subset\spt(\al_k)$. Let $b>0$ be arbitrary. Then there exists $D''=D''(L,A,,b,D_0)$ such that $\Phi(\B{x_k}{\frac{r_k}{D_0}-b})$ contains a point which is $D'' r_k$ away from $A_{p_k}(\frac{r_k}{D_0},D_0r_k)\supset \spt(\tau_k)$ for all $k$. Thus $Q$ is not rigid.

	Now we argue the other direction. If $\Phi$ is not rigid, then for any choice of $b>0$ (we will determine the value of $b$ later), we find (depending on value of $b$) constants $\eps_0>0$, $M_0$, a sequence $(x_k)$ in $\R^n$, $r_k\to\infty$, cubical chains $\varphi_k\in\I_{n-1,c}(\B{x_k}{r_k})$ and $\tau_k\in\I_{n,c}(\B{p_k}{M_0 r_k})$ with $\D\tau_k=\iota(\varphi_k)$ such that the following holds for all $k$.
	\begin{enumerate}
		\item $\M(\varphi_k)\leq M_0\cdot r_k^{n-1}$ and $\M(\tau_k)\leq M_0\cdot r_k^n$.
		\item Let $\mu_k$ be the canonical filling of $\varphi_k$. Then there is $y_k\in \{y\in\spt(\mu_k)\mid d(y,\spt(\varphi_k))>b\}$ with $\B{\Phi(y_k)}{\eps_0 r_k}\cap\spt(\tau_k)=\emptyset$.
	\end{enumerate}
	In particular,
	\begin{equation}
	\label{eq_deep_inside}
	\B{\Phi(y_k)}{\eps_0 r_k}\cap\spt(\iota(\varphi_k))=\emptyset.
	\end{equation}

	Let $\varphi'_k=\pi_\# \iota(\varphi_k)$, where $\pi:X\to\R^n$ is defined in Section~\ref{subsec_cubulated_quasiflats}. As $\spt(\iota(\varphi_k))$ is contained in a fixed neighborhood of $Q$, by applying $\pi$ to \eqref{eq_deep_inside}, we can find $\eps_1$ independent of $k$ with $\B{y_k}{\eps_1 r_k})\cap\spt(\varphi'_k)=\emptyset$. 
	Thus 
	\begin{equation}
	\label{eq_deep_inside1}
	 d(y_k,\spt(\vp'_k))\ge \eps_1 r_k.
	\end{equation} 
	
	Suppose $\nu'_k$ is the canonical filling of $\varphi'_k$.	
We now claim that choosing $b$ large enough will guarantee $y_k\in \spt(\nu'_k)$. By Proposition~\ref{prop_chain_map} (5), there is $h\in\I_n(\R^n)$ such that $\D h=\varphi'_k-\varphi_k$ and $\spt(h)\subset N_a(\spt(\varphi_k))$ with the constant $a$ as in Proposition~\ref{prop_chain_map} (5). By assumption (2) above, $\B{y_k}{b}\subset \spt(\mu_k)$. As $\nu'_k=\mu_k+h$, if we choose $b>a$, then $y_k\in \spt(\nu'_k)$, thus the claim follows.

	By applying Federer-Fleming deformation to $\vp'_k$, we assume both $\vp'_k$ and $\nu'_k$ are cubical and the previous estimates still hold. Then $\M(\vp'_k)\le M'_0\cdot r^{n-1}_k$. By choosing a slightly larger $b$, we assume $y_k\in \spt(\nu'_k)$ still holds. Thus by \eqref{eq_deep_inside1}, we have $\B{y_k}{\eps_1 r_k}\subset\spt(\nu'_k)$.
	It follows that $\varphi'_k$ represents a non-trivial class in $\tilde H_{n-1}(\R^n\setminus \B{y_k}{\eps_1 r_k})$.

	In what follows, $a_1,a_2,\cdots$ will be constants depending only on $L,A,n$ and $X$ (independent of $k$). Let $q_k=\Phi(y_k)$. Proposition~\ref{prop_chain_map} (5) implies that $\Fill(\vp_k-\vp'_k)\le a_1\cdot\M(\vp_k)$. Thus $\Fill(\iota(\vp_k)-\iota(\vp'_k))\le a_2\cdot\M(\vp_k)$ by Proposition~\ref{prop_chain_map} (2). Let $\hat\tau_k$ be a filling of $\iota(\vp_k)-\iota(\vp'_k)$ given in Lemma~\ref{lem:density0}. By Lemma~\ref{lem:density0} (1), 
\begin{equation}
\label{eqn_upper_bound_order_n-1}
\M(\hat\tau_k)\le 2a_2\cdot\M(\vp_k)\le 2a_2M_0\cdot r^{n-1}_k\,.
\end{equation}

 By Lemma~\ref{lem:density0} (2), for every $x\in \spt(\hat\tau_k)$ and every $0\le r\le d(x,\spt(\iota(\vp_k)-\iota(\vp'_k)))$, we have 
\begin{equation}
\label{eqn_lower_mass_bound_tau_hat}	
\M(\hat\tau_k)\ge\|\hat\tau_k\|(\B{x}{r}) \ge Dr^n\,.
\end{equation}
	Then combining \eqref{eqn_upper_bound_order_n-1} and \eqref{eqn_lower_mass_bound_tau_hat}, it follows that $\spt(\hat\tau_k)$ is contained in the $a_3r^{\frac{n-1}{n}}_k$-neighborhood of $$\spt(\iota(\vp_k)-\iota(\vp'_k))\subset \spt(\iota(\vp'_k))\cup\spt(\iota(\vp_k)).$$ 
	Note that $\spt(\iota(\vp'_k))\subset N_{a_4}(\Phi(\spt(\vp'_k)))$ by Proposition~\ref{prop_chain_map} (1). This together with \eqref{eq_deep_inside} and \eqref{eq_deep_inside1} imply that $$d(q_k,\spt (\hat\tau_k))\ge a_5\eps_1r_k-a_6r^{\frac{n-1}{n}}_k\ge a_7\eps_1 r_k$$ when $r_k$ is sufficiently large. Let $\tau'_k=\tau_k-\hat\tau_k$. Then $\D\tau'_k=\iota(\vp'_k)$ and $d(q_k,\spt (\tau'_k))\ge a_7\eps_1r_k$. Thus $\spt(\tau'_k)\subset A_{q_k}(a_7\eps_1 r_k,(M_0+a_8)r_k)$. However, $\M(\tau'_k)\le \M(\tau_k)+\M(\hat\tau_k)\le (M_0+a_7)\cdot r^n_k$. Thus super-Euclidean divergence
	fails by considering the sequence $\varphi'_k$ and $\partial\tau'_k=\iota(\varphi'_k)$ (recall that $\varphi'_k$ represents a non-trivial class in $\tilde H_{n-1}(\R^n\setminus \B{y_k}{\eps_1 r_k})$).  
\end{proof}
 
\begin{definition}
	\label{def:Morse}
	Suppose $X$ is a complete metric space satisfying condition {\rm (CI$_{n-1}$)}. An $n$-dimensional quasiflat $Q\subset X$ is \emph{Morse} if it has $\de$-super-Euclidean divergence. The \emph{Morse parameter} of $Q$ is $\de$.
\end{definition}

\subsection{Quasi-isometry invariance}
\label{subsec_QI_invariance}
\begin{proposition}
	\label{prop:QI invariance}
	Let $X$ and $Y$ be two complete metric spaces satisfying condition~{\rm (CI$_{n-1}$)} with constant $c$. 
	Let $Q\subset X$ be an $n$-dimensional Morse $(L,A)$-quasiflat. Let $\Psi:X\to Y$ be an $L'$-Lipschitz $(L',A')$-quasi-isometry with an $L'$-Lipschitz quasi-isometry inverse. 
	Then $\Psi(Q)$ is a Morse quasiflat with its Morse parameter depending only on $L,A,L',$ $A',n,c$ and the Morse parameter of $Q$.
\end{proposition}

\begin{proof}
	By Lemma~\ref{lem:rigid vs divergence}, it suffices to show super-Euclidean divergence is invariant under quasi-isometry, which is slightly cleaner to work with. Suppose $Q$ is represented by $\Phi:\R^n\to X$ such that $\R^n$ has a cubulation of appropriate scale $R=R(L,A,L',A')$. Let $Q'=\Psi(Q)$ be represented by $\Phi'=\Psi\circ\Phi$. Let $\iota$ and $\iota'$ be chain maps as in Proposition~\ref{prop_chain_map} associated with $Q$ and $Q'$ respectively.
	
	Let $\Psi'$ be an $L'$-Lipschitz quasi-isometry inverse of $\Psi$. Suppose $Q$ has $\de$-super-Euclidean divergence. 
	
	Take $D>1$ and $x\in\R^n$. Define $\Phi'(x)=p'$ and $\Phi(x)=p$. Suppose that $\varphi\in\bZ_{n-1,c}(\R^n\setminus\B{x}{r})$ is a cubical chain representing a nontrivial class in $\tilde{H}_{n-1}(\R^n\setminus\B{x}{r})$. 
	Suppose $\M(\varphi)\leq D\cdot r^{n-1}$. Let $\tau'\in\I_{n,c}(Y\setminus\B{p'}{\frac{r}{D}})$ such that $\partial\tau'=\iota'(\varphi)$. 
	
	Let $\tau=\Psi'_\#(\tau')$. Using Theorem~\ref{thm:isop-ineq}, we can build  a homology $h$  between $\iota(\varphi)$ and $\Psi'_\#(\iota'(\varphi))=\D\tau$
	skeleton by skeleton such that $\D h=\iota(\varphi)-\D\tau$, $\M(h)\le C\cdot \M(\varphi)$ and $\spt(h)\subset N_M(\spt (\iota(\varphi))\cup \spt(\D\tau))$ where $C$ and $M$ depend only on $L,A,L',A',c$ and $n$.
	
	Let $\tau_1=\tau+h$. Then $\spt(\tau_1)\in X\setminus\B{p}{\frac{r}{C'D}}$ for $C'>1$ depending only on $C,L',A',L,A$. As $Q$ has $\de$-super-Euclidean divergence, $\M(\tau_1)\ge \delta(r) \cdot r^n$. Thus $\M(\tau')\ge \frac{1}{(L')^n}\M(\tau)\ge\frac{\delta(r)\cdot r^n-C\cdot \M(\varphi)}{(L')^n}\ge \frac{\delta(r)\cdot r^n-CD\cdot r^{n-1}}{(L')^n}$. Thus $Q'$ has $\de'$-super-Euclidean divergence with $\de'=\de'(\de,L,A,L',A',c,n)$.
\end{proof}

\begin{cor}
	\label{cor:qi invariance}
	Suppose $X$ and $Y$ are piecewise Euclidean simplicial complexes. 
	Let $Q\subset X$ be an $n$-dimensional Morse $(L,A)$-quasiflat. 
	Suppose there exists $f:(0,\infty)\to (0,\infty)$ such that any $\ell$-Lipschitz map from the unit sphere $\mathbb S^m$ to $X$ (resp. $Y$) with $m\le n$ can be extended to an $f(\ell)$-Lipschitz map 
	from the unit disk to $X$ (resp. $Y$). Let $\Psi:X\to Y$ be an $(L',A')$-quasi-isometry with an $L'$-Lipschitz quasi-isometry inverse. Then $\Psi(Q)$ is a Morse quasiflat with its Morse parameter depending only on $L,A,L',$ $A',n,f$ and the Morse parameter of $Q$.
\end{cor}

The assumption of Corollary~\ref{cor:qi invariance} enables us to build Lipschitz quasi-isometries $X^{(n)}\to Y^{(n)}$ and $Y^{(n)}\to X^{(n)}$ skeleton by skeleton. 

If $X$ and $Y$ are simplicial complexes with cocompact automorphism groups satisfying condition~{\rm (CI$_{n-1}$)}, 
then the assumption of Corollary~\ref{cor:qi invariance} is satisfied. Indeed, condition~{\rm (CI$_{n-1}$)} implies that $k$-th reduced homology of $X$ and $Y$ are trivial for $k\le n-1$ (by Federer-Fleming deformation), 
thus the $k$-th homotopy group is also trivial for $k\le n-1$. The cocompactness guarantees that spheres can be filled with uniform control.

\section{Asymptotic conditions on spaces and conditions on asymptotic cones}
\label{sec:cone and space}
In this section we relate the conditions in asymptotic cones, discussed in Section~\ref{subsec_cone_conditions}, to the
asymptotic conditions within a space as discussed in Section~\ref{subsec_asymptotic condition}. The main result is Proposition~\ref{prop_bridge}
which provides a connection in an appropriate setting, for instance for spaces with Lipschitz combings.

\subsection{Full support and super-Euclidean divergence}
\label{subsec_full_support_and_divergence}
Throughout this section, we let $Q$ be a quasiflat represented by a map $\Phi:\R^n\to X$ and we let  $\iota$ be an associated chain map as in Proposition~\ref{prop_chain_map}.
\begin{lemma}\label{lem_suplindiv_implies_fullsup}
	Let $X$ be a complete metric space which satisfies {\rm (SCI$_{n-1}$)} with constant $c$. Let $\Phi:\R^n\to X$ be an $(L,A)$-quasiflat with image $Q$.
	Suppose that $Q$ has $\delta$-super-Euclidean divergence. 
	Then for any asymptotic cone $X_\om$, the ultralimit $Q_\om$ of $Q$ has full support in $X_\om$ with respect to $H^{L}_{*}$.
\end{lemma}

\begin{proof}
	Let $(X_\om,p_\om)=\om\lim(\frac{1}{r_k}X,p_k)$ for some sequence $r_k\to\infty$. Set $X_k=\frac{1}{r_k}X$.
	Suppose that $Q_\om$ does not have full support with respect to $H^{L}_{*}$. Denote by $\Phi_\om:\R^n\to X_\om$ the ultralimit of $\Phi$. There exists a point $q_\om=\Phi_\om(x_\om)$ which does not lie in the support of
	$Q_\om$. After translating $\R^n$, we may assume $x_k\equiv 0$. Then we find $D>1$ and a Lipschitz cycle $\varphi$ representing an element in  $\bZ_{c,n-1}(A_{0}(\frac{1}{D},D))$ in $\R^n$ and 
	a Lipschitz chain $\tau$ representing an element in 
	$\I_{c,n}(A_{q_\om}(\frac{1}{D},D))$ that fills $\Phi_{\om\#}\varphi$, and such that 
	$\M(\tau)+\M(\varphi)<D$. Moreover, $\varphi$ represents a nontrivial element in $\tilde{H}_{n-1}(\R^n\setminus B_{0}(\frac{1}{D}))$. Let  $R=R(D,c)$ be a fine scale to be determined later.
	
	By barycentric subdivision, we can write $\tau$ as a sum of Lipschitz simplices with uniform Lipschitz constant and such that each of them has image in a ball of radius $R$.
	Using this representation, we pull-back the $0$-skeleton of $\tau$ to a collection of points $\tau_k^{(0)}$ in $X_k$, such that each pair of points in $\tau_k^{(0)}$ corresponding to the bounardy of an edge of $\tau$ 
	is contained in a ball of radius $R$.
	Then we fill in skeleton by skeleton, using the strong coning inequalty, to obtain $\tau_k\in\I_{c,n}(X_k)$.
	Note that there is a constant $C=C(D,c,n)$ such that 
	$\M(\tau_k)\leq C\cdot\M(\tau)$ and $\spt(\tau_k)\subset N_{CR}(\tau_k^{(0)})$.
	After further subdivision, we may assume that $\D\tau_k=\iota \varphi_k$, where  $\varphi_k\in\bZ_{c,n-1}(A_{x_k}(\frac{1}{D},D))$ is a cubical cycle homologous to $\varphi$ in $A_{0}(\frac{1}{D},D)$. 
	Now it is clear that if we choose $R$ small enough, then the $\tau_k$ will stay at a uniform positive distance from $p_k$.
	Therefore   $Q$ cannot have super-Euclidean divergence in $X$.
\end{proof}

In order to gain information on our space from properties of its asymptotic cones, we will rely on the following variation of
Wenger's compactness theorem, \cite{wenger2011compactness}.

\begin{lemma}\label{lem_compactness}
	Let $(X_k)$ be a sequence of complete metric spaces satisfying condition~{\rm (EII$_{n-1}$)} with uniform constant. Further, let $(\tau_k)$ be a sequence of chains
	in $\I_n(X_k)$ such that there exists a constant $D>0$ with
	\[\M(\tau_k)+\M(\D\tau_k)<D \text{ and }\diam\spt(\tau_k)<D.\]
	Then, for every $\epsilon>0$ there exists a regularized sequence $(\tau_k')$ in $\I_{c,n}(X_k)$ such that
	\[\M(\tau'_k)+\M(\D\tau'_k)<D \text{ and }\diam\spt(\tau'_k)<D;\]
	the family of support sets $(\spt(\tau_k'))$ is uniformly totally bounded;
	and $\tau_k$ is homologous to $\tau_k'$ inside $N_\eps(\spt(\tau_k))$ (if the support sets $(\spt (\D\tau_k))$ is already uniformly bounded, then we can assume in addition that $\D\tau'_k=\D\tau_k$). Moreover, there exists a subsequence $(\tau'_{k_l})$, a compact metric space $Z$, a chain $\tau'\in\I_{c,n}(Z)$ 
	and isometric embeddings $\psi_l:\spt(\tau'_{k_l})\to Z$ such that $\psi_{l\#}\tau'_{k_l}$ converges to $\tau'$ in the flat distance in $\ell^\infty(Z)$.
\end{lemma}

\begin{proof}
	For the first half of the claim, we can first apply Proposition~\ref{prop_approx} to $\D\tau_k$ to obtain  sequence $(\D\tau'_k)$ with uniformly totally bounded supports such that $\D\tau'_k-\D\tau_k=\D H_k$ for some $H_k$ supported in $N_{\frac{\eps}{2}}(\spt(\D\tau_k))$. Then we apply Proposition~\ref{prop_approx} to $\tau_k+H_k$ to obtain the desired sequence. Note that if the sequence $\spt(\D\tau_k)$ is already uniformly totally bounded, then we do not need the first step.
	 Now we have a sequence  $(\tau_k')$ with $\tau_k'\in\I_{c,n}(X_k)$ such that
	$\M(\tau'_k)+\M(\D\tau'_k)<D$, the family of support sets $(\spt(\tau_k'))$ is uniformly totally bounded and
	$\tau_k$ is homologous to $\tau_k'$ inside $N_\eps(\spt(\tau_k))$. By \cite{gromov1981groups},	there exists a compact metric space $Z$ with a sequence of isomeric embeddings 
	$\psi_l:\spt(\tau'_{k_l})\to Z$ such that $\psi_l(\spt(\tau'_{k_l}))$ converges to $Y\subset Z$ in the Hausdorff distance.
	By \cite[Theorem 5.2 and Theorem 8.5]{ambrosio2000currents}, we can assume $\psi_{l\#}\tau'_{k_l}$ weakly converges to $\tau'\in \I_{c,n}(Z)$ by passing to a subsequence. 
	Moreover, by \cite[Proposition 2.2]{wenger2011compactness},  $\spt(\tau')\subset Y$. 
	Finally, by \cite{wenger2007flat}, $\psi_{l\#}\tau'_{k_l}$ converges to $\tau'$ in flat distance, if we replace $Z$ by $\ell^\infty(Z)$.
\end{proof}

\begin{lemma}\label{lem_fullsup_implies_suplindiv}
	Let $X$ be a complete metric space which satisfies {\rm (EII$_{n-1}$)} with constant $c$. Let $\Phi:\R^n\to X$ be an $(L,A)$-quasiflat with image $Q$.
	Suppose that  in any asymptotic cone $X_\om$ the ultralimit $Q_\om$ of $Q$  has full support with respect to $H^{AK}_{*,c}$.
	Then $Q$ has $\delta$-super-Euclidean divergence in $X$. 
\end{lemma}

\begin{proof}
We choose a cubulation $\mathcal{C}$ of $\R^n$ and a chain map $\iota$ associated to $\Phi$, as provided by Proposition~\ref{prop_chain_map}.
	Suppose that $Q$ does not have super-Euclidean divergence.
	Then there exist 
	\begin{itemize}
		\item a constant $D>1$;
		\item a sequence $r_k\to\infty$;
		\item a sequence of cycles $\varphi_k\in\bZ_{n-1,c}(A_{x_k}(\frac{r_k}{D},D r_k))$, cubical with respect to $\mathcal{C}$, and such that  $x_k\in\R^n$ and $\M(\varphi_k)\leq D\cdot r_k^{n-1}$; 
		\item a sequence of chains $\tau_k\in\I_{n,c}(A_{p_k}(\frac{r_k}{2LD},2LD r_k))$ with $p_k=\Phi(x_k)$ and $\M(\tau_k)\leq D\cdot r_k^n$ such that $\D \tau_k=\iota\varphi_k$.
	\end{itemize}

	Up to a translation on $\R^n$, we assume $x_k\equiv 0$.
	Set $X_k=\frac{1}{r_k} X$ and also rescale $\R^n$ by $\frac{1}{r_k}$. However, we still use the notation $\tau_k$ and $\varphi_k$ 
	to denote the corresponding currents in the rescaled spaces.  After modifying $\varphi_k$ by a homology inside $A_0(\frac{1}{D},D)$, 
	we may also assume that all the $\varphi_k$	are equal to a single cubical cycle $\varphi$.
By Lemma~\ref{lem_compactness}, after passing to a subsequence, there exists a regularized sequence $\tau'_k$ with $\D\tau'_k=\iota_k\varphi$, a compact metric space $Z$ with a sequence of isomeric embeddings $\psi_k:\spt(\tau'_{k})\to Z$ such that $\psi_{k\#}\tau'_{k}$ converges to some $\tau_\infty'\in\I_{n,c}(Z)$ with respect to flat distance in $\ell^\infty(Z)$.

	Let $(X_\omega,p_\omega)$ be an ultralimit of $(X_k,p_k)$ and $Q_\om\subset X_\om$ be the ultralimit of the $Q_k$'s. By the previous discussion, 
	$\spt(\tau_\infty')$ with the induced metric from $Z$ embeds isometrically via $\psi_\infty^{-1}$ 
into $X_\omega$ with $\psi_\infty^{-1}(\spt(\D\tau_\infty'))\subset Q_\om\cap A_{p_\om}(\frac{1}{2LD},2LD)$ where $\psi_\infty^{-1}$ is a partial limit of the $\psi_k^{-1}$.
	Hence, in order to arrive at a contradiction, it is enough to show that the support of the canonical filling of $(\psi_\infty^{-1})_\#\D\tau_\infty'$ in $Q_\om$
	contains the point $p_\om$.
	Let $\pi_k:X_k\to\R^n$ be a Lipschitz extension of 
	$\frac{1}{r_k}(\Phi|_{\mathcal{C}^{(0)}})^{-1}$ where $\mathcal{C}$ denotes the cubulation underlying $\iota$. Then the Lipschitz constant of $\pi_k$ is uniformly bounded.
	Denote by $\hat\pi_k:Z\to\R^n$ a Lipschitz extension of $\pi_k\circ\psi_k^{-1}$ and suppose that $(\hat\pi_k)$ converges uniformly to a limit map $\hat\pi_\infty:Z\to\R^n$.
	Then $(\hat\pi_k)_\# \iota_k\varphi$ converges to $(\hat\pi_\infty)_\#\D\tau_\infty'$ since the masses of the $\iota_k\varphi$ are uniformly bounded and 
	$\iota_k\varphi=\D\tau_k'$ converges to $\D\tau_\infty'$. By Proposition~\ref{prop_chain_map}, the cycle $(\hat\pi_k)_\# \iota_k\varphi$
	is homologous to $\varphi$ in the $\frac{a}{r_k}$-neighborhood of $\spt(\varphi)$. Since $\varphi$ is homologically non-trivial in $A_0(\frac{1}{D},D)$,
	we conclude that for large enough $k$, the canonical filling of  $(\hat\pi_k)_\# \iota_k\varphi$ contains a multiple of the fundamental class $\bb{B_0(\frac{1}{2D})}$
	as a piece. Hence the same holds true for the limit $(\hat\pi_\infty)_\#\D\tau_\infty'$.
	By construction, $\hat\pi_\infty$ factors as $\Phi_\om^{-1}\circ\psi_\infty^{-1}$ on $\spt(\D\tau'_\infty)$.
	Since $\Phi_\om$ and $\psi_\infty^{-1}$ are bilipschitz (on $\spt(\D\tau'_\infty)$), we have 
	$(\hat\pi_\infty)_\#\D\tau_\infty'=(\Phi_\om^{-1}\circ\psi_\infty^{-1})_\#\D\tau_\infty'=(\Phi_\om^{-1})_\#((\psi_\infty^{-1})_\#\D\tau_\infty')$. 
	In particular, $(\Phi_\om)_\#(\hat\pi_\infty)_\#\D\tau_\infty'=((\psi_\infty^{-1})_\#\D\tau_\infty')$, and $\Phi_\om$ sends $0$ to $p_\om$, we conclude that $p_\om$ lies in the support of the canonical filling of $(\psi_\infty^{-1})_\#\D\tau_\infty'$ in $Q_\om$
	as required.
\end{proof}

Let $\mathcal{X}$ be a family of complete metric spaces.
Let $\mathcal{Q}$ be a family of $n$-dimensional quasiflats or quasidisks with uniform quasi-isometry constants such that each member of $\mathcal{Q}$ lies inside some member of $\mathcal{X}$. We say $\mathcal{Q}$ has \emph{asymptotic property $P$} if each ultralimit $Q_\om$ of a sequence of elements in $\mathcal{Q}$  has property $P$ in the associated asymptotic cone.

\begin{proposition}\label{prop_bridge}
	Let $\mathcal{X}$ be a collection of complete metric spaces which satisfies {\rm (SCI$_{n-1}$)} with uniform constants. Let $\mathcal{Q}$ be a family of $n$-dimensional quasiflats or quasidisks in these metric spaces with uniform quasi-isometry constants.
	Suppose that  in any asymptotic cone $X_\om$, which contains an ultralimit $Q_\om$ of elements in $\mathcal{Q}$, holds that the homology theories $\tilde{H}^{AK}_{*,c}$ and $\tilde{H}^L_{*}$ are isomorphic.
	Then each element in $\mathcal{Q}$ has $\delta$-super-Euclidean divergence for some uniform $\delta$ if and only if $\mathcal{Q}$ has asymptotic full support property with respect to $H^L_{*}$ in any asymptotic cone.
	
	As a consequence, if $\mathcal{X}$ is a family of metric spaces with Lipschitz combings (with uniform constants), then each element in $\mathcal{Q}$ has $\delta$-super-Euclidean divergence for some uniform $\delta$ if and only if $\mathcal{Q}$ has full support property with respect to $H_{*}$.
\end{proposition}

\begin{proof}
	The first paragraph follows from Lemma~\ref{lem_suplindiv_implies_fullsup} and Lemma~\ref{lem_fullsup_implies_suplindiv} (they work the same way for a family of quasiflats or quasidisks). The second paragraph follows from Proposition~\ref{prop_cone_conditions} and Remark~\ref{rmk_combing}.
\end{proof}

The following is a consequence of Proposition~\ref{prop_bridge} and a version of the K\"unneth formula for relative homology (cf. \cite[pp. 190, Proposition 2.6]{dold2012lectures}).
\begin{cor}
	\label{cor:products}
	Suppose $X_1$ and $X_2$ are metric spaces with Lipschitz combings. Let $Q_i\subset X_i$ be Morse quasiflats (cf. Definition~\ref{def:Morse}). Then $Q_1\times Q_2$ is a Morse quasiflat in $X_1\times X_2$.
\end{cor}

\section{Contracting properties for projection of cycles}
\label{sec_cycles_close_to_Morse_quasiflats}
In this section we introduce the coarse neck property and cycle contracting property, both serve as alternative characterizations of Morse quasiflats. The main goal of this section is to prove Proposition~\ref{prop_small_neck1}, which roughly says that $(\mu,b)$-rigidity implies the cycle contracting property. This is a key step in the proof of Theorem~\ref{thm:equivalence intro}. 

Readers who are willing to work under the stronger assumption of a Lipschitz bicombing, can skip this section and instead read Section~\ref{sec_short_cut} for a short cut. This leads to a weaker version of Theorem~\ref{thm:equivalence intro} without the cycle contracting property (see Theorem~\ref{thm_weak_equivalence}).

The quasi-isometry constants of the quasiflat $Q$ often add to complications in proofs. The reader might find it helpful to think about the relatively clean situation where $Q$ comes from a bilipschitz embedding $\mathbb R^n\to X$ with corresponding Lipschitz retraction $\pi:X\to \mathbb R^n$. 

\subsection{Coarse neck property and cycle contracting property}
\label{subsec_cnp_ccp}
\begin{definition}
	\label{def_tcp}
	An $(L,A)$-quasiflat $Q$ in a metric space $X$ has the \emph{coarse neck property} (CNP), if there exists a constant $C_0>0$ such that the following holds. For  given positive constants $C,\rho$, $b_1$ and $b_2$,  
	there exists $\ul R=\ul R(C,\rho,b_1,b_2)$ such that for any $p\in X$ with $d(p,Q)<b_2$, any $R\ge \ul R$ and any $\tau\in \I_n(B_p(b_1R)\setminus N_{\rho R}(Q))$ satisfying
	\bit
	\item $\M(\tau)\le C\cdot R^n$;
	\item $\sigma:=\partial\tau\in \I_{n-1}(X)$, $\spt(\si)\subset N_{2\rho R}(Q)$ and $\M(\si)\le C\cdot R^{n-1}$;
	\eit
	we have 
	\[\Fill(\si)\le C_0\cdot \rho R\cdot \M(\si).\] 
	Note that the definition of CNP depends on the parameter $C_0$ and the function $\ul R=\ul R(C,\rho,b_1,b_2)$.
\end{definition}

\begin{definition}
	\label{def_ccp}
	Let $X$ be a proper metric space satisfying condition~{\rm (CI$_{n-1}$)}.
	Let $Q\subset X$ be an $n$-dimensional $(L,A)$-quasiflat which is an $A_1$-homology retract in the sense of Definition~\ref{def_tame}. We say $Q$ has the \emph{cycle contracting property} (CCP) if the following holds.  Given $C,\rho,b_1>0$. Then for any $\eps>0$, there exists $\ul{R}=\ul{R}(C,\rho,b_1,\eps)$ such that the following holds for any $R>\ul{R}$. 
	
	Let $\si\in \bZ_{n-1,c}(X)$ such that 
	\bit
	\item $\M(\si)\le C \cdot R^{n-1}$, $\diam(\spt(\si))\le b_1\cdot R$ and $d(\spt(\si),Q)\le b_1\cdot R$;
	\item $\si$ admits a filling $\tau\in \I_{n,c}(X)$ such that $\M(\tau)\le C\cdot R^n$, $\spt(\tau)\subset N_{b_1R}(\spt(\si))$ and $\spt(\tau)\cap N_{\rho R}(Q)=\emptyset$.
	\eit
	Let $\mathcal{C}$ be the class of integral currents whose boundary is a sum of $\si$ and a current supported in $N_{A_1}(Q)$ and let $M$ be the infimum of masses of elements in $\mathcal{C}$. Then for any $\al\in \I_{n,c}(X)$ such that $\spt(\D\al-\si)\subset N_{2A_1}(Q)$ and $M(\al)\le M$, we have $$\Fill(\D\al-\si)\le \eps R\cdot \M(\si).$$
\end{definition}

\begin{remark}
	In Definition~\ref{def_ccp} we require $\spt(\D \alpha-\si)\subset N_{2A_1}(Q)$ rather than $\spt(\D \alpha-\si)\subset N_{A_1}(Q)$. The reason is of technical nature, namely a mass minimizing sequence in $\mathcal{C}$ might not sub-converge to an integral current. There are
	at least two alternative ways around this to set up Definition~\ref{def_ccp}. The first is to allow $\alpha$ to be a flat chain, thereby including limits of a minimizing sequence in $\mathcal C$. The second is to restrict to integral currents,  require $\spt(\D \alpha-\si)\subset N_{A_1}(Q)$, but relax the mass bound to $\M(\alpha)\le M+\de$ where $\de$ is a small number. Both alternatives add complications to the proof: the former leads to extra issues with flat chains, and the latter leads to an extra layer of constants in the proof. However, our methods still treat these variations - details are left to the interested readers. 
\end{remark}

\begin{remark}
	\label{rmk:existence}
	The above definition is empty if $\alpha$ does not exist, however, this is ruled out by our assumption on $Q$ and $X$. Indeed, let $\{\alpha_i\}$ be a mass minimizing sequence in $\C$. As $Q$ is a homology retract, $\C$ is a non-empty collection, 
	moreover, we can assume $\M(\alpha_i)\le A_1\M(\si)\cdot b_1R$. Let $M=\inf_{i\ge 1}\M(\alpha_i)$. Take $\eps_0>A_1$. Then we can find $\alpha\in \I_{n,loc}(X)$ such that 
	\begin{enumerate}
		\item $\M(\al)\le M$;
		\item $\spt(\D \alpha-\si)\subset N_{\eps_0}(Q)$;
		\item $\al$ is a minimal filling of $\D \al$.
	\end{enumerate}
	By the coarea inequality, for each $\alpha_i$, there exists $A_1<\eps_i<\eps_0$ such that $\alpha'_i=\alpha_i\on \{d_Q\ge \eps_i\}$ satisfies 
	$$\M(\partial \alpha'_i)\le \M(\si)+\frac{A_1\M(\si) b_1R}{\eps_0-A_1}\,.$$
	Up to passing to a subsequence, we suppose $\alpha'_i$ weakly converges to $\alpha\in \I_{n,loc}(X)$. Then $\M(\al)\le M$ by lower-semicontinuity of mass. (2) is clear. To arrange (3), we can replace $\alpha$ by a minimal filling of its boundary (cf. Theorem~\ref{thm:plateau}).
\end{remark}

\begin{prop}
	\label{prop_small_neck1}
	Suppose $X$ is a proper metric space satisfying condition~{\rm (SCI$_{n-1}$)}. Let $Q\subset X$ be an $(L,A)$-quasiflat which is an $A_1$-homology retract in the sense of Definition~\ref{def_tame} and is $(\mu,A_2)$-rigid in the sense of Definition~\ref{def_rigid}. Then $Q$ has CNP with $C_0>0$ depending only on $L,A,A_1,\dim(Q)$ and $X$; and $\ul{R}$ depending only on $L,A,A_1,A_2,\mu,$ and $X$.
\end{prop}

\begin{remark}
	\label{rmk:cone}
	Under a stronger assumption that $X$ has a Lipschitz bicombing, Proposition~\ref{prop_small_neck1} follows from Lemma~\ref{lem_rig_iff_supdiv}, Lemma~\ref{lem_suplindiv_implies_fullsup}, Proposition~\ref{prop_cone_conditions} and Lemma~\ref{lem_sequence}. However, as this route passes through asymptotic cones, we lose the dependence of $\ul{R}$ on other parameters in Proposition~\ref{prop_small_neck1}.
\end{remark}

The following is the key to prove Proposition~\ref{prop_small_neck1}.

\begin{prop}
	\label{prop_contraction}
	Suppose $X$ is a proper metric space satisfying condition~{\rm (SCI$_{n-1}$)}. Let $Q\subset X$ be an $(L,A)$-quasiflat which is an $A_1$-homology retract in the sense of Definition~\ref{def_tame} and is $(\mu,A_2)$-rigid in the sense of Definition~\ref{def_rigid}. Then $Q$ has CCP with $\ul{R}$ depending only on $L,A,A_1,A_2,\mu$ and $X$.
\end{prop}

Proposition~\ref{prop_small_neck1} follows from Proposition~\ref{prop_contraction} as follows. Let $\si$ be as in Definition~\ref{def_tcp}. Let $\al$ and $M$ be as in Definition~\ref{def_ccp}. The assumption of Proposition~\ref{prop_small_neck1} implies $M\lesssim \rho R\M(\si)$. Thus Proposition~\ref{prop_small_neck1} holds. 

\subsection{Proof of Proposition~\ref{prop_contraction}}
We prove Proposition~\ref{prop_contraction} in the rest of this section. Assume without loss of generality that $\rho\ll1$, $R\gg 1$, $\rho R\ge \max\{100A,1\}$ and $A=A_1=A_2>1$. Let $\si$, $\al$ and $M$ be as in Definition~\ref{def_ccp}. By Remark~\ref{rmk:existence}, we assume in addition that $\al$ is a minimal filling of $\D\al$. Definition~\ref{def_ccp} implies that we can assume without loss of generality that there is a base point $p\in Q$ such that $\spt(\tau)\subset B_p(b_1\cdot R)$ and $\spt(\si)\subset B_p(b_1\cdot R)$.

Let $n=\dim(Q)$. Suppose $Q$ is represented by $\Phi:\R^n\to Q$.
Throughout this subsection, $\lambda,\lambda_1,\lambda_2,\ldots$ will be constants which depend only on $L,A,n,\mu$ and $X$. We will also write $k_1\lesssim k_2$ if $k_1\le \lambda'k_2$ for some $\lambda'$ depending only on $L,A,n,\mu$ and $X$.

Here is an outline of the proof.
We will start with a sequence of lemmas (Lemma~\ref{lem_weak_limit} to Lemma~\ref{lem_si'}) gradually establishing control on the geometry of the slice $\alpha_h$ of $\alpha$ at distance $h$ from $Q$. We eventually want to show that $\alpha_h$ has a filling of small mass when $h$ is small. This relies on two key estimates. The first is Lemma~\ref{lem_si'} (2), where we control the filling radius of such a slice,\ i.e. 
we show that it can be filled in a small neighborhood of itself. The second is Lemma~\ref{lem_si'} (3), where we bound the 
``Minkowski content'' of the slice. The first estimate is obtained by unwinding the definition of $(\mu,b)$-rigid; and the second is a consequence of a relative version of the monotonicity formula (cf. Lemma~\ref{lem_alpha_control}). Once we have these two estimates, a compactness argument (cf. Lemma~\ref{lem_small_filling}) implies that the slice has small filling volume.

\begin{lem}
	\label{lem_weak_limit}
	There exist constants $\{\lambda_i\}_{i=1}^2$ depending only on $L,A,n$ and $X$ such that the following holds. 
	\begin{enumerate}
		\item $\D\alpha=\si-\si'$ where $\spt(\si')\subset N_{3A}(Q)$;
		\item $\M(\al)\le M\le \lambda_1\M(\si)\cdot R$;
		\item for any point $x\in\spt (\alpha)$ such that $h=d(x,Q)$ satisfies $\frac{\rho R}{2}\ge h> 6A$, we have $\M(\alpha\on  B_x(\frac{h}{2}))\ge \lambda_2\cdot h^n$;
		\item there exists $a\geq 1$ depending only on $L,A,n,C$ and $X$ such that $\spt(\alpha)\subset N_{a R}(Q)$.
	\end{enumerate}
\end{lem}

\begin{proof}
	(1) is clear. (2) follows from the definition of homology retract. (3) follows from Lemma~\ref{lem:density}. For (4), let $x\in\spt(\alpha)$ be a point such that $d(x,Q)=aR\ge R$. Since we are assuming $\rho\ll 1$, $B_x(aR-\frac{R}{2})\cap \spt(\D \alpha)=\emptyset$. By Lemma~\ref{lem:density}, 
	we know $a^nR^n\lesssim \M(\alpha\on B_x(aR-\frac{R}{2}))\le \M(\al)\le \lambda_1\M(\si)\cdot R\le \lambda_1 C\cdot R^n$. Thus (4) follows.
\end{proof}

The following is a relative version of Lemma~\ref{lem:density}.
\begin{lemma}[Monotonicity of volume]
	\label{lem_alpha_control}
	The following estimate holds for $\al$. For $3A\le r_1<r_2<\rho R$:
	\begin{equation}
	\label{eq_height_slice}
	\frac{\M(\al\on\{d_Q\le r_2\})}{\M(\al\on\{d_Q\le r_1\})}\geq \left(\frac{r_2}{r_1}\right)^\kappa\,
	\end{equation}
	where $0<\kappa\le 1$ depends only on $L,A,n$ and $X$.
\end{lemma}

\begin{proof}
	Let $3A\le r< \rho R$. Since $Q$ is an $A$-homology retract (cf. Definition~\ref{def_tame}), there exists $\beta_r\in \I_n(X)$ 
	such that $\M(\beta_r)\le A r\M(\slc{\al,d_Q,r})$ and $\spt(\D\beta_r-\slc{\al,d_Q,r})\subset N_A(Q)$. Then 
	$$
	\M(\al-\al\on\{d_Q\le r \}+\beta_r)\ge M\ge \M(\al).
	$$
	Thus $\M(\al\on\{d_Q\le r \})\le \M(\beta_r)\le A r\M(\slc{\al,d_Q,r})$. Let $f(r)=\M(\al\on\{d_Q\le r \})$. By coarea inequality, we have $\M(\slc{\al,d_Q,r})\le f'(r)$ for a.e. $r$. Then $f(r)\le A rf'(r)$ for a.e. $3A\le r< \rho R$. Since $\ln(f(t))$ is non-decreasing,
	$$
	\ln (f(r_2))-\ln (f(r_1))\ge \int_{r_1}^{r_2}\frac{f'(t)}{f(t)}dt\ge \int_{r_1}^{r_2}\frac{1}{A t}dt. 
	$$
	Then the lemma follows.
\end{proof}

For $3A\le h< \rho R$ we conclude 
\begin{equation}
\label{eq_alpha1}
\M(\al\on\{d_Q\le h \})\le \left(\frac{h}{\rho R} \right)^\kappa \cdot \M(\al) \le \left(\frac{h}{\rho R} \right)^\kappa \cdot(\lambda\M(\si)\cdot R)
\end{equation}
where the first inequality follows from Lemma~\ref{lem_alpha_control} and the second inequality follows from Lemma~\ref{lem_weak_limit} (2).

Denote by $\si_h:=\slc{\al,d_Q,h}$ the slice at hight $h$.
It follows from the coarea inequality and \eqref{eq_alpha1} that for any $3A\le h<\rho R$, there exists $h'\in [h,1.01h]$ such that
\begin{equation}
\label{eq_sigma0}
\M(\si_{h'})\lesssim \left(\frac{1}{h'} \right)\cdot\left(\frac{h'}{\rho R} \right)^\kappa \cdot(\lambda\M(\si)\cdot R).
\end{equation}

Let $Z$ be a maximal $\frac{h'}{2}$-separated net in $\spt(\sigma_{h'})$. From Lemma~\ref{lem_weak_limit} (3) and \eqref{eq_alpha1} we can bound the cardinality of $Z$: 
\begin{equation}
\label{eq_Z}
\#Z\lesssim \left(\frac{1}{h'} \right)^n\cdot\left(\frac{h'}{\rho R} \right)^\kappa \cdot(\lambda\M(\si) \cdot R).
\end{equation}

The following lemma gives control on the diameter of the slice $\si_h$ in terms of $h$. It will be used in a compactness argument later. On first reading, one can assume Lemma~\ref{lem_diamter control} and directly pass to the key estimates in Lemma~\ref{lem_si'}.

\begin{lem}[Diameter control for slices]
	\label{lem_diamter control}
	There exist functions $\phi,\varphi:[0,\infty)\to[0,\infty)$ depending only on $L,A,n,C,\rho$ and $X$ such that
	\bit
	\item $\lim_{x\to\infty}\varphi(x)=0$;
	\item $\phi$ is decreasing;
	\item  for $\varphi(R) R\le h\le \rho R$, we have
	$\spt(\si_h)\subset B_p(\phi(\frac{h}{R})R)$.
	\eit
\end{lem}

\begin{proof}
	Let $\eps=\frac{h}{R}$. Take $\delta>\frac{100A}{R}$ whose value will be determined later ($\delta$ is much smaller than $\eps$). By \eqref{eq_alpha1} and the coarea inequality, $\M(\si_{\delta R})\lesssim \left(\frac{1}{\delta R} \right)\cdot\left(\frac{\delta R}{\rho R} \right)^\kappa \cdot\M(\al)$.

	Let $\hat\alpha=\alpha\on\{d_Q\ge \delta R\}$. Define $c:X\to \mathbb R$ by $c(\cdot)=d(\cdot,\spt(\si))$. Let $x\in \spt(\si_h)$ and let $r_x=\frac{c(x)}{R}$. Let $r\in(c(x)/2,c(x))$ such that 
	$$\M(\slc{\hat\alpha,c,r})\lesssim \frac{\M(\hat\alpha)}{r_xR}\  \mathrm{and}\  \M(\slc{\si_{\delta R},c,r})\lesssim \frac{\M(\si_{\delta R})}{r_xR}.$$ Note that
	$$\D (\hat\al\on\{c\le r\})-\si=\slc{\hat\alpha,c,r}+(\si_{\de R})\on\{c\le r\}=:\xi.$$
	Let $\zeta$ be a minimal filling of $\slc{\si_{\delta R},c,r}$. Let $\xi_1=\slc{\hat\alpha,c,r}-\zeta$ and $\xi_2=\xi-\xi_1$. As each $\xi_i$ is a cycle, for $i=1,2$, let $\beta_i\in\I_{n,c}(X)$ be as in Definition~\ref{def_tame} such that $\spt(\D\beta_i-\xi_i)\subset N_A(Q)$. 
	
	Let $\bar\alpha=\hat\alpha\on\{c\le r\}-\beta_1-\beta_2$. Then $\spt(\D \bar \alpha-\si)\subset N_A(Q)$. It follows that $\M(\bar\alpha)\ge M\ge \M(\al)$. However, when $r_x>\eps$, $\M(\hat\alpha\on\{c\le r\})\le \M(\hat\alpha)-\M(\hat\alpha\on \bar B_x(\frac{h}{2}))$, thus by Lemma~\ref{lem_weak_limit} (3), 
	\begin{equation}
	\label{eq_decrease}
	\M(\beta_1)+\M(\beta_2)\ge\M(\hat\alpha\on \bar B_x(\frac{h}{2}))\ge \lambda_2\cdot h^n=\lambda_2\eps^n \cdot R^n\,.
	\end{equation}
	
	Now we estimate $\M(\beta_i)$. By the discussion above, $\M(\slc{\si_{\delta R},c,r})\lesssim C\cdot\frac{\delta^{\kappa-1}R^{n-2}}{\rho^{\kappa}r_x}$. By Theorem~\ref{thm:isop-ineq} (1), 
	$\M(\zeta)\lesssim \left(\frac{C\delta^{\kappa-1}}{\rho^{\kappa}r_x}\right)^{\frac{n-1}{n-2}}\cdot R^{n-1}$ and $\spt(\zeta)$ is contained in the $\lambda_3\left(\frac{C\delta^{\kappa-1}}{\rho^{\kappa}r_x}\right)^{\frac{1}{n-2}} R$ -- neighborhood of $\spt(\slc{\si_{\delta R},c,r})$. Define $r_0$ such that $\lambda_3\left(\frac{C\delta^{\kappa-1}}{\rho^{\kappa}r_0}\right)^{\frac{1}{n-2}}=\delta$. If $r_x>r_0$, then $\spt(\beta_2)\subset N_{2\delta R}(Q)$ and $\spt(\beta_1)\subset N_{aR}(Q)$ (by Lemma~\ref{lem_weak_limit} (4)). If in addition $\delta R>3A$, then
	$$
	\M(\beta_1)\le AaR\M(\xi_1)\le aAR(\M(\slc{\hat\alpha,c,r})+\M(\zeta))\lesssim \frac{a\M(\alpha)}{r_x}+aR\M(\zeta)
	$$
	and $\M(\beta_2)\le A\cdot (2\delta R)\cdot\M(\xi_2)\le 2\delta AR(\M(\alpha_{\de R})+\M(\zeta))$. This together with \eqref{eq_decrease} imply 
	$$
	\left[\frac{1}{r_x}+\left(\frac{\delta^{\kappa-1}}{\rho^{\kappa}r_x}\right)^{\frac{n-1}{n-2}}\right]+\left[ \frac{\delta^{\kappa-1}}{\rho^{\kappa}}\cdot \delta + \left(\frac{\delta^{\kappa-1}}{\rho^{\kappa}r_x}\right)^{\frac{n-1}{n-2}}\cdot \delta \right]\ge \Lambda\cdot \eps^n
	$$
	where $\Lambda$ is a constant depending only on $L,A,n,C$ and $X$. Choose $\delta$ such that $\frac{\delta^\kappa}{\rho^{\kappa}}=\frac{\Lambda \eps^n}{2}$. If $\eps$ is large enough such that the resulting $\delta$ satisfies $\delta R\ge 3A$, then the discussion above goes through. This gives the function $\varphi$ in the lemma. Let $r_1$ such that 
	$$
	\frac{1}{r_1}+\left(\frac{\delta^{\kappa-1}}{\rho^{\kappa}r_1}\right)^{\frac{n-1}{n-2}}+ \left(\frac{\delta^{\kappa-1}}{\rho^{\kappa}r_1}\right)^{\frac{n-1}{n-2}}\cdot \delta =\frac{\Lambda}{2}\cdot\eps^n
	$$
	Then $r_x\le\max\{r_0,r_1\}$, which gives the function $\phi$ in the lemma.
\end{proof}

\begin{lem}
	\label{lem_si'}
	There exist constants $\{\lambda_i\}_{i=4}^{6}$ depending only on $L,A$, $n$ and $X$ such that the following holds. Suppose $3A\le h\le \rho R$ and suppose $\sigma_h$ satisfies $\eqref{eq_sigma0}$ with $h'$ replaced by $h$. Then we can find a cubical chain $\sigma'_h\in \I_{n-1}(\R^n)$ such that the following holds:
	\begin{enumerate}
		\item $\M(\si'_{h})\le A\cdot \M(\sigma_h)$;
		\item $($bound on filling radius$)$ for any $\delta>0$, there exists $\ul{R}=\ul{R}(Q,C,\delta)$ such that if $R\ge \ul{R}$ and $h\ge \max\{A,\delta R\}$, then $\sigma'_h$ can be filled inside $N_{\lambda_4 h}(S')$ where $S'=:\pi(\spt(\sigma_h))$ and $\pi:X\to \R^n$ is a
		quasi-retraction as in Section~\ref{subsec_cubulated_quasiflats}.
		\item $($bound on Minkowski content at a given scale$)$ $\L^n(N_{\lambda_4 h}(S'))\le \left(\frac{h}{\rho R} \right)^\kappa \cdot(\lambda_5\M(\si)\cdot R)$;
		\item for any $\delta>0$, there exists $\ul{R}'=\ul{R}'(\delta)$ such that if $\delta R\le h_1\le h_2\le \rho R$, then $\Fill(\sigma'_{h_1}-\sigma'_{h_2})\le \lambda_{6} \left(\frac{h_2}{\rho R} \right)^\kappa R\cdot \M(\si)$;
		\item $\Fill(\si_h-\iota(\si'_h))\le \left(\frac{h}{\rho R} \right)^\kappa \cdot(A\la\M(\si)\cdot R)$.
	\end{enumerate}
\end{lem}
We use $\L^n$ to denote the $n$-dimensional Lebesgue measure.
\begin{proof}
	As $Q$ is an $A$-homology retract, we define $\si'_{h}$ to be a cubical chain such that $\si_h-\iota(\si'_h)=\D H$ for $H\in \I_{n,c}(X)$ with $\M(H)\le Ah\cdot\M(\si_h)$ and $\spt(H)\subset N_{Ah}(\spt(\si_h))$. (1) follows from Definition~\ref{def_tame}. 
	
	Now we prove (2). Represent $Q$ as a map $\Phi:\R^n\to X$.

	Let $\phi$ and $\varphi$ be as in Lemma~\ref{lem_diamter control}. Choose $R_0$ such that $\varphi(R)<\delta$ whenever $R>R_0$. Let $\al_0=\al\on\{d_Q\ge h\}+\tau$. When $\rho R\ge h\ge \max\{A,\delta R\}$ and $R>R_0$, we have
	\bit
	\item $\M(\si'_h)\le \left(\frac{\delta}{\rho}\right)^{\kappa-1}\lambda CA\cdot R^{n-1}$ (by \eqref{eq_sigma0});
	\item $\M(H+\al_0)\le (\rho^{1-\kappa}\lambda C+AC\rho+C)\cdot R^n$;
	\item $\spt(\si'_h)\subset B_{p'}(A\phi(\delta)R)$ for some $p'\in\R^n$ and $\spt(H+\al_0)\subset B_p(\phi(\delta)R)$ (by Lemma~\ref{lem_diamter control}).
	\eit
	Let $M=\max\{\delta^{1-\kappa}\lambda CA,\lambda C+AC+C,A\phi(\delta),1\}$ and let $\mu_M$ be as in Definition~\ref{def_rigid}. Choose $\bar R>R_0$ such that $\mu_M(R)<\frac{\delta}{M}\cdot R$ whenever $R\ge \bar R$. 
	
	Let $W$ be the support of the canonical filling of $\si'_h$. Let $a$ and $W_a$ as in Definition~\ref{def_rigid}. Then for $R>\bar R$,
	$$\Phi(W_a)\subset N_h(\spt(H+\al_0))\subset N_h(\spt(H))\subset N_{h+Ah}(\spt (\sigma_h))$$
	where the second $\subset$ follows from $N_h(\spt(\al_0))\cap N_A(Q)=\emptyset$. Thus $\pi(\Phi(W_a))\subset N_{\lambda'(1+A)h}(S')$ where $S'=\pi(\spt (\sigma_h))$ and $\lambda'=\lambda'(L,A,n)$ is the Lipschitz constant of $\pi$. Note that there is $a'=a'(a,L,A,n,X)$ such that $W_{a'}\subset \pi(\Phi(W_a))$. Thus $W_{a'}\subset N_{\lambda'(1+A)h}(S')$. However, $W\setminus W_{a'}\subset N_{a'}(\D W)\subset N_{A+a'}(S')$. Thus (2) follows.

	For (3), let $Z$ be as in \eqref{eq_Z}. Then $S'\subset \pi(N_{\frac{h}{2}}(Z))\subset N_{\lambda'h}(Z)$. Thus $$\L^n(N_{\la_4h}(S'))\le \L^n(N_{((\la'+\la_4)h)}(Z))\lesssim (\# Z)\cdot h^n\lesssim \frac{1}{h^n} \left(\frac{h}{\rho R} \right)^\kappa R\M(\si)h^n$$
	by \eqref{eq_Z}. Then (3) follows.

	Now we prove (4). As $\Fill(\sigma_{h_1}-\sigma_{h_2})\le \M(\alpha\on\{d_Q\le h_2\})\le  \left(\frac{h_2}{\rho R} \right)^\kappa \cdot(\lambda\M(\si)\rho \cdot R)$, we know
	$\Fill(\pi_\# \si_{h_1}-\pi_\# \si_{h_2})\lesssim \left(\frac{h_2}{\rho R} \right)^\kappa \M(\si)\cdot R$. As $Q$ is an $A$-homology retract, $\Fill(\pi_\#\si_{h_i}-\si'_{h_i})\le A\M(\si_h)\le A\left(\frac{\delta}{\rho}\right)^{\kappa-1}\cdot \la\M(\si)\lesssim \frac{\rho}{R\delta}\cdot \left(\frac{\de R}{\rho R} \right)^\kappa R\M(\si)$ for $i=1,2$. Thus (4) follows by choosing $\ul{R}'$ such that $\ul{R}'\delta=1$. (5) follows from \eqref{eq_sigma0} and the definition of homology retract.
\end{proof}

Let $\delta>0$ and $R\ge \max\{\ul R(Q,C,\delta),\ul{R}'(\delta)\}$ be as in Lemma~\ref{lem_si'}.

For each $\max\{A,\delta R\}\le hR\le \rho R$, define $\si''_h=\si'_{\bar h R}$ where $\bar hR\in [hR,1.1hR]$ is chosen such that $\si_{\bar hR}$ satisfies \eqref{eq_sigma0}. 

Let $\varphi$ be as in Lemma~\ref{lem_diamter control}. By Lemma~\ref{lem_diamter control}, \eqref{eq_sigma0} and Lemma~\ref{lem_si'}, for any $\max\{A,\varphi(R) R,\delta R\}\le hR\le \rho R$,
\bit
\item $\diam(\spt (\si''_h))\le A\phi(h)R$ and $\M(\sigma''_h)\le \lambda_7 \left(\frac{h}{\rho}\right)^{\kappa-1}\M(\si)$;
\item there is an open set $O_h\subset\R^n$ with $\L^n(O_h)\le \la_8 \left(\frac{h}{\rho} \right)^\kappa  R\M(\si)$ such that $\sigma''_h$ is the boundary of an element in $\I_{n,c}(O_h)$;
\item for $\max\{\frac{A}{R},\delta \}\le h_1\le h_2\le \rho$, $\Fill(\si''_{h_1}-\si''_{h_2})\le \la_9 \left(\frac{h_2}{\rho} \right)^\kappa   R\M(\si)$.
\eit

\begin{proof}[Proof of Proposition~\ref{prop_contraction}]
	Let $\M(\si)=tR^{n-1}$. By the isoperimetric inequality that $\Fill(\si)\le D\M(\si)^{\frac{n}{n-1}}= D t^{\frac{n}{n-1}}\cdot R^n$. Suppose
	\begin{equation}
	t<\rho_0:=\left(\frac{\eps}{D}\right)^{n-1}.
	\end{equation}
	Then $\Fill(\si)\le Dt^{\frac{1}{n-1}}R\cdot tR^{n-1}\le \eps R\cdot\M(\si)$. Now we assume $t\ge \rho_0$.
	
	We rescale the metric on $X$ by a factor $\frac{1}{R}$ and rescale $\{\si''_h\}$ accordingly. Apply Lemma~\ref{lem_small_filling} below with $\bar\eps=\rho_0\eps$, $\bar\rho=\rho\eps$, $g(h)=\lambda_7 \left(\frac{h}{\rho}\right)^{\kappa-1}\cdot C$, $f(h)=\max\{\lambda_8,\lambda_9\}\left(\frac{h}{\rho}\right)^\kappa\cdot C$ and $\phi(h)$ as in Lemma~\ref{lem_diamter control}, and choose $\delta$ as in Lemma~\ref{lem_small_filling} which depends only on $\eps,\rho,C,L,A,\mu,X$ and $n$. Take $R$ such that $R\ge \max\{\ul R(Q,C,\delta),\ul{R}'(\delta)\}$. By Lemma~\ref{lem_si'} the family $\{\si''_h\}_{\delta\le h\le \bar\rho}$ satisfies the assumptions of Lemma~\ref{lem_small_filling} (after rescaling by $\frac{1}{R}$). By the choice of $\delta$, there exists $h_0$ with $\bar\rho \ge h_0 \ge \delta $ such that $\Fill(\si''_{h_0})\le \eps\rho_0 R^n\le \eps t R^n$. Then
	$$
	\Fill(\D\al-\si)\le \M(\al\on\{d_Q\le \bar h_0R\})+\Fill(\si_{\bar h_0 R}-\iota(\si''_{h_0}))+\Fill(\iota(\si''_{h_0})).$$
	By \eqref{eq_alpha1}, $$
	\M(\al\on\{d_Q\le \bar h_0R\})\le  \left(\frac{\bar h_0R}{\rho R} \right)^\kappa \cdot(\lambda\M(\si)\cdot R)\le \eps^\kappa\lambda\M(\si)\cdot R.$$
	By Lemma~\ref{lem_si'} (5) and the choice of $\bar\rho$,  $\Fill(\si_{\bar h_0 R}-\iota(\si''_{h_0}))\le \left(\frac{h_0}{\rho} \right)^\kappa \cdot(A\la\M(\si)\cdot R)\le \eps^\kappa\cdot A\la\M(\si)\cdot R$. Note that $\Fill(\iota(\si''_{h_0}))\lesssim \Fill(\si''_{h_0})\lesssim \eps t R^n\lesssim \eps \M(\si)R$. Thus the proposition follows.
\end{proof}

\begin{lem}
	\label{lem_small_filling}
	Let $f$, $g$ and $\phi$ be functions from $\mathbb R_{\ge 0}$ to $\mathbb R_{\ge 0}$ such that $\lim_{x\to 0}f(x)=0$. Take $\bar\rho>0$.
	For any $\bar\eps>0$, there exists $\delta=\delta(\bar\eps,\bar\rho,f,g,\phi)$ such that the following holds. Suppose $\{\zeta_h\}_{\delta\le h\le  \bar\rho}$ is a collection of elements in $\I_{n-1,c}(\R^n)$ such that
	\begin{enumerate}
		\item $\diam(\spt (\zeta_h))\le \phi(h)$;
		\item $\M(\zeta_h)\le g(h)$;
		\item $\Fill(\zeta_{h_1}-\zeta_{h_2})\le f(h_2)$ for any $h_2\ge h_1\ge \delta$;
		\item $\zeta_h$ bounds in an open set $O_h$ with $\L^n(O_h)\le f(h)$ for any $h\ge \delta$. 
	\end{enumerate}
	Then there exists $\bar\rho\ge h_0\ge \delta$ such that $\Fill (\zeta_{h_0})\le \bar\eps$.
\end{lem}

\begin{proof}
	We argue by contradiction and suppose there is a monotone decreasing sequence $\de_i\to 0$ such that for each $i$, we have a sequence $\{\zeta^i_h\}_{\de_i\le h\le \bar\rho}$ satisfying conditions (1) - (4) of the lemma and $\Fill(\zeta^i_h)\ge\bar\eps$ for any $i$ and $\de_i\le h\le \bar\rho$. For $j\le i$, let $\zeta_{ij}=\zeta^i_{\de_j}$. For each fixed $j$, conditions (1) and (2) imply that $\zeta_{ij}$ sub-converges in the flat distance as $i\to\infty$. By condition (3), we can apply a diagonal argument to $\{\zeta_{ij}\}_{j\ge i}$ to extract a subsequence $\{\zeta_{n_j,j}\}_{j\ge 1}$ which is Cauchy in the flat distance. Suppose the canonical filling of $\zeta_{n_j,j}$ is represented by $u_j\in L^1(\mathbb E^n,\mathbb R)$. Then $\{u_j\}$ is a Cauchy sequence in the $L^1$-distance and $\spt(u_j)\subset O_{\de_j}$. Suppose $u_i\stackrel{L^1}{\to}u$. Let $E_i=\spt (u_i)$ and $E=\spt (u)$. Condition (4) implies that $\L^n(E_i)\to 0$. Thus $u\cdot\chi_{E\setminus E_i}\stackrel{L^1}{\to}u$. However, $||u\cdot\chi_{E\setminus E_i}||_1\le||u_i-u||_1\to 0$. Thus $u=0$ and $||u_i||_1\to 0$, which contradicts that $\Fill(\zeta_{n_j,j})\ge \bar\eps$ for all $j$.
\end{proof}

\section{Coarse neck decomposition and the main theorem}
\label{sec_neck_decomposition}
In this section we introduce our last general characterization of Morse quasiflats,  the {\em coarse piece property}, which is a ``coarse analogue'' of Definition~\ref{def_piece_decomposition}. We then prove in Theorem~\ref{thm:equivalence intro} and Theorem~\ref{thm:conditions in asymptotic cones} that all the hyperbolic notions on quasiflats we have introduced so far become equivalent under appropriate assumptions on the ambient space.

\begin{definition}
	\label{def_neckp}
	Let $X$ be a metric space with base point $p$.
	An $n$-dimensional quasiflat $Q\subset X$ has the \emph{coarse piece property} if there exist $a>0$ and a function 
	$\ul{R}:(\R_{\ge 0})^3\to\R_{\ge 0}$ such that for given $\eps,b>0$, 
	the following holds for any $R\ge \ul{R}(\eps,b,d(p,Q))$. 
	
	Take $T\in \I_{n,c}(X)$ with $\spt(\D T)\subset N_a(Q)$, $\spt(T)\subset B_p(bR)$ and $\M(T)\le b\cdot R^n$. Suppose there exists $T'\in \I_{n,c}(X)$ such that $\D T=\D T'$ with 
	$\spt(T')\subset N_a(Q)$ and $\M(T')\le b\cdot R^n$. Then $T$ admits a piece decomposition $T=U+V$ such that the following holds.
	\begin{enumerate}
		\item Let $\si=:\partial U-\partial T=-\partial V$. Then $\Fill(\si)\le \eps\cdot R^{n}$.
		\item Let $\omega$ be a minimal filling of $\si$. Then $\Fill(U+\omega-T')\le  \eps\cdot R^{n+1}$.
	\end{enumerate}
\end{definition}

\begin{lem}[Coarse piece decomposition]
	\label{lem_decomposition_lemma}
	Suppose $X$ satisfies {\rm (SCI$_{n}$)}. Let $Q\subset X$ be an $n$-dimensional $(L,A)$-quasiflat with CNP (cf. Definition~\ref{def_tcp}). Let $C$ be the constant in Proposition~\ref{lem_homology}. 
	Let $p\in X$ be a base point. 
	
	For given $\eps,b>0$, there exists $\ul{R}$ depending only on $b,\eps$, $d(p,Q),L,A,$ $n,X$ and the CNP parameter of $Q$ such that the following holds for any $R\ge \ul{R}$. 
	Take $T\in \I_{n,c}(X)$ satisfying $\spt(\D T)\subset N_C(Q)$, $\spt(T)\subset  B_p(bR)$ and $\M(T)\le b\cdot R^n$. Suppose there exists $T'\in \I_{n,c}(X)$ such that $\D T=\D T'$ with $\spt(T')\subset N_C(Q)$ and $\M(T')\le b R^n$. Then $T$ admits a piece decomposition $T=U+V$ such that the following holds.
	\begin{enumerate}
		\item Let $\si=:\partial U-\partial T=-\partial V$. Then $\Fill(\si)\le \eps\cdot R^{n}$.
		\item Let $\omega$ be a minimal filling of $\si$. Then $\spt(U+\omega-T')\subset N_{\eps R}(Q)$. 
		\item $\Fill(U+\omega-T')\le  \eps\cdot R^{n+1}$.
	\end{enumerate}
\end{lem}

\begin{proof}
	Take small constants $\rho$ and $\delta$ whose value will be determined later. Let $T_{x,y}=T\on\{xR<d_Q<yR\}$. We define a piece $\widehat T$ of $T$ as follows. If $\M(T_{\frac{\rho}{2},\rho})\le \delta\M(T)$, 
	then we define $\widehat T=T_{\frac{\rho}{2},\rho}$. Now assume $\M(T_{\frac{\rho}{2},\rho})>\delta\M(T)$. If $\M(T_{\frac{\rho}{2^2},\frac{\rho}{2}})\le \delta\M(T)$, then we define 
	$\widehat T=T_{\frac{\rho}{2^2},\frac{\rho}{2}}$. If $\M(T_{\frac{\rho}{2^2},\frac{\rho}{2}})>\delta\M(T)$, then we look at $\M(T_{\frac{\rho}{2^3},\frac{\rho}{2^2}})$ and repeat the previous process. 
	This process terminates after at most $\left \lceil{\frac{1}{\delta}}\right \rceil$ steps and we obtain $\widehat T=T_{\frac{\rho}{2^n},\frac{\rho}{2^{n-1}}}$ with 
	$\M(\widehat T)\le \delta \M(T)$ and $1\le n\le \left \lceil{\frac{1}{\delta}}\right \rceil$. Thus there exists $\frac{\rho}{2^n}<\dot \rho<\frac{\rho}{2^{n-1}}$ such that 
	$\si:=\slc{\widehat T,d_Q,\dot\rho\cdot R}$ satisfies $\M(\si)\le \frac{\M(\widehat T)}{\rho R/2^n}\le \frac{2^n \delta}{\rho R}\M(T)$. Let $U:= T\on\{d_Q\le \dot{\rho}\cdot R\}$ and $V=T\on\{d_Q> \dot{\rho}\cdot R\}$. Then
	\bit
	\item $\spt(U) \subset N_{\rho R}(Q)$ and $\spt( V)\cap N_{\rho_1 R}(Q)=\emptyset$ with $\rho_1=\rho\cdot 2^{-\frac{1}{\delta}}$.
	\item $\M(\si)\le2^{\frac{1}{\delta}} \rho^{-1}\delta b\cdot R^{n-1}$.
	\eit
	Let $b'=\max\{b,2^{\frac{1}{\delta}} \rho^{-1}\delta b\}$.
	Let 
	$$\ul{R}:=\sup_{1\le n\le \left \lceil{\frac{1}{\delta}}\right \rceil}\{\ul{R}(b',\frac{\rho}{2^n},b,d(p,Q))\},$$
	where $\ul{R}(b',\frac{\rho}{2^n},b,d(p,Q))$ is as in Definition~\ref{def_tcp}. By Definition~\ref{def_tcp}, 
	\begin{equation}
	\label{eq_mass_of_tau}
	\Fill(\si)\le C_0 \dot{\rho} R\M(\si)\le C_0\frac{\rho R}{2^{n-1}}\frac{2^n \delta}{\rho R}\M(T)=2C_0\delta\M(T)
	\end{equation}
	for $C_0$ depending only on $X$ and $Q$.
	
	Let $\beta=U+\omega-T'$. Then $\M(\beta)\le \M(U)+\M(\omega)+\M(T')\le \M(U)+\M(V)+\M(T')\le \M(T)+\M(T')$. By \eqref{eq_mass_of_tau}, $\spt(\beta)\subset N_{(\delta_1+\rho)R}(Q)$ where 
	$\delta_1=\left(\frac{2C_0\delta b}{D} \right)^{\frac{1}{n}}$ and $D$ is the constant in Lemma~\ref{lem:density}. By Corollary~\ref{cor_small_fill_close_to_Q}
	\begin{align}
	\label{eq_small_filling}
	\Fill(\beta)\le \M(H)\le C(\delta_1+\rho)R(\M(T)+\M(T')).
	\end{align}
	The lemma follows from \eqref{eq_mass_of_tau} and \eqref{eq_small_filling} by choosing $\rho$ and $\delta$ small.
\end{proof}

\begin{thm}
	\label{thm:equivalence}
	Let $Q\subset X$ be an $(L,A)$-quasiflat in a proper metric spaces $X$. Consider the following conditions:
	\begin{enumerate}
		\item $Q$ is $(\mu,b)$-rigid (cf. Definition~\ref{def_rigid}).
		\item $Q$ has super-Euclidean divergence (cf. Definition~\ref{def_divergence2}).
		\item $Q$ has cycle contracting property (cf. Definition~\ref{def_ccp}).
		\item $Q$ has coarse neck property (cf. Definition~\ref{def_tcp}).
		\item $Q$ has coarse piece property (cf. Definition~\ref{def_neckp}).
	\end{enumerate}
	Then the following hold:
	\begin{enumerate}[label=(\alph*)]
		\item $(1)$ and $(2)$ are equivalent if $X$ satisfies {\rm (CI$_{n-1}$)};
		\item $(1),(2)$, $(3)$ and $(4)$ are equivalent if $X$ satisfies {\rm (SCI$_{n-1}$)};
		\item $(1),(2),(3),(4)$ and $(5)$ are equivalent if $X$ satisfies {\rm (SCI$_{n}$)}.
	\end{enumerate}
	Moreover, all the above equivalences hold with uniform control on the parameters (e.g. if $X$ satisfies {\rm (SCI$_{n-1}$)} and $Q$ has super-Euclidean divergence, then $Q$ has CNP with parameters of CNP depends only on parameters 
	of super-Euclidean divergence and $X$).
\end{thm}
In the following proof, $a_1,a_2,\ldots$ will be constants depending only on $L,A,n$ and $X$.
\begin{proof}
	(a) is Lemma~\ref{lem_rig_iff_supdiv}. For (b), $(1)\Rightarrow(3)$ is Proposition~\ref{lem_homology} and Proposition~\ref{prop_contraction}. $(3)\Rightarrow(4)$ is already explained in Section~\ref{subsec_cnp_ccp}. Now we show $(4)\Rightarrow(1)$. 
	Let $\Phi,\varphi,\tau$ and $\nu$ be as in Definition~\ref{def_rigid}. We apply Lemma~\ref{lem_decomposition_lemma} with $T'=\iota(\nu)$, $T=\tau$ and $\D T=\iota(\varphi)$ to obtain a coarse piece decomposition $\tau=U+V$ as in Lemma~\ref{lem_decomposition_lemma}. Denote by
	$\omega$ a minimal filling of $\D V$. Note that we only used condition (SCI$_{n-1}$) so far as one only needs (SCI$_{n}$) for Lemma~\ref{lem_decomposition_lemma} (3).
	
	Let $\pi:X\to \R^n$ be a Lipschitz continuous quasi-retraction associated to $Q$ as in Section~\ref{subsec_cubulated_quasiflats}. Let $\alpha=U+\omega$. As $\Fill(\D V)\le \eps R^n$, $\spt(\omega)\subset N_{\eps'}(\spt(\D \om))$ where $\eps'=\left(\frac{\eps}{D}\right)^{\frac{1}{n}}R$. Thus 
	\begin{equation}
	\label{eq_equi_1}
	\spt(\alpha)\subset N_{\eps'}(\spt(U))\subset N_{\eps'}(\spt (T)).
	\end{equation}
	Let $h\in \I_{n,c}(\R^n)$ be such that $\D h=\D\pi_\#\alpha-\beta$, $\beta$ is cubical and $\spt(h)\subset N_{a_1}(\spt(\D\pi_\#\alpha))$. Set $\alpha_\square=\pi_\#\alpha+h$ and $T'_\square=\pi_\#T'+h$. Then $\D\alpha_\square=\D T'_\square$ and $\spt(\alpha_\square)\subset N_{a_1}(\spt(\pi_\#\alpha))$. Note that $\alpha_\square=T'_\square$ as they are top dimensional. Let $\bar\alpha=\iota(\alpha_\square)$ and $\bar T'=\iota (T'_\square)$. Then 
	\begin{equation}
	\label{eq_equi_2}
	\spt(\bar T')=\spt(\bar\alpha)\subset N_{a_2\eps R}(\spt(\alpha)).
	\end{equation}
	Let $W_a$ be as in Definition~\ref{def_rigid}. Let $W'=\spt(T'_\square)$ and $W'_a=\{x\in W'\mid d(x,\spt(\D T'_\square))>a\}$. We find $a_3$ and $a_4$ such that 
	\begin{equation}
	\label{eq_equi_3}
	W_{a_3}\subset W'_{a_4}\ \text{and}\ \Phi(W'_{a_4})\subset N_{a_5}(\spt(\bar T')).
	\end{equation}
	The first $\subset$ follows from $T'_\square=\pi_\#\iota(\nu)+h$, and the second $\subset$ is from Proposition~\ref{prop_chain_map} (4). Now \eqref{eq_equi_1}, \eqref{eq_equi_2} and \eqref{eq_equi_3} imply $Q$ is $\mu$-rigid.
	
	Now we prove (c). $(4)\Rightarrow(5)$ is Lemma~\ref{lem_decomposition_lemma}. It remains to show $(5)\Rightarrow(1)$. Let $\varphi,\tau,W$ and $\nu$ be as in Definition~\ref{def_rigid}. Let $T'=\iota(\nu)$, $T=\tau$ and $\D T=\iota(\varphi)$. Let $T=U+V$ and $\omega$ be as in Definition~\ref{def_neckp}. By Proposition~\ref{prop_chain_map}, $T'$ is $(a_6,a_7)$-quasi-minimizing mod
	$N_{a_7}(\Phi((\spt (\D W))^{(0)}))$. Let $x\in \spt(T')$ such that $$B_x(5a_7)\cap N_{a_7}(\Phi((\spt (\D W))^{(0)}))=\emptyset$$ and let $r_0=d(x,\spt(U+\omega))$. Then by Definition~\ref{def_neckp} (2) and the filling density estimate Lemma~\ref{lem:fill-density}, $a_8(r_0)^{n+1}\le \eps R^{n+1}$. This together with Proposition~\ref{prop_chain_map} (4) and \eqref{eq_equi_1} implies that $Q$ is $\mu$-rigid.
\end{proof}

\begin{remark}
	\label{rmk:quasidisk equivalence}
	The formulation and proof of Theorem~\ref{thm:equivalence} extends to the case of quasidisks instead of quasiflats.
\end{remark}

The following is a consequence of Proposition~\ref{prop_cone_conditions}, Proposition~\ref{prop_bridge} and Theorem~\ref{thm:equivalence}.
\begin{thm}
	\label{thm:conditions in asymptotic cones body}
	Suppose $X$ is a proper metric space satisfying {\rm (SCI$_{n}$)}. Let $\mathcal{Q}$ be a family of $n$-dimensional quasiflats with uniform quasi-isometry constants. Suppose in addition that any asymptotic cone of $X$ satisfies {\rm (SCI$_{n}$)}, {\rm (EII$_{n+1}$)} and coning inequalities up to dimension $n$ for singular chains, Lipschitz chains and compact supported integral currents. Then the following conditions are equivalent and they are all  equivalent to each of the conditions in Theorem~\ref{thm:equivalence} and with uniform parameters for each element of $\mathcal{Q}$.
	\begin{enumerate}
		\item  $\mathcal{Q}$ has the asymptotic piece property (Definition~\ref{def_piece_decomposition}).
		\item  $\mathcal{Q}$ has the asymptotic neck property (Definition~\ref{def_neck_decomposition}).
		\item  $\mathcal{Q}$ has the asymptotic weak neck property (Definition~\ref{def_weak_neck_decomposition}).
		\item $\mathcal{Q}$ has the asymptotic full support property (Definition~\ref{def_full_support}) with respect to reduced singular homology.
		\item $\mathcal{Q}$ has the asymptotic full support property with respect to  reduced homology induced by  Ambrosio-Kirchheim currents.
	\end{enumerate}
	All asymptotic cones here are taken with base points in elements of $\mathcal{Q}$. 
\end{thm}
The assumptions of Theorem~\ref{thm:conditions in asymptotic cones body} are satisfied if any asymptotic cone of $X$ admits a Lipschitz combing (Remark~\ref{rmk_combing}).

\begin{remark}
	Similar to Definition~\ref{def_rigid}, we can also formulate pointed versions of Definitions \ref{def_divergence2}, \ref{def_tcp} and \ref{def_neckp}. The proofs of Theorem~\ref{thm:equivalence} and Theorem~\ref{thm:conditions in asymptotic cones body} extend to the pointed setting (in Theorem~\ref{thm:conditions in asymptotic cones body} we only consider asymptotic cones with constant base point). 
\end{remark}

\section{Stability of Morse quasiflats}
\label{sec_Morse_Lemma}
In this section we prove several stability properties of Morse quasiflats, including a version of the Morse Lemma. These are proven in a general setting, we only require the ambient metric space to satisfy coning inequalities. However, we also provide alternative simpler proofs in Section~\ref{subsec:work with cones} under the stronger assumption that the ambient metric space has a Lipschitz bicombing.

Recall that $(\mu,b)$-rigid quasidisks are defined in Definition~\ref{def_rigid} and Remark~\ref{rmk_rigid quasidisks}. Note that the chain map $\iota$ from Proposition~\ref{prop_chain_map} and 
the chain projection from Definition~\ref{def_chain_projection} can be defined for quasidisks instead of quasiflats.

In the following we will assume that the domain of a quasidisk is an $n$-dimensional cube in Euclidean space $\R^n$. The cube is subdivided into smaller cubes whose size depends on the quasi-isometry constants. 
The domain is endowed with the $\ell^\infty$-metric rather than $\ell^2$-metric, so that the level sets of the distance function to a point is the boundary of some rectangle.
\begin{lemma}
	\label{lem_homotopy_disk}
	Let $X$ be a complete metric space satisfying the cone inequality condition ~{\rm (CI$_{k}$)}. Let $D$ and $D'$ be $n$-dimensional quasidisks in $X$ with $D\subset N_r(D')$. Let $\la>0$ be a given constant. 
	Then there exists $C$ depending only on $X,n,\lambda$ and the quasi-isometry constants of $D$ and $D'$ such that the following holds.
	
	Suppose $\dd$ is the domain of $D$. Take a $k$-dimensional ($k\le n$) rectangular subcomplex $P$ of $\dd$ such that ratio of the longest side and the shortest side is $\le \la$. Let $\bar\si$ be a cubical $k$-chain which represents the relative fundamental class of $P$. Put $\si=\iota(\bar\si)$. Let $\si'=(\si)_{D'}$ be the chain projection as in Definition~\ref{def_chain_projection}. We assume $\D \si'=(\D\si)_{D'}$ (this can always be arranged). Then there exist $\tau\in \I_{k+1,\cs}(X)$ and $\beta\in\I_{k,\cs}(X)$ such that 
	\begin{enumerate}
		\item $\D\tau=\si'-\si+\beta$;
		\item $\M(\tau)\le Cr\cdot\M(\si)$ and $\spt(\tau)\subset N_{Cr}(\spt(\si))$;
		\item $\M(\beta)\le Cr\cdot\M(\D\si)$ and $\spt(\beta)\subset N_{Cr}(\spt(\D\si))$.
	\end{enumerate}
\end{lemma}
Taking a new cubical subdivision of $\dd$ such that the size of cubes are comparable to $r$. The proof is  straight forward by induction on skeleta with respect to this new subdivision. 
This together with bounding filling radius in terms of mass (cf. Theorem~\ref{thm:isop-ineq} (2)) gives the extra control on the location of $\spt(\tau)$ and $\spt(\beta)$ in (3) and (4).

\begin{lemma}
	\label{lem_quasi_disk_close}
	Let $D$ and $D'$ be $n$-dimensional $(L,A)$-quasidisks in a metric space $X$. Suppose $d_H(\D D,\D D')\le A_1$ and $D\subset N_{A_2}(D')$. Then there exists $C$ depending only on $L,A,A_1,A_2$ and $n$ such that $d_H(D,D')\le C$.
\end{lemma}

\begin{proof}
	Let $\dd$ and $\dd'$ be the domains of $D$ and $D'$ respectively. Suppose the domains have radius $R\gg A,A_1,A_2$. 
	We can define a map $f:\dd\to\dd'$ as follows. Take a point $p\in\dd$, send to its image $\hat p\in X$, and send $\hat p$ to one of its nearest points $\hat p'$ in $D'$, and send $\hat p'$ to a corresponding point $p'\in\dd'$. 
	Up to perturbing $f$ a bounded amount comparable to $\max\{A,A_1,A_2\}$, we can assume $f$ is continuous, $f(\D\dd)\subset \D\dd'$ and $f|_{\D\dd}$ has degree 1. By a standard homological argument, we know $f$ is surjective. Hence the lemma follows.
\end{proof}

\begin{prop}
	\label{prop_Morse_lemma_for_quasidisks}
	Suppose $X$ is a complete metric space satisfying condition ~{\rm (CI$_{n}$)}. Given $(\mu,b)$ as in Definition~\ref{def_rigid} and positive constants $L,A,A',n$, there exists $C$ depending only on $\mu,b,L,A,A',n$ and $X$ such that the following holds.
	
	Let $D$ and $D'$ be two $n$-dimensional $(L,A)$-quasidisks in $X$ such that $d_H(\D D,\D D')<A'$ and $D$ is $(\mu,b)$-rigid. Then $d_H(D,D')<C$.
\end{prop}

To illustrate the idea, we start with a discussion of the proof focusing on the simpler case when $X$ is $\cat(0)$, and $D$ and $D'$ are represented by two $L$-bilipschitz embeddings $f_1: \dd\to X$ and $f_2:\dd\to X$ such that $\dd$ is a disk in $\mathbb R^n$ with radius $R$ and $(f_1)_{|\D\dd}=(f_2)_{|\D\dd}$. The proof uses a similar strategy as \cite[Proposition 4.5]{higherrank}.

Note that there is a Lipschitz retraction $\pi':X\to D'$ with $\Lip(\pi')$ depending only on $X,L$ and $n$.
Take $\tau,\tau'\in\I_{n,c}(X)$ representing the fundamental classes of $D$ and $D'$ respectively. Let $p\subset \dd$ be a base point and $\hat p=f_1(p)$. For $r>0$, let $\dd_r$ be the collection of points in $\dd$ having distance $\le r$ from $p$. Let $\tau_r=(f_1)_{\#}(\bb{\dd_r})$ where $\bb{\dd_r}$ is the fundamental class of $\dd_r$.
Throughout the following proof, $\lambda_1,\lambda_2,\ldots$ will be constants whose values depend only on $X,L,n$.

Take $r_0$ is comparable to $R$ such that $B_{p}(r_0)$ contains $\dd$. Take $\frac{1}{1000}>\eps_0>0$. The key is to show we can find $\ul{r}>0$ depending only on $\mu,b,X,n,L$ such that the following holds true. There exist a sequence of scales $\{r_i\}_{i=0}^\ell$ with $r_{i+1}= \frac{r_i}{2}$ and $\frac{r_\ell}{2}\le \ul{r}\le r_\ell$, and $\beta_i\in\I_{n,c}(X)$ such that the following holds for each $1\le i\le n$:
\begin{enumerate}
	\item $\D\beta_i=\D\tau_i-\D\tau'_i$ where  $\tau_i=\tau_{r_i}$ and $\tau'_i=\pi'_{\#}(\tau_i)$;
	\item $\M(\beta_i)\le \eps_0 r^{n}_i$ and $\spt(\beta_i)\subset N_{\eps_0 r_i}(\spt(\D\tau_i))$;
	\item $\spt(\tau'_i)\subset N_{Cr_i}(\spt(\tau_i))$ where $C>1$ is a constant (depending only on $L$) such that $\spt(\tau')\subset N_{Cr_0}(\spt(\tau))$.
\end{enumerate}

We now define $\beta_i$ inductively, and choose the correct $\ul{r}$ so the induction goes through for each $i$. 
The $i=0$ case follows from our choice of $C$ ($\beta_0=0$). Now suppose $\beta_k$ is defined with the required properties.  

As $\M(\tau_k)\le \lambda_1 r^n_k$ and $\M(\beta_k)\le  r^{n}_k$, we have $\M(\tau'_k)\le \lambda_2r^n_k$. Thus $\M(\tau'_k+\beta_k)\le \lambda_3r^n_k$. 
As $\D\tau_k=\D(\tau'_k+\beta_k)$ and $\spt(\tau'_k+\beta_k)\subset B_{\hat p}(\lambda_4r_k)$, the $(\mu,b)$-rigidity of $D$ implies that for any $\eps>0$, there exists $\ul{r}$ such that if $r_k\ge \ul{r}$, then $\spt(\tau_k)\subset N_{\eps r_k}(\spt(\tau'_k+\beta_k))$. We will later choose $\eps$ ($\eps\ll \eps_0$), which determines the value of $\ul{r}$.  
Let $\hat\tau_k=\tau_{(1-2\eps_0)r_k}$.
As $\spt(\beta_k)\subset N_{\eps_0 r_k}(\spt(\D\tau_k))$, $\spt(\hat\tau_k)\subset N_{\eps r_k}(\spt (\tau'_k))$.

Let $\tau_{k+1}=\tau_{r_{k+1}}$ and $\tau'_{k+1}=\pi'_{\#}(\tau_{k+1})$. As each point in $\spt(\tau_{k+1})$ is at most distance $\eps r_k=2\eps r_{k+1}$ from $D'$ and $\pi'$ is a Lipschitz retraction, we know that $\spt(\tau'_{k+1})\subset N_{\lambda_5\eps r_{k+1}}(\spt(\tau_{k+1}))$. Define $\beta_{k+1}$ to be the current induced by the geodesic homotopy between $\D\tau_{k+1}$ and $\pi'_{\#}(\D\tau_{k+1})$. As each point in $\spt(\D\tau_{k+1})$ is at most distance $2\eps r_{k+1}$ from $D'$ and $\M(\D\tau_{k+1})=\lambda_6r^{n-1}_{k+1}$, we know $\M(\beta_{k+1})\le \lambda_7\eps r^n_{k+1}$ and $\spt(\beta_{k+1})\subset N_{\lambda_8\eps r_{i+1}}(\spt(\D\tau_{k+1}))$. Moreover, $\D \beta_{k+1}=\tau_{k+1}-\tau'_{k+1}$. Thus we choose $\eps$ such that $\lambda_7\eps<\eps_0$, $\lambda_8\eps<\eps_0$, $\lambda_5\eps<C$. This will ensure we can continue the induction as long as $r_k\ge \ul{r}$.

When the induction reaches $\frac{r_\ell}{2}\le \ul{r}\le r_\ell$, we can still deduce $\spt(\hat\tau_\ell)\subset N_{\eps r_\ell}(\spt (\tau'_\ell))\subset N_{2\eps \ul{r}}(\spt (\tau'_\ell))$, which implies $d(\hat p,D')\le 2\eps\ul{r}$. As $\hat p$ can be any point in $D$, we are done by Lemma~\ref{lem_quasi_disk_close}.

Now we switch to the general case of Proposition~\ref{prop_Morse_lemma_for_quasidisks}, which follows the same line, but with extra layers of approximations.
 
 \begin{proof}
 	Throughout the following proof, $\lambda_1,\lambda_2,\ldots$ will be constants whose values depend only on $X,L,A,A',n$.
 	Let $\dd$ and $\dd'$ be the domains of $D$ and $D'$. Recall that $\dd$ and $\dd'$ are cubes with $\ell^\infty$ metric. Let $R=\max\{\diam(\dd),\diam(\dd')\}$. Let $\pi:X\to \dd'$ be a Lipschitz quasi-retraction as in Definition~\ref{def_retraction}. Let $\tau=\iota(\bb{\dd})$ where $\bb{\dd}$ denotes the fundamental class. By Proposition~\ref{prop_chain_map}, $d_H(D,\spt(\tau))\le \lambda_1$. 
 	Let $\tau':=(\tau)_{D'}$ be as in Definition~\ref{def_chain_projection}.

   Let $p\in \dd $ be a base point and let $\hat p\in X$ be its image. For $r>0$, let $\dd_r=\bar B_p(r)$. Then $\D\dd_r$ carry the structure of a cycle. Applying Federer Fleming deformation to this cycle to obtain a cubical cycle $S'_r$ (w.r.t. the cubulation of $\dd$). Let $\square_r$ be the cubical chain filling $S'_r$. Define $\tau_r=\iota(\square_r)$. Then $\M(\tau_r)\le \lambda_2 r^n$ and $\M(\D\tau_r)=\M(\iota(S'_r))\le\lambda_2 r^{n-1}$.
   
   Take $r_0$ comparable to $R$ such that $B_{p}(r_0)$ contains $\dd$. Then $\tau_{r_0}=\tau$. As $d_H(\D D, \D D')\le A'$, we have $\spt(\D\tau)\subset N_{\lambda_2 A'}(D')$. As $\M(\D\tau)\le \lambda_2 r^{n-1}_0$, by Lemma~\ref{lem_homotopy_disk}, there is $\beta_0\in\I_{n,c}(X)$ such that $\D\beta_0=\D(\tau-\tau')$, $\M(\beta_0)\le \lambda_3A\cdot r^{n-1}_0$ and $\spt(\beta_0)\subset N_{\lambda_3 A}(\spt(\D\tau))$. We assume without loss of generality that $r_0\ge R\ge 1000\lambda_3A$.

   Take $\eps_0=\frac{1}{1000}$. We claim there is $\ul{r}>0$ depending only on $\mu,b,X,n,$ $L,A,A'$ such that the following holds true. There exist a sequence of scales $\{r_i\}_{i=0}^\ell$ with $r_{i+1}= \frac{r_i}{2}$ and $\frac{r_\ell}{2}\le \ul{r}\le r_\ell$, and $\beta_i\in\I_{n,c}(X)$ such that the following holds for each $1\le i\le n$:
   \begin{enumerate}
   	\item $\D\beta_i=\D\tau_i-\D\tau'_i$ where  $\tau_i=\tau_{r_i}$ and $\tau'_i=(\tau_i)_{D'}$;
   	\item $\M(\beta_i)\le \eps_0 r^{n}_i$ and $\spt(\beta_i)\subset N_{\eps_0 r_i}(\spt(\D\tau_i))$;
   	\item $\spt(\tau'_i)\subset N_{Cr_i}(\spt(\tau_i))$ where $C>1$ is a constant (depending only on $L$ and $A$) such that $\spt(\tau')\subset N_{Cr_0}(\spt(\tau))$.
   \end{enumerate}
  Once this claim is established, we can finish the proof as in the simpler case discussed before. Now we prove this claim. The case $i=0$ follows from our choice of $\beta_0,C,r_0$ and $\eps_0$.
 
 Now suppose $\beta_k$ is defined with the required properties. Then $\M(\tau'_k+\beta_k)\le \lambda_4r^n_k$, $\D\tau_k=\D(\tau'_k+\beta_k)$ and $\spt(\tau'_k+\beta_k)\subset B_{\hat p}(\lambda_5r_k)$, the $(\mu,b)$-rigidity of $D$ implies that for any $\eps>0$, there exists $\ul{r}$ such that if $r_k\ge \ul{r}$, then $\spt(\tau_k)\subset N_{\eps r_k}(\spt(\tau'_k+\beta_k))$. We will later choose $\eps$ ($\eps\ll \eps_0$), which determines the value of $\ul{r}$.  
 Let $\hat\tau_k=\tau_{(1-2\eps_0)r_k}$.
 As $\spt(\beta_k)\subset N_{\eps_0 r_k}(\spt(\D\tau_k))$, $\spt(\hat\tau_k)\subset N_{\eps r_k}(\spt (\tau'_k))$.

 Let $\tau_{k+1}=\tau_{r_{k+1}}$ and $\tau'_{k+1}=(\tau_{k+1})_{D'}$. As each point in $\spt(\tau_{k+1})$ is at most distance $\eps r_k=2\eps r_{k+1}$ from $D'$, we know that $\spt(\tau'_{k+1})\subset N_{\lambda_6\eps r_{k+1}}(\spt(\tau_{k+1}))$. As  $\spt(\D\tau_{k+1})\subset N_{\eps r_k}(\spt (\tau'_k))\subset N_{\lambda_7\eps r_{k+1}}(D')$ and $\M(\D\tau_{k+1})=\lambda_3r^{n-1}_{k+1}$, we deduce from Lemma~\ref{lem_homotopy_disk} that there exists $\beta_{k+1}$ such that $\M(\beta_{k+1})\le \lambda_8\eps r^n_{k+1}$ and $\spt(\beta_{k+1})\subset N_{\lambda_9\eps r_{i+1}}(\spt(\D\tau_{k+1}))$. Moreover, $\D \beta_{k+1}=\tau_{k+1}-\tau'_{k+1}$. Thus we choose $\eps$ such that $\lambda_8\eps<\eps_0$, $\lambda_9\eps<\eps_0$, $\lambda_6\eps<C$. This will ensure we can continue the induction as long as $r_k\ge \ul{r}$.
 \end{proof}

\begin{prop}\label{prop_divergence}
	Let $X$ be a complete metric space satisfying condition ~{\rm (CI$_{n}$)} and
	let $Q, Q'\subset X$ be $(L,A)$-quasiflats with $\dim Q=n$. Suppose that $Q$ is $(\mu,b)$-rigid. Then there exist $A'$ and $\eps$ depending only on $X,L,A,n,b$ and $\mu$ such that either $d_H(Q,Q')\le A'$, or 
	\begin{equation}
	\label{eq_linear_div}
	\limsup_{r\to\infty} \frac{d_H(B_p(r)\cap Q,B_p(r)\cap Q')}{r}\ge \eps
	\end{equation}
	for some (hence any) $p\in Q$.
\end{prop}

\begin{proof}
	By a similar ``going down on scale'' argument as in Proposition~\ref{prop_Morse_lemma_for_quasidisks}, we can find $A'$ and $\eps$ as required such that either $Q\subset N_{A'}(Q')$ or \eqref{eq_linear_div} holds. 
	However, in the first case we must have $\dim Q=\dim Q'=n$ as $Q$ is rigid. Thus $d_H(Q,Q')\le A'$ by Lemma~\ref{lem_quasi_disk_close}.
\end{proof}

\section{Criteria and Examples for Morseness}
\label{sec_example}
In this section we give several more characterizations of Morse quasiflats in $\cat(0)$ spaces and present several related examples and non-examples of Morse quasiflats.

\subsection{Failure of Morseness and flat half-spaces}

Let $X$ be a $\cat(0)$ space. Let $\Sigma_pX$ be the space of directions at $p$. Let $T_pX$ be the tangent cone of $X$ at a point $p\in X$, which is defined to be the Euclidean cone over $\Sigma_pX$. Then $T_pX$ is a $\cat(0)$ space, and there is a 1-Lipschitz map $\log_p:X\to T_pX$ sending a geodesic emanating from $p$ to its corresponding geodesic emanating from the cone point $o$ of $T_pX$. We refer to \cite[Chapter II.3]{bridson_haefliger} for more background on tangent cones. 

Given a scaling sequence $\eps_i\to 0$, we obtain a blow up $Y$ at $p$ which is the ultralimit $\lim_\om(\frac{1}{\eps_i}X,p)$. There is an exponential map $\exp:T_pX\to Y$ which is an isometric embedding (see e.g. \cite[Section 5.2]{lytchak2005differentiation}).

Let $\Phi:\R^n\to X$ be a bilipschitz embedding. Let $E$ be a countable dense subset of the unit sphere of $\R^n$. By the proof of \cite[Proposition 6.1]{lytchak2005differentiation} (see \cite[Lemma 11]{stadler2018structure} as well for explanation), there exists a full measure subset $G\subset\R^n$ such that for any $p\in G$,
\begin{enumerate}[label=(\alph*)]
	\item $p$ is a point of metric differentiability with non-degenerate metric differential (cf. \cite{kirchheim1994rectifiable});
	\item for any vector $\vec{v}\in \mathbb R^n$, the points in $\Sigma_{\Phi(p)}X$ represented by the geodesics from $\Phi(p)$ to $\Phi(p+t\vec v)$ form a Cauchy sequence as $t\to 0^+$, hence converges to a point in $\Sigma_{\Phi(p)}X$ as $t\to 0^+$. We denote this point by $\D\Phi(\vec v)$.
\end{enumerate}

\begin{lemma}
	\label{lem:CAT differential}
	Suppose $p$ satisfies (a) and (b) as above.
	Let $x=\Phi(p)$ and let $Q=\im \Phi$. Take a sequence $\eps_i\to 0$. Let $\Phi_\om$ be the ultralimit of $\{\frac{1}{\eps_i}\Phi:(\frac{1}{\eps_i}\R^n,p)\to(\frac{1}{\eps_i}X,x)\}$. Let $Y=\lim_\om(\frac{1}{\eps_i}X,p)$ and let $Y_e=\exp(T_xX)\subset Y$. Then
	\begin{enumerate}
		\item $\im \Phi_\om\subset Y_e$;
		\item $F:=\im \Phi_\om$ is a flat in $Y$;
		\item there exists a function $\de:\R_{\ge 0}\to \R_{\ge 0}$ with $\lim_{r\to 0^+}\frac{\de(r)}{r}=0$ such that $\log_x(\Phi(\B{p}{r}))\subset N_{\delta(r)}(F)$;
		\item define $\Phi_i:\frac{1}{\eps_i}B_p(\eps_i)\to Y_e$ by $\Phi_i=\log_x\circ (\frac{1}{\eps_i}\Phi)$, then for any $\de>0$, there exists $i_0$ such that $d(\Phi_i(x),\Phi_\om(x))\le \eps$ for any $i\ge i_0$ and $x$ in the domain of $\Phi_i$ (we identify the domain of $\Phi_i$ with an unit ball in the domain $\Phi_\om$);
		\item for a non-trivial abelian group $\f$, if $$H_n(Q,Q\setminus\{x\},\f)\to H_n(X,X\setminus\{x\},\f)$$ is not injective, then $H_n(F,F\setminus\{o\},\f)\to H_n(Y,Y\setminus\{o\},\f)$ is not injective where $o=\log_x(x)$.
	\end{enumerate}	
\end{lemma}

\begin{proof}
	(1) follows from property (b) as above. For (2), note that $\im \Phi_\om$ is an isometrically embedded normed vector space in $Y$ by property (a) above, however, the norm has to be Euclidean as $Y$ is $\cat(0)$. 
	
	Take $\eps>0$. Take a finite subset $E'\subset E$ such that $E'$ is an $\eps$-net in the unit sphere of $\R^n$. Take $r$ such that $$d_{\Sigma_x X}(\log_x(\overline{\Phi(p)\Phi(p+t\vec{v})}),\D\Phi(\vec{v}))<\eps$$ for any $\vec{v}\in E'$ and $0\le t\le r$. As $\log_x$ is 1-Lipschitz, we deduce that $\log_x(\Phi(\B{p}{t}))\subset N_{C\eps t}(F)$ for any $0\le t\le r$ where $C$ is constant depending only on $\lip(\Phi)$. Thus (3) follows. (4) is proved similarly.
	
	Take $r_0$ such that $\de(r)\le \frac{r}{1000\lip(\Phi)}$ for any $0\le r\le r_0$. The non-injectivity assumption in (5) implies that there is a non-trivial element $\si\in \tilde{H}_{n-1}(\Phi(S_p(r_0)),\f)$ such that $\si=\D\tau$ for a singular chain $\tau$ in $X$ with $p\notin\im \tau$. Let $\si'=\log_p(\si)$ and $\tau'=\log_p(\tau)$. Then $o\notin \im\tau'$ and $d(\im \si',o)\ge \frac{r_0}{\lip(\Phi)}$. Denote the $\cat(0)$ projection onto $F$ by $\pi_F$. Let $\si''=\pi_F(\si')$ and let $\tau''$ be a chain induced by the geodesic homotopy between $\si'$ and $\si''$. The choice of $r_0$ implies that $o\notin \spt(\tau'')$ and $o\notin \spt(\si'')$. By (4), $d(\pi_F\circ\Phi_i(x),\Phi_\om(x))\le 2\eps$ for any $x$ in the domain of $\Phi_i$, thus $[\si'']$ is non-trivial in $H_n(F,F\setminus\{o\},\f)$. As $\si''=\D(\tau'+\tau)$ with $o\notin \im (\tau'+\tau'')$, (5) follows.
\end{proof}

A \emph{blow up} of a metric space $X$ is an ultralimit of $(\frac{1}{\eps_i}X,p_i)$ with $\eps_i\to 0$. An \emph{iterated blow up} of $X$ is a metric space $Y$ such that there is a finite chain of metric spaces $X=X_0,X_1,\ldots,X_n=Y$ such that $X_{i+1}$ is a blow up of $X_i$.
\begin{lemma}
	\label{lem_blow_up}
	Let $X$ be a $\cat(0)$ space and $Q\subset X$ be a bilipschitz flat. Let $\f$ be a non-trivial abelian group. If there exists $p\in Q$ such that the map $H_n(Q,Q\setminus\{p\},\f)\to H_n(X,X\setminus\{p\},\f)$ is not injective, then there exists an iterated blow up $Z$ of $X$ such that the limit $Q_Z$ of $Q$ in $Z$ is an $n$-flat which bounds a flat half-space in $X_\om$.
\end{lemma}

\begin{proof}
 We claim it is possible to take a blow up $X_0$ of $X$ such that the limit $Q_0$ of $Q$ is flat and there exists $p_0\in Q_0$ such that the map $H_n(Q_0,Q_0\setminus\{p_0\},\f)\to H_n(X_0,X_0\setminus\{p_0\},\f)$ is not injective. To see this, note that the collection of all $p\in Q$ such that $H_n(Q,Q\setminus\{p\},\f)\to H_n(X,X\setminus\{p\},\f)$ is not injective is an open subset of $Q$. Take one such $p$ which also satisfies conditions (a) and (b) before Lemma~\ref{lem:CAT differential}. Taking $\eps_i\to 0$ and blowing up at $p$, we have $Q_0=\lim (\frac{1}{\eps_i}Q,p)$ sitting inside $X_0=\lim (\frac{1}{\eps_i}X,p)$. Now the claim follows from Lemma~\ref{lem:CAT differential}. 
	
	Define $Y_1=T_{p_0} Y_0$ and $Q_1=\log_{p_0}(Q_0)$. Then $Q_1$ is a flat in $Y_1$. The map $\log_{p_0}$ implies that $H_n(Q_1,Q_1\setminus\{o_1\},\f)\to H_n(Y_1,Y_1\setminus\{o_1\},\f)$ is not injective where $o_1$ is the cone point of $Y_1$. Take $p_1\in Q_1$ such that $p_1\neq o_1$. Then $H_n(Q_1,Q_1\setminus\{p_1\},\f)\to H_n(Y_1,Y_1\setminus\{p_1\},\f)$ is not injective. Define $Y_2=T_{p_1} Y_1$ and $Q_2=\log_{p_1}(Q_1)$. Then we have non-injectivity of local homology for $Q_2\to Y_2$ as before. Moreover, there is a line $\ell$ in $Q_2$ such that $Y_2=\ell\times Y'_2$ and $Q_2=\ell\times Q'_2$. Then non-injectivity of local homology holds for $Q'_2\to Y'_2$. Repeat the previous process for $Y'_2$. This produces a sequence $(Y_i,Q_i)_{i=1}^n$ such that $Y_i=T_{p_{i-1}}Y_{i-1}$ and $Q_i=\log_{p_{i-1}}(Q_{i-1})$. $Y_i$ splits off more and more line factors and we assume in the end $Y_n=Q_n\times Q^{\perp}_n$. Non-injectivity of local homology holds for $Q_n\to Y_n$. Hence $Q_n$ bounds an isometrically embedded copy of $Q_n\times[0,a]$. 
	
	We claim for any pair of compact sets $(C,K)$ with $C\subset K$, $C\subset Q_n$ and $K\subset Y_n$, there exist a sequence of pairs of compact subsets $(C_i,K_i)$ and a scaling sequence $\la_i\to 0$ such that
	\begin{enumerate}
		\item $C_i\subset K_i$;
		\item $C_i\subset Q_0$ and $K_i\subset X_0$;
		\item $\lim_i\frac{1}{\la_i}(C_i,K_i)=(C,K)$ in the sense of Gromov-Hausdorff.
	\end{enumerate}
	We induct on $n$. The base case $n=0$ is trivial. In general, by \cite[Lemma 2.1]{kleiner1999local}, we can always take $(C'_i,K'_i)$ satisfying all the above conditions with (2) replaced by $C'_i\subset Q_{n-1}$ and $K'_i\subset X_{n-1}$. The claim follows by applying the induction assumption to each pair $(C'_i,K'_i)$ and running a diagonal argument. Now we apply the claim to the case where $C$ is a top-dimensional unit cube in $Q_n$ and $K=C\times[0,a]$. It gives a blow up $Y'$ of $X_0$ such that the limit $Q'\subset Y'$ of $Q_0\subset X_0$ has a top-dimensional cube $C'\subset Q'$ which bounds $C'\times [0,a]$. Now the lemma follows by blowing up $Y'$ again.
\end{proof}

\begin{prop}
	\label{lem_coefficient}
	Let $Q$ be an $n$-quasiflat in a $\cat(0)$ space $X$. Let $\f$ be a non-trivial abelian group. Consider the following conditions.
	\ben
	\item There exists an asymptotic cone $X_\om$ of $X$ and $p_\om\in Q_\om$ such that the map $H_n(Q_\om,Q_\om\setminus\{p_\om\},\f)\to H_n(X_\om,X_\om\setminus\{p_\om\},\f)$ is not injective.  
	\item There exists an asymptotic cone $X_\om$ of $X$ such that the limit $Q_\om$ of $Q$ is an $n$-flat which bounds a flat half-space in $X_\om$. 
	\item There exists an ultralimit $X_\om=\lim_{\om}(X,p_i)$ such that the limit $Q_\om$ of $Q$ is an $n$-flat which bounds a flat half-space in $X_\om$. 
	\een
	Then (1) and (2) are equivalent. If we assume in addition that $Q$ is a flat, then all three conditions are equivalent.
\end{prop}
Note that there is no scaling in condition (3).

\begin{proof}
	Clearly $(2)\Rightarrow(1)$. Now we prove $(1)\Rightarrow(2)$. By Lemma~\ref{lem_blow_up}, $X_\omega$ has an iterated blow up $Z$ where the limit of $Q_\om$ in $Z$ is a flat bounding a flat half-space. Now (2) follows as rescaling an asymptotic cone of $X$ gives another asymptotic cone of $X$, and
	ultralimits of asymptotic cones are asymptotic cones (cf. \cite[Chapter 10.7]{ggt}).

	Now we assume $Q$ is a flat. $(3)\Rightarrow(2)$ is clear. $(2)\Rightarrow(3)$ follows essentially from the argument in \cite{francaviglia2010large}. Asymptotic cones in \cite{francaviglia2010large} are defined using fixed base point, and $(3)\Rightarrow(2)$ can be seen as a version of \cite[Theorem B]{francaviglia2010large} with varying base point (we do not need local compactness). The idea is to consider compact sets $K_i=C_\om\times[0,a_i]$ in $X_\om$ with $C_\om$ being a top-dimensional unit cube in $Q_\om$. Let $D_i=C_\om\times\{a_i\}$. We approximate each $D_i\to X_\om$ by a sequence of continuous maps $f_{ij}$ into $X$. Take $a_i\to\infty$ and study the geometry of $\im f_{ij}$ and their projections on $Q$ will give the required half-flat. We refer to \cite[Section 3]{francaviglia2010large} for more details.
\end{proof}

The following results apply to all coarse median spaces \cite{bowditch2013coarse}, whose asymptotic cones are bilipschitz to $\cat(0)$ spaces. 
\begin{cor}
	\label{cor:bilip}	
	Let $X$ be a metric space such that any asymptotic cone of $X$ is bilipschitz to a $\cat(0)$ space. Let $Q$ be an $n$-quasiflat in $X$. Let $\f$ be a non-trivial abelian group. Then the following are equivalent:
	\ben
	\item  There exist an asymptotic cone $X_\om$ of $X$ and $p_\om\in Q_\om$ such that the map $H_n(Q_\om,Q_\om\setminus\{p_\om\},\f)\to H_n(X_\om,X_\om\setminus\{p_\om\},\f)$ is not injective.  
	\item There exists an asymptotic cone $X_\om$ of $X$ such that the limit $Q_\om$ of $Q$ is an $n$-dimensional bilipschitz flat which bounds a bilipschitz embedded flat half-space in $X_\om$. 
	\een
\end{cor}

\begin{proof}
	$(2)\Rightarrow (1)$ is clear. Now we show $(1)\Rightarrow (2)$.
	Let $Y_\omega$ be a $\cat(0)$ space with a bilipschitz map $f:Y_\om\to X_\omega$. Let $R_\om=f^{-1}(Q_\om)$. By Lemma~\ref{lem_blow_up}, there is an iterated blow up $Y'_\om$ of $Y_\om$ such that the limit $R'_\om\subset Y'_\om$ of $R_\om$ is a flat bounding a flat half-space. The map $f$ gives rise to a bilipschitz map $g:Y'_\om\to X'_\om$ where $X'_\om$ is an iterated blow up of $X_\om$. Let $Q'_\om\subset X'_\om$ be the limit of $Q_\om\subset X_\om$. As $R_\om=f^{-1}(Q_\om)$, we have $R'_\om=g^{-1}(Q'_\om)$. Thus $Q'_\om$ bounds a bilipschitz embedded flat half-space in $X'_\omega$. Moreover, $X'_\omega$ is an asymptotic cone of $X$ (cf. \cite[Chapter 10.7]{ggt}).
\end{proof}

\begin{cor}\label{cor_Morse_crit}
	Suppose $X$ is a proper $\cat(0)$ space. Suppose $F\subset X$ is a flat such that the stabilizer of $F$ in $\isom(X)$ acts cocompactly on $F$. Then $F$ is Morse if and only if $F$ does not bound an isometrically embedded half-flat.
\end{cor}

For an abelian group $\f$, we say an $n$-dimensional quasiflat $Q$ in $X$ is \emph{$\f$-Morse} if for any asymptotic cone $X_\om$ (with base points in $Q$), the inclusion $Q_\om\to X_\om$ induces injective map on the $n$-th local homology with coefficient $\f$ at each point in $Q_\om$.
\begin{cor}
	Suppose $X$ is a proper $\cat(0)$ space. Let $Q$ be a quasiflat and let $\f$ be a non-trivial abelian group. Then $Q$ is $\mathbb Z$-Morse if and only if $Q$ is $\f$-Morse.
\end{cor}

\subsection{Metrics on half-planes}
\label{subsec_halfplane}
It is natural to ask whether the ``halfspace criterion'' in Proposition~\ref{lem_coefficient intro} holds for more general metric spaces. We conjecture that for metric spaces with convex geodesic bicombings, Proposition~\ref{lem_coefficient intro} still holds with flat replaced by isometrically embedded normed vector spaces, and flat halfspace replaced by halfspaces in normed vector spaces. However, there is a limit on how far we can push this ``halfspace criterion''. Our goal in this subsection is to present an example which shows that naive generalizations of this criterion to spaces satisfying coning inequalities fail.

\begin{prop}
	\label{prop:halfplane}
	There exists a Riemannian metric $d$ on the upper half plane $\R^2_{\ge 0}$ such that
	\begin{enumerate}
		\item $X=(\R^2_{\ge 0},d)$ satisfies condition {\rm (CI$_{n}$)} for any $n$;
		\item the boundary $\ell$ of $\R^2_{\ge 0}$ is a quasi-geodesic which is not a Morse quasi-geodesic (it violates the super-Euclidean divergence condition);
		\item there does not exist an asymptotic cone $X_\om$ of $X$ (with base points in $\ell$) such that the limit $\ell_\om$ of $\ell$ bounds a bilipschitz embedded flat half plane.
	\end{enumerate}
\end{prop}

We will construct such metric on $\R^2_{\ge 0}$ in several steps (Definition~\ref{def_l_bad_ball}, Definition~\ref{def:very bad} and Definition~\ref{def:bad plane}).
\begin{definition}
	\label{def_l_bad_ball}
	Pick $L\in [1,\infty)$.  An {\bf  $L$-bad ball} is a smooth Riemannian metric $g_L$ on the ball $B(0,R_L)\subset\R^2$ with the polar coordinate form
	$$
	g_L=dr^2+(f(r)r)^2d\th^2\,,
	$$
	where:
	\ben
	\item $f:[0,R_L]\ra [1-\frac{1}{L},L]$ is smooth function.
	\item $\max f=L$.
	\item $f|_{[0,L]}\equiv 1$ and $f|_{[\frac{R_L}{10L},R_L]}\equiv 1$.
	\item (Slow change on annuli) \; $|\D_r(\log f(r))|<\frac{1}{Lr}$.
	\item (Standard area)\; $|B(0,R_L)|_{g_L}=\pi R_L^2$.
	\item $f(r)r$ is an increasing function.
	\een
\end{definition}

\begin{lemma}
	\label{lem_coning_inequality_bad_ball}
	There is a constant $A<\infty$ such that every $L$-bad ball satisfies a coning inequality with constant $A$, i.e.  every Lipschitz $1$-cycle $\si$ 
	contained in an $R$-ball has a filling $\tau$ with $\M(\tau)\leq AR\cdot\M(\si)$.
\end{lemma}
\begin{proof}
	Let $B_L$ be an $L$-bad ball and suppose $L\geq 2$.
	Let $R>0$ and denote by $\tilde A_o(R,2R)$ the annulus in the flat cone $C_\alpha$ with with tip $o$ and cone angle $\alpha=f(R)$.
	From the slow change of annuli, we see that for all $r\in[R,2R]$ holds $\frac{1}{2}\leq\frac{f(r)}{f(R)}\leq 2$.
	
	Let $\psi:A_0(R,4R)\to \tilde A_o(R,4R)$ be the natural radial isometric homeomorphism, which is arclength preserving on the inner boundary circles. 
	The control on $f$ ensures that $\psi$ is locally $4$-bilipschitz. Now let $\ga$ be a Jordan curve in $B_L$. Let $R$ be the smallest radius such that $\ga$
	is contained in $B_0(2R)$. We will distinguish two cases. Let $S\subset R^2$ be a flat sector with angle $\frac{\pi}{4}$ and tip at the origin. Choose $C>0$ such that there exists
	a point $p\in\R^2$ with $B_r(CR)\subset S$ and $\|p\|=2R$.
	
	Suppose the diameter of $\ga$ is less than $\frac{CR}{2}$. Note that the angle $\alpha$ is at least $\frac{1}{2}$ by our assumption on $L$ and the definition of $f$. 
	Then $\psi(\gamma)$ is contained in a ball of radius $CR$ centered at a point on the outer boundary of $\tilde A_o(R,4R)$. This concludes the first case by the cone inequality in $C_\alpha$ (note that by taking $C$ sufficiently large, we can assume any geodesic between two points in $\ga$ does not escape $A_0(R,4R)$, thus the diameter of $\ga$ can be computed inside $A_0(R,4R)$, similarly, we assume the diameter of $\varphi(\ga)$ is computed inside $\tilde A_o(R,4R)$).

	So let us assume the diameter of $\ga$ is at least $\frac{CR}{2}$. Then we obtain a filling by coning off $\ga$ at the origin. By condition (6) as above, the area of the cone is at most $R\cdot \mathcal{H}^1(\ga)$ and the proof is complete.
\end{proof}

\begin{remark}\label{rmk_doubling}
	The comparison to annuli in flat cones also shows that the doubling constant of an $L$-bad ball becomes unbounded as $L\to\infty$.
\end{remark}

\begin{definition}
	\label{def:very bad}
	Choose some $r_0>1$.  The {\bf very bad ball of scale $r_0$} is obtained from the Euclidean ball $B_0(r_0)\subset\R^2$ as follows.  
	Setting $a_k:=2^{-k}r_0$, $p_k:=(a_k,0)\in \R^2$, for we replace the Euclidean ball $B_{p_k}(\frac{a_k}{10})$ with a $k$-bad ball rescaled to have radius 
	$\frac{a_k}{10}$, provided the rescaled ball is Euclidean up to the radius $k$, or equivalently, $\frac{k}{R_k}\cdot \frac{a_k}{10}\ge k$ using the notation of Definition~\ref{def_l_bad_ball}. There are only finitely many $k$ such that $a_k\ge 10 R_k$ so we are replacing only finitely many balls.
	
	An \emph{expanded very bad ball} of scale $r_0$ is obtained by gluing the boundary of a very bad ball of scale $r_0$ to the inner boundary of the Euclidean annulus $A_0(r_0,2r_0)\subset\R^2$. 
\end{definition}

\begin{lem}
	\label{lem_doublingII}
	Let $r_k\to\infty$ and let $\{B_k\}$ be a sequence of bad balls of scale $r_k$. Let $Z$ be an ultralimit of the rescaled sequence $\{\frac{1}{r_k}B_k\}$. Then $Z$ is not doubling.
\end{lem}

\begin{proof}
	For any integer $k>0$, we see ultralimit of (properly scaled) $k$-bad balls in $Z$. Thus $Z$ contains isometrically embedded rescaled $k$-bad balls for any $k>0$. Hence $Z$ is not doubling by Remark~\ref{rmk_doubling}.
\end{proof}

\begin{definition}
	\label{def:bad plane}
	Modify the flat metric on the half-plane $\R^2_+$ as follows. For every pair of positive integers $j, k$, let $p_{jk}$ be the point of $\R^2_+$ with coordinate $(j2^k,2^k)$. We replace the ball $B_{p_{jk}}(10^{-1}2^k)$ with a very bad ball of scale $10^{-1}2^k$ for each pair $j,k\ge 0$. 
	Let $X$ denote the resulting Riemannian manifold. 
\end{definition}

\begin{remark}
	The key point for such an arrangement is that the size of the very bad ball is comparable to its height as well as to the size of ``Euclidean regions'' in between very bad balls.
\end{remark}

There are two metrics on $X$, one is induced by the Riemannian metric defined in Definition~\ref{def:bad plane}, denoted by $d_X$, and one is the original Euclidean metric, denoted by $d_{euc}$.

We claim that $X$ has the following properties. It satisfies the coning inequality for $1$-cycles. However, for any sequence $(x_k)$ on $\D X$ and any sequence
$\la_k\to\infty$, the asymptotic cone $\om\lim (\frac{1}{\la_k}X,x_k)$ is not doubling and therefore not bilipschitz to the Euclidean halfplane.

We will establish the claims on $X$ step by step.

\begin{lemma}\label{lem_bd_qg}
	The boundary $\D X$ is a quasigeodesic.
\end{lemma}

\begin{proof}
	The metric on $X$ is flat in a neighborhood of $\D X$ and therefore the canonical parametrization of $\D X$ is $1$-Lipschitz.
	On the other hand, the metric $g_L$ on an $L$-bad ball fulfills  
	\begin{equation}
	\label{eq:metric comparison}
	g_L\geq (1-\frac{1}{L})g_{euc}
	\end{equation}
	by definition. Hence $\D X$
	is a bilipschitz line.
\end{proof}

\begin{lemma}\label{lem_bilip_bddist}
	The distance to the boundary in $X$ is comparable to the Euclidean distance to the boundary
	\[\frac{1}{C}d_{euc}(\cdot,\D X)\leq d_{X}(\cdot,\D X)\leq C d_{euc}(\cdot,\D X).\]
\end{lemma}

\begin{proof}
	
	The left inequality follows from \eqref{eq:metric comparison}. To see the right inequality, we choose a point $p$ in $X$ and  consider the vertical segment $\ga$ joining $p$
	to a point $q$ on $\D X$ which realizes the Euclidean distance. Since $q$ lies on the boundray, it is not contained in a bad ball. Recall that the Euclidean metric on an $L$-bad ball is only altered on an
	inner core of radius $\frac{R_L}{10}$. If $\ga$ intersects the core of a bad ball disjoint from $p$, then we replace the segment inside the core by the shortest path along the boundary of the core. 
	This increases the Euclidean length of $\ga$ only by factor $\frac{\pi}{2}$.
	Now if $p$ itself lies in the core of an $L$-bad ball $B_L$, then we change the segment of $\ga$ inside $B_L$ to a piecewise radial geodesic. We first join $p$ to the center of $B_L$ and then
	join the center to the first intersection point of $\ga$ with $B_L$. The Euclidean length of the original segement is at least $\frac{9}{10}R_L$ whereas the length of the new path is at most $\frac{11}{10}R_L$.
	This concludes the proof as the length of the modified path in $X$ is the same as its Euclidean length.
\end{proof}

\begin{lemma}\label{lem_coning_inequality_halfspace}
	Suppose that expanded very bad balls satisfy {\rm (CI$_n$)} for any $n$ with constants independent of the scale of the very bad ball. Then $(X,d_X)$ satisfies {\rm (CI$_n$)} for any $n$.
\end{lemma}

\begin{proof}
	Let $\ga$ be a smooth Jordan curve in $X$. We distinguish two cases according to the size of $\ga$ relative to $\D X$.
	
	Suppose $d_X(\ga,\D X)\leq 100 \diam_X(\ga)$. Set $D_0=\diam_X(\ga)$ and $D=\diam_{euc}(\ga)$. By Lemma~\ref{lem_bilip_bddist} and \eqref{eq:metric comparison}, there exists $C$ independent of 
	$\ga$ such that $D\le CD_0$ and $d_{euc}(\ga,\D X)\leq 100C D_0$. Choose the minimal natural number $k$ such that $CD_0 \leq 2^{k}$. Now we choose the smallest Euclidean rectangle $P$ containing 
	$\ga$ with two vertices on $\D X$ and the other two at centers of very bad balls at hight $2^k$. By Definition~\ref{def_l_bad_ball} (5) and rotation symmetry of the metric on a bad ball, 
	we know that Area$(P)$ and length$(\D P)$ measured with respect to $d_{euc}$ and $d_X$ result the same value. The height and width of $P$ is $\le 200CD_0$ by the distribution of very bad balls, 
	thus we obtain Area$(P)\leq (200CD_0)^2$. Since the diameter of $\ga$ is always less than twice its length, we
	found the required filling.
	
	Suppose $d_X(\ga,\D X)> 100 \diam_X(\ga)$. Note that $\ga$ lies either in an entirely Euclidean region or intersects at least one very bad ball $B$. 
	In the latter case, $\ga$ is the expanded very bad ball $B'$ around $B$ (if $\ga$ escapes $B'$, then it travels through Euclidean regions whose size is comparable to the size of $B$, which is comparable to 
	$d_X(\ga,\D X)$ by Lemma~\ref{lem_bilip_bddist}; this gives a lower bound for $\diam_X(\ga)$ which contradicts $\diam_X(\ga) <\frac{1}{100} d_X(\ga,\D X)$). As $B'$ is surrounded by Euclidean regions, 
	$\diam_B(\gamma)=\diam_X(\ga)$. This finishes the proof by our assumption.
\end{proof}

The following is a consequence of estimates regarding the rectangle $P$ in the proof of Lemma~\ref{lem_coning_inequality_halfspace} and Lemma~\ref{lem_bilip_bddist}.
\begin{cor}\label{cor_bd_not_sd}
	The quasi-geodesic $\D X$ does not have super-Euclidean divergence. 
\end{cor}

\begin{lem}
	\label{lem:very bad ball CI}
	There is a constant $A<\infty$ such that every expanded very bad ball satisfies a coning inequality with constant $A$.
\end{lem}

\begin{proof}
	The proof is quite similar to Lemma~\ref{lem_coning_inequality_halfspace} as it can be viewed as a ``polar coordinate'' version of Lemma~\ref{lem_coning_inequality_halfspace}. 
	Let $x$ be the center of the very bad ball $B$ and $\gamma$ be a smooth Jordan curve in $B$. Again we consider two cases $d(p,\gamma)\ge 100\diam(\ga)$ and $d(p,\gamma)< 100\diam(\ga)$. 
	The details are left to the reader.
\end{proof}

\begin{proof}[Proof of Proposition~\ref{prop:halfplane}]
	It remains to prove (3). Let $\lim_{\om} (\frac{1}{\la_k}X,x_k)=X_\om$ be an asymptotic cone of $X$ where $(x_k)$ is a sequence on $\D X$ and $\la_k\to\infty$. 
	Let $\ell_\om$ (resp. $x_\om$) be the limit of $\D X$ (resp. $x_k$). We argue by contradiction and suppose $\ell_\om$ bounds a bilipschitz half plane $H_\om$. 
	
	Let $K$ be the closed upper half of the unit disk in $\R^2$. Suppose $\D K=s\cup s'$ where $s$ is straight and $s'$ is an arc on the unit circle. 
	Let $s\cap s'=\{v_1,v_2\}$. Let $f_\infty:K\to X_\om$ be a bilipschitz embedding such that $f_\infty(0)=x_\om$, $f_\infty(K)\subset H_\om$ and $f_\infty(s)\subset \D H_\om$. 
	Then there is a sequence of Lipschitz maps $f_k: K\to \frac{1}{\la_k}X$ such that $\lim_\om f_k=f_\infty$, $f_k(s)\subset \D X$ and $f_k(0)=x_k$. To construct such $f_k$'s, we take a fine enough triangulation of $K$, 
	approximating $f_\infty$ on the $0$-skeleton, then extending skeleton by skeleton using the coning inequality in $X$. By Lemma~\ref{lem_bd_qg} and Lemma~\ref{lem_bilip_bddist}, there exists $\de>0$ such that for $k$ 
	sufficiently large, $\frac{1}{\la_k}d_{euc}(f_k(s'),f_k(0))> \de$ and $f_k(v_1),f_k(v_2)$ are in different components of $\D X\setminus B_k$ where $B_k=\{x\in X\mid d_{euc}(x,f_k(0))\le \la_k\de\}$. 
	Thus $B_k\subset \im f_k$. Thus $\im f_k$ contains very bad balls of size comparable to $\diam(\im f_k)$. As $\im f_\infty=\lim_\om \im f_k$, we know $\im f_\infty$ contains an isometrically embedded copy of an ultralimit 
	of rescaled very bad balls. Thus $\im f_\infty$ is not doubling by Lemma~\ref{lem_doublingII}, which yields a contradiction.
\end{proof}

\begin{remark}
	We point out the following stronger result.	
	Let $Z$ and $B_k$ be as in Lemma~\ref{lem_doublingII}. We define the \emph{center} of $Z$ to be the limit of centers of the $B_k$'s. 
	Then a point $p\in X_\om\setminus \ell_\om$ either has an open neighborhood bilipschitz a small disk in $\R^2$, 
	or has a neighborhood isometric to a neighborhood of the center of $Z$, which is homeomorphic to $\R^2$. The proof is left to the interested reader.
\end{remark}

\begin{remark}
	It is also true that the ultralimit of the boundary $\D X$ above does not have full support.
	Pulling back Lipschitz maps from asymptotic cones as in the proof of  Proposition~\ref{prop:halfplane}, one can show that the coning inequality of $X$ passes on to its asymptotic cones.
	This is enough to conclude that the defintion of full support does not depend on the choice of homology theory. Hence we can conlcude from Corollary~\ref{cor_bd_not_sd} and Proposition~\ref{prop_bridge}.
\end{remark}

\appendix
\section{Some properties of quasiflats}

\subsection{Lipschitz quasiflats and quasi-retractions}
\label{subsec:Lipschitz quasifalts}
\begin{definition}
	Let $K\subset X$ be a closed subset. A map $\pi:X\to K$ is called a $\lambda$-quasi-retraction if it is $\lambda$-Lipschitz and the restriction
	$\pi|_K$ has displacement $\leq \lambda$.
\end{definition}

\begin{lemma}\label{lem_quasiflat_to_bilipschitzflat}
	Let $X$ be a length space and  let $\Phi:\E^n\to X$ be a  $L$-Lipschitz $(L,A)$-quasiflat with image $Q$. Then there exist $\bar L$ depending only on $L$ and $A$, a metric space $\bar X$, an $\bar L$-bilipschitz embedding $X\to \bar X$ and an $L$-Lipschitz retraction $\bar X\to X$ with the following additional properties. $\bar X$
	contains a $\bar L$-bilipschitz flat $\bar Q$ such that $d_H(Q,\bar Q)<L$ and $d_H(X,\bar X)<L$.

\end{lemma}

\begin{proof}
	We glue $\E^n\times[0,L]$ to $X$ along $\E^n\times\{L\}$ via $\Phi$. The resulting space equipped with the induced length metric
	is denoted by $\bar X$. We define $\bar Q=\mathbb E^n\times \{0\}$, viewed as a subset of $\bar X$.
	Note that for $x,y\in\bar Q$, $d_{\bar X}(x,y)<2L$ implies $d_{\bar X}(x,y)=d_{\E^n}(x,y)$. Moreover, if $d_{\E^n}(x,y)\geq 2L(2L+A)$, then $d_{\bar X}(x,y)\geq\frac{1}{2L}d_{\E^n}(x,y)$.
	Hence for points $x,y \in\E^n$ with $2L<d_{\E^n}(x,y)< 2L(2L+A)$ holds $d_{\bar X}(x,y)\geq\frac{1}{2L+A}d_{\E^n}(x,y)$ and the canonical embedding $\E^n\hookrightarrow \E^n\times\{0\}\subset\bar X$
	is $(2L+A)$-bilipschitz. The remaining statements follow immediately from the construction.
\end{proof}

\begin{corollary}\label{cor_quasiretraction}
	Let $X$ be a length space and  let $\Phi:\E^n\to X$ be an  $L$-Lipschitz $(L,A)$-quasiflat with image $Q$. Then there exist constants $\lambda_1$ and $\lambda_2$ which depend only on $L,A$ and $n$, and a $\lambda_1$-Lipschitz quasiretraction $\pi:X\to Q$ such that $d(x,\pi(x))\le \lambda_2$ for any $x\in Q$. Moreover, $\lambda_2\to 0$ as $A\to 0$.
\end{corollary}

\begin{proof}
	Choose a thickening $\bar X$ as in Lemma \ref{lem_quasiflat_to_bilipschitzflat} and denote by $\bar Q\subset \bar X$ the bilipschitz flat close to $Q$.
	By McShane's extension lemma we obtain a Lipschitz retraction $\bar X\to \bar Q$ where the Lipschitz constant is controlled by $L,A,n$. Composing with the natural projection
	$\bar X\to X$ we obtain the required map since $Q$
	is at distance $L$ from $\bar Q$.
\end{proof}

The following lemma is a consequence of Section~\ref{subsec_homotopy}.

\begin{lemma}
	\label{lem_homotopy}
	Let $X$ be a metric space with an $L'$-Lipschitz bicombing and let $Q\subset X$ an $n$-dimensional $(L,A)$-quasiflat with $\lambda$-Lipschitz quasiretraction $\pi:X\to Q$.
	Let $\sigma\in \I_n(X)$ and let $S=\spt(\sigma)$. Suppose that $h:[0,1]\times S\to X$ is the homotopy from $S$ to $\pi(S)$ induced by the bicombing on $X$.

	Let $\rho=\sup_{x\in S}\{d(x,\pi(x))\}$.
	Then there exists a constant $C>0$ depending only on $L',L,A$ and $n$ such that
	\[\M(h_\#(\llbracket 0,1\rrbracket\times \sigma))\leq C\cdot\lambda^n\cdot\rho\cdot\M(\si).\]
	
\end{lemma}

\begin{cor}
	\label{cor_homotopy}
	Let $Q,\pi,\lambda_1,\lambda_2$ be as in Corollary~\ref{cor_quasiretraction}. Suppose $X$ has $L'$-Lipschitz bicombing. Let $\sigma\in \I_n(X)$ such that $S=\spt(\sigma)\subset N_{2\rho}(Q)\setminus N_\rho(Q)$. Suppose that $h:[0,1]\times S\to X$ is the homotopy from $S$ to $\pi(S)$ induced by the bicombing on $X$. Suppose $\rho>\lambda_2$. Then there exists a constant $C>0$ depending only on $L',L,A$ and $n$ such that
	\[\M(h_\#(\llbracket 0,1\rrbracket\times \sigma))\leq C\cdot\lambda^n\cdot\rho\cdot\M(\si).\]
\end{cor}

\begin{proof}
	By Lemma~\ref{lem_homotopy}, it suffices to show $d(x,\pi(x))\le C\rho$ for any $x\in S$. Let $z$ be a point in $Q$ such that $\rho\le d(x,z)=d(x,Q)\le 2\rho$. Then $d(x,\pi(x))\le d(x,z)+d(z,\pi(z))+d(\pi(z),\pi(x))\le 2\rho+ \lambda_2+\lambda_1 d(x,z)\le 2\rho+ \rho+\lambda_1\cdot(2\rho)=(3+2\lambda_1)\rho$.
\end{proof}

\subsection{Proof of Proposition~\ref{lem_homology}}
\label{subsec_proof}
Using regular cubulations at scale $R$, one can decompose chains in $\R^n$ into pieces of size $R$. More generally, in a metric space $X$ we can decompose according to Lipschitz maps $\varphi:X\to\R^n$.
For integral chains of compact support in $X$ we are going to recursively define a weak surrogate of a simplicial decomposition subordinate to (the inverse image of) a given cubulation of $\R^n$.

Let us choose an orientation on $\R^n$. If $\mathcal{C}$ is a cubulation of $\R^n$, then each cube $B\in\mathcal{C}$ inherits an orientation of $\R^n$.
For a cube $B\in\mathcal{C}$, we define the \emph{face decomposition} of $\D B$ to be $\D B=\sum_i \eps_iB_i$ where each $B_i$ is a codimension 1 face of $B$ and the sign $\eps_i=\pm 1$ is determined by our orientation.
\begin{definition}
	\label{def:stratification}
	Let $\mathcal{C}$ be a regular cubulation of $\R^n$  at scale $R$ and let $\varphi:X\to\R^n$ be a Lipschitz map.
	For a constant $C>0$, a \emph{($C$-controlled) rectifiable stratification (at scale $R$)} of an $m$-cycle $S\in\bZ_{m,c}(X)$ is 
	a collection of currents $(S_{B})_{B\in\mathcal{C}}$ such that
	\begin{enumerate}
		\item  for every $B\in\mathcal{C}^{(n-k)}$ holds $S_B\in \I_{m-k,c}(X)$
		and $\|S_B\|$ is concentrated on $\varphi^{-1}(B^\circ)$ where $B^\circ$ denotes the interior of $B$;
		\item  for every $B\in\mathcal{C}^{(n)}$ holds $S_B=S\on \varphi^{-1}(B)$;
		\item for $B\in\mathcal{C}$, let $\D B=\sum_{i=1}\eps_i B_i$ be the face decomposition of $\D B$, then $\sum_{i=1} \eps_i S_{B_i}$ is a piece decomposition of $\D S_B$;
		\item $\sum_{B\in\mathcal{C}^{(n-k)}} \M(S_B) \le C\cdot\frac{\M(S)}{R^k}$. 
	\end{enumerate}
\end{definition}

Note (1) and (2) imply that  $S=\sum_{B\in\mathcal{C}^{(n)}}S_B$ is a piece decomposition.

If $S\in\bZ_{m,c}(X)$ has a rectifiable stratification $(S_{B})_{B\in\mathcal{C}}$, then we call the collection $(S_{B})_{B\in\mathcal{C}^{(n-k)}}$ the {\em codimension-$k$-skeleton}, $S^{(m-k)}$.

Before we turn to the existence of rectifiable stratifications, we provide an auxiliary lemma.

\begin{lemma}\label{lem_boundary_of_cube}
	Let $\varphi:X\to \R^k$ be a Lipschitz map and let $\si\in\bZ_{k}(X)$ be a cycle. Denote by $p_i:X\to \R$ the composition of $\varphi$ with the projection onto the $i$-th  coordinate axis.
	Suppose that 
	\begin{enumerate}
		\item $\slc{\si,p_i,1}$ and $\slc{\si,p_i,-1}$ exist and are elements in $\I_{k-1}(X)$ for all $1\leq i\leq k$, in particular $\|\si\|(p_i^{-1}\{\pm 1\})=0$;
		\item for any $i\neq j$ hold $\|\slc{\si,p_i,1}+\slc{\si,p_i,-1}\|(p_j^{-1}\{\pm 1\})=0$.
	\end{enumerate}
	Let $B=[-1,1]^k\subset\R^k$ and $B_i^\pm=[-1,1]^{i-1}\times\{\pm 1\}\times[-1,1]^{k-i}$ with their interior denoted by $B^\circ$ and $(B_i^\pm)^\circ$. Then the face decomposition of $\D B=\sum_{i=1}^k (B_i^+ -B_i^-)$ induces a piece decomposition 
	\[\D (\si\on \varphi^{-1}(B^\circ))=\sum_{i=1}^k \left(\slc{\si,p_i,1}\on \varphi^{-1}((B_i^+)^\circ)-\slc{\si,p_i,-1}\on \varphi^{-1}((B_i^-)^\circ)\right).\]
\end{lemma}

\begin{proof}
	
	We will use the abbreviation $\slc{\si,p_i,s}|_{-1}^1=\slc{\si,p_i,1}-\slc{\si,p_i,-1}$.
	We want to prove the following slightly more general statement.
	\[\D (\si\on \bigcap\limits_{i\in I}p_i^{-1}\{(-1,1)\})=\sum_{j\in I} (\slc{\si,p_j,s}|_{-1}^1\on \bigcap\limits_{i\in I\setminus \{j\}}p_i^{-1}\{(-1,1)\}).\]
	
	Let us introduce the characteristic functions $\chi_i:=\chi|_{\{-1< p_i< 1\}}$. Then we can write $\si\on\bigcap\limits_{i\in I}p_i^{-1}\{(-1,1)\}=\si\on\prod_{i\in I} \chi_i$. 
	
	We induct on the cardinality of $I$. The case $|I|=1$ is clear.
	
	Suppose $I=\{i_1,\ldots,i_m\}$. For $1\le j\le m$, let $I_j=I\setminus \{i_j\}$. By induction, $\si\on\prod_{i\in I_1}\chi_i$ is a normal current.
	\begin{align*}
	&\D(\si\on\prod_{i\in I} \chi_i)=\slc{\si\on\prod_{i\in I_1}\chi_i,p_{i_1},s}|_{-1}^1+(\D(\si\on\prod_{i\in I_1}\chi_i))\on\chi_{i_1}\\
	&=\slc{\si\on\prod_{i\in I_1}\chi_i,p_{i_1},s}|_{-1}^1+ \sum_{j\in I_1}(\slc{\si,p_j,s}|_{-1}^1\on \prod_{i\in I_1\setminus \{j\}}\chi_i)\on \chi_{i_1}\\
	&=\slc{\si\on\prod_{i\in I_1}\chi_i,p_{i_1},s}|_{-1}^1+ \sum_{j\in I_1}(\slc{\si,p_j,s}|_{-1}^1\on \prod_{i\in I\setminus\{j\}}\chi_i)\ .
	\end{align*}
	
	There are $m$ terms in the above summation, denoted by $S_1,\ldots,S_m$ from left to right. Note that for $j>1$, $\|S_j\|$ is concentrated on the set 
	$$A_j=(\bigcap\limits_{i\in I_j}p_i^{-1}\{(-1,1)\})\cap p_j^{-1}\{\pm 1\}\ .$$
	
	We claim $\|S_1\|$ is concentrated on $A_1$. Recall $\|\si\|(\{p_{i_1}=\pm1\})=0$, then $\|\si\on\prod_{i\in I_1}\chi_i\|(\{p_{i_1}=\pm1\})=0$ by \cite[Lemma 4.7]{Lan3}. By induction,
	$$\D (\si\on\prod_{i\in I_1}\chi_i)=\sum_{j\in I_1}\slc{\si,p_j,s}|_{-1}^1\on \prod_{i\in I_1\setminus \{j\}}\chi_i.$$
	Thus by assumption (2), $\|\D(\si\on\prod_{i\in I_1}\chi_i)\|(\{p_{i_1}=\pm1\})=0.$ Then $$\slc{\si\on\prod_{i\in I_1}\chi_i,p_{i_1},1}=\slc{\si\on\prod_{i\in I_1}\chi_i,p_{i_1},1+}.$$ 
	As $\|\si\on\prod_{i\in I_1}\chi_i\|$ is concentrated on 
	$\bigcap\limits_{i\in I_1}p_i^{-1}\{(-1,1)\}$, Lemma~\ref{lem_slice} implies that  $\|\slc{\si\on\prod_{i\in I_1}\chi_i,p_{i_1},1}\|$ is concentrated on $\bigcap\limits_{i\in I_1}p_i^{-1}\{(-1,1)\}\cap\{p_{i_1}=1\}$.
	Thus $\|S_1\|$ is concentrated on $A_1$ as claimed.
	
	Note that $A_{j_1}\cap A_{j_2}=\emptyset$ when $j_1\neq j_2$, so the above sum is a piece decomposition.
	
	Now we repeat the discussion with $I_1$ replaced by $I_2$ to obtain another piece decomposition. By comparing these, we conclude 
	$$\slc{\si\on\prod_{i\in I_1}\chi_i,p_{i_1},s}|_{-1}^1=\slc{\si,p_{i_1},s}|_{-1}^1\on \prod_{i\in I\setminus\{i_1\}}\chi_i,$$ which finishes the induction.
\end{proof}

\begin{proposition}\label{prop_rec_strat}
	Let $\varphi:X\to \R^n$ be an $L$-Lipschitz map. Then there exists $C=C(L,n)>0$ such that the following holds.
	For any $S\in\bZ_{m}(X)$ and $R>0$ there exists a regular cubulation $\mathcal{C}$ of $\R^n$ at scale $R$ such that $S$
	has a $C$-controlled rectifiable stratification subordinate to $\mathcal{C}$.
\end{proposition}

\begin{proof}
	The action of $R\cdot \Z^n$ on $\R^n$ induces a 1-Lipschitz covering map $h:\R^n\to T$ where $T$ is a product of $n$-circles of length $R$. 
	For a subset $I=\{i_1,i_2,\ldots,i_k\}\subset \{1,2,\ldots,n\}$, let $\R^I\subset \R^n$ and $T^I\subset T$ be the associated subspaces. 
	Denote by $p_I:X\to \R^I$ the composition of $\varphi$ with the projection $\R^n\to \R^I$. Let $\pi_I=h\circ p_I$.
	
	First we claim there exist a point $(s_1,s_2,\ldots,s_n)\in T$ and a constant $C=C(L,n)$ such that for any collection of mutually disjoint subsets $I_1,I_2,\ldots,I_k$ of $\{1,\ldots,n\}$, we have
	\begin{enumerate}
		\item[(i)] $\si:=\left\langle\ldots\left\langle\left\langle S,\pi_{I_1},s_{I_1}\right\rangle,\pi_{I_2},s_{I_2}\right\rangle\ldots,\pi_{I_k},s_{I_k}\right\rangle\in \I_{m-|I|}(X)$;
		\item[(ii)] permuting the order of $I_1,\ldots,I_k$ in the definition of $\si$ will result in the same current;
		\item[(iii)] $\si=\slc{S,\pi_I,s_I}$ where $I=\cup_{i=1}^k I_i$;
		\item[(iv)] $\M(\si)\le C\cdot \frac{\M(S)}{R^{|I|}}$.
	\end{enumerate}
	Here $|I|$ denotes the cardinality of $I$. To see the claim, note that by repeatedly applying \cite[Theorem 6.5]{Lan3} and Fubini, 
	we can find a full measure subset $A_0\subset T$ such that (i), (ii) and (iii) hold. To arrange (iv), by the coarea formula, for each subset $I$ of $\{1,\ldots,n\}$, we can find a subset $A_I\subset A_0$ such that
	$\L^n(A_I)\ge (1-\eps)\L^n(A_0)$ and $\M(\si)\le C_\eps\cdot\frac{\M(S)}{R^{|I|}}$ whenever $(s_1,\ldots,s_n)\in A_I$. By choosing $\eps$ sufficiently small (depending on $n$), 
	the intersection of all $A_I$ with $I$ ranging over subsets of $\{1,\ldots,n\}$ is non-empty and the claim follows.
	
	Let $\mathcal{C}$ be the cubulation (of scale $R$) with vertex set $h^{-1}((s_1,\ldots,s_n))$. 
	Let $B\in\mathcal{C}^{(n-k)}$ be a cube and let $I:=I_B$ be the smallest subset of $\{1,\ldots,n\}$ such that $h(B)\subset T^I$, in particular $|I|=n-k$. We define $S_B=\slc{S,\pi_I,s_I}\on \varphi^{-1}(B)$ and $S_{B^\circ}=\slc{S,\pi_I,s_I}\on \varphi^{-1}(B^\circ)$.
	We claim that the collection $(S_B)_{B\in\mathcal{C}}$ satisfies Definition~\ref{def:stratification}.
	
	Definition~\ref{def:stratification} (2) holds by definition and (4) follows from item (iv). We claim $S_B=S_{B^\circ}$. If $B$ is top-dimensional, then this claim follows from $\|S\|(\pi_i^{-1}(s_i))=0$ for any $i\in\{1,\ldots,n\}$. As we slice at regular point when defining iterated slices, a similar argument implies the lower dimensional case of the claim, which implies the second half of Definition~\ref{def:stratification} (1). Now Definition~\ref{def:stratification} (3) follows from Lemma~\ref{lem_boundary_of_cube} and properties (i)-(iii). Note that (iv) implies $S_B$ has finite mass for each $B$, hence $\D S_B$ has finite mass by (3). The first half of Definition~\ref{def:stratification} (1) follows from (3), (i), (ii) and (iii).
\end{proof}

Now we are ready to prove Proposition~\ref{lem_homology}.

\begin{proof}[Proof of Proposition~\ref{lem_homology}]
	We are going to use the Federer Fleming deformation, henceforth refered to as (FFD), with respect to the cubulation $\mathcal{C}_{R_0}$.
	By Proposition \ref{prop_rec_strat}, there exists a second regular cubulation $\mathcal{C}$ of $\R^n$ at the larger scale $R$ such that $S$
	has a rectifiable stratification $(S_B)_{B\in\mathcal{C}}$ subordinate to $\mathcal{C}$.
	
	We will write $a\lesssim b$ if $a\le C'b$ for a constant $C'$ depending only on $L,A,m,n$ and $c$.
	
	We only prove the case $m\le n$ and the $m>n$ case is similar.
	
	First we claim that there exists a constant $C_0=C_0(L,A,n)$ and a family of currents $(P_B)_{B\in\mathcal{C}}$ such that 
	\begin{enumerate}[label=(\alph*)]
		\item for every $B\in \mathcal{C}^{(n-k)}$ holds $P_B\in \I_{m-k,c}(\R^n)$ and $\spt(P_B)\subset N_{C_0}(\varphi(\spt(S_B)))$;
		\item each $P_B$ is a cubical chain with respect to $\mathcal{C}_{R_0}$;
		\item for $B\in\mathcal{C}$, let $\D B=\sum_{i} \eps_iB_i$ be the face decomposition of $\D B$, then $\D P_B=\sum_{i} \eps_iP_{B_i}$;
		\item $\sum_{B\in\mathcal{C}^{(n-k)}} \M(P_B) \le C_0\cdot\frac{\M(S)}{R^k}$;
		\item $\Fill(\varphi_\# S - \sum_{B\in \mathcal{C}^{(n)}}P_B)\le C_0\cdot \M(S)$.
	\end{enumerate}
	
	We define $(P_B)_{B\in\mathcal{C}}$ inductively as follows. Let $\tau_B=\varphi_\# S_B$. Take $S_B$ with $B\in \mathcal{C}^{(n-m)}$. 
	Then $S_B$ is a 0-dimensional cycle. Applying  (FFD) to $\tau_B$ with respect to the cubulation $\mathcal{C}_{R_0}$, we obtain a cubical chain $P_B$ and a 
	homotopy $h_B$ with $\D h_B=P_B -\tau_B$. Note that $\M(P_B)\lesssim \M(S_B)$ and $\M(h_B)\lesssim R_0 \M(S_B)$. 
	Clearly $\sum_{B\in\mathcal{C}^{(n-m)}} \M(P_B) \lesssim \frac{\M(S)}{R^m}$ and $\sum_{B\in\mathcal{C}^{(n-m)}} \M(h_B) \lesssim R_0\cdot \frac{\M(S)}{R^m}\lesssim \frac{\M(S)}{R^{m-1}}$.

	Suppose $P_B$ and $h_B$ with $B\in \mathcal{C}^{(n-m+k-1)}$ are already defined such that conditions $(a)-(d)$ hold, $\D h_B=P_B-\tau_B$ and $$\sum_{B\in\mathcal{C}^{(n-m+k-1)}} \M(h_B) \lesssim \frac{\M(S)}{R^{m-k}}.$$
	Take $B$ with $B\in \mathcal{C}^{n-m+k}$. Let $\D B=\sum_i \eps_iB_i$ be the face decomposition. Define $P'_B=\tau_B+\sum_i \eps_i h_{B_i}$. Then $\D P'_B=\sum_{i} \eps_iP_{B_i}$. 
	Applying FFD to $P'_B$ with respect to $\mathcal{C}_{R_0}$, we obtain a cubical chain $P_B$ and a homotopy $h_B$ with $\D h_B=P_B-\tau_B$ (note that $\D P'_B$ is fixed when applying the radial push-out procedure of FFD to $P'_B$). Then
	$$
	\sum_{B\in\mathcal{C}^{(n-m+k)}} \M(P'_B) \lesssim \sum_{B\in\mathcal{C}^{(n-m+k)}} \M(\tau_B)+\sum_{B\in\mathcal{C}^{(n-m+k-1)}} \M(h_B)
	\lesssim \frac{\M(S)}{R^{m-k}}
	$$
	and 
	$$
	\sum_{B\in\mathcal{C}^{(n-m+k)}} \M(h_B) \lesssim R_0 \sum_{B\in\mathcal{C}^{(n-m+k)}} \M(P'_B)
	\lesssim \frac{\M(S)}{R^{m-k-1}}\ ({\rm as}\ R\ge R_0)
	$$
	and 
	$$\sum_{B\in\mathcal{C}^{(n-m+k)}} \M(P_B)\lesssim \sum_{B\in\mathcal{C}^{(n-m+k)}} \M(P'_B)\lesssim \frac{\M(S)}{R^{m-k}}.$$
	Then the claim follows.
	
	For each $P_B$, define $T_B=\iota(P_B)$ where $\iota$ is defined with respect to $\mathcal{C}_{R_0}$ (see Proposition~\ref{prop_chain_map}). Define $P=\sum_{B\in \mathcal{C}^{(n)}} P_B$ and $T=\sum_{B\in \mathcal{C}^{(n)}}T_B=\iota(P)$. 
	Now we construct $H\in \I_{m+1}(X)$ such that $\D H=S-T$.
	
	We start with $B$ with $B\in \mathcal{C}^{(n-m)}$. By the estimates in Definition~\ref{def_retraction}, $S_B-T_B$ is a cycle of diameter $\lesssim R$. We apply the strong cone inequality to obtain $H_B$ such that 
	$\diam(\spt (H_B))\lesssim R$ and $\M(H_B)\lesssim R(\M(S_B)+\M(T_B))$. Thus $\sum_{B\in \mathcal{C}^{(n-m)}}\M(H_B)\lesssim R \cdot\frac{\M(S)}{R^{m}}\lesssim \frac{\M(S)}{R^{m-1}}$.
	
	Suppose $H_B$ with $B\in \mathcal{C}^{(n-m+k-1)}$ are already defined such that 
	\begin{enumerate}
		\item $\diam(\spt (H_B))\lesssim R$;
		\item $\sum_{B\in \mathcal{C}^{(n-m+k-1)}}\M(H_B)\lesssim \frac{\M(S)}{R^{m-k}}$;
		\item $\D H_B=S_B-T_B-\sum_i \eps_iH_{B_i}$ where $\D B=\sum_i B_i$ is the face decomposition.
	\end{enumerate}
	
	Take $B$ with $B\in \mathcal{C}^{(n-m+k)}$. Consider $\si_B=S_B-T_B-\sum_i \eps_iH_{B_i}$. Let $\D B_i=\sum_{ij} \eps_{ij}B_{ij}$ be the face decomposition of $B_i$. 
	Then $\D\si_B=\sum_i\eps_iS_{B_i}-\sum_i\eps_iT_{B_i}-\sum_i\eps_i(S_{B_i}-T_{B_i}-\sum_{ij} \eps_{ij}B_{ij})=0$ by the sign convention. Applying the strong cone inequality to $\si_B$ we obtain a filling $H_B$. 
	As $\diam(\spt (H_{B_i}))\lesssim R$ for each $i$, we have $\diam(\spt (\si_B))\lesssim R$, hence $\diam(\spt (H_B))\lesssim R$. Moreover, $\M(H_B)\lesssim R\cdot (\M(S_B)+\M(T_B)+\sum_{B_i}\M(H_{B_i})$. 
	Thus by Proposition~\ref{prop_rec_strat} and the previous claim,
	\begin{align*}
	&\sum_{B\in \mathcal{C}^{(n-m+k)}}\M(H_B)\lesssim  R\cdot \sum_{B\in \mathcal{C}^{(n-m+k)}}\si_B\lesssim R\cdot \sum_{B\in\mathcal{C}^{(n-m+k)}}(\M(S_B)+\M(T_B))\\ &+ R\cdot\sum_{B\in\mathcal{C}^{(n-m+k-1)}}\M(H_B)\lesssim \frac{\M(S)}{R^{m-k-1}}\,.
	\end{align*}
	All items of the conclusions of the proposition now follow.
\end{proof} 

\section{Neck property and coarse neck property}
\label{sec_short_cut}
The purpose of this section is to provide a short cut from the neck property to coarse neck property under stronger assumptions on the ambient metric spaces, which leads to a weaker version of Theorem~\ref{thm:equivalence}.
\begin{lemma}
	\label{lem_sequence}
	Let $(X_k,Q_k,r_k,p_k,\tau_k)$ be a sequence and $D>0$, $\rho<1$ constants, such that:
	\bit
	\item $X_k$ is a complete metric space with an $L'$-Lipschitz bicombing.
	\item $Q_k\subset X_k$ is an $n$-dimensional $(L,A)$-quasiflat containing the base point $p_k\in Q_k$.
	\item $r_k\ra\infty$.
	\item for any asymptotic cone $X_\omega$ arising from the rescaled sequence $(\frac{1}{r_k}X_k,p_k)$, $Q_\omega$ has the neck property in $X_\omega$ (cf. Definition~\ref{def_neck_decomposition}) with uniform constant $C'$;
	\item $\tau_k\in \I_n(X_k)$, $\spt(\tau_k)\subset B(p_k,r_k)\setminus N_{\rho r_k}(Q_k)$ and $\M(\tau_k)\leq D r_k^n$.
	\item $\si_k:=\D \tau_k$ satisfies $\spt (\si_k)\subset N_{2\rho r_k}(Q_k)$ and $\M(\si_k)\leq D r_k^{n-1}$.
	\eit
	Then there exists a constant $C$ depending only on $C',L',L,A$ and $n$ such that the following holds for all sufficiently large $k$.
	
	\[	\Fill(\si_k)	\le C\rho r_k\M(\sigma_k).\]
	
\end{lemma}

\begin{proof}
	Since $X_k$ has a Lipschitz bicombing, we may assume that the quasiflat $Q_k$ is represented by a Lipschitz quasi-isometric embedding. Let $X'_k=\frac{1}{r_k}X_k$ and $Q'_k=\frac{1}{r_k}Q_k$. 
	We still use $\tau_k$ and $\si_k$ to denote the rescaled currents. By Lemma~\ref{lem_compactness}, we can pass to a regularized sequence $(\tau_k')$ with $\tau_k'\in\I_{n,c}(X'_k)$ such that
	$\tau'_k$ comes with mass bounds and support control and such that $\tau'_k$ is homologous to $\tau_k$ in a small tubular neighborhood of $\spt(\tau_k)$. 
	Moreover, we may assume that there is a compact metric space $Z$, a chain $\tau\in\I_{n,c}(Z)$ 
	and isometric embeddings $\psi_k:\spt(\tau_{k})\to Z$ such that $\psi_{k\#}\tau_{k}$ converges to $\tau$ in the flat distance in $\ell^\infty(Z)$. We set $\si=\D\tau$.
	
	Let $X_\omega$ be an ultralimit of $X'_k$ and $Q_\om\subset X_\om$ be the ultralimit of the $Q'_k$'s. We view $\spt(\tau)$ with the induced metric from $Z$ as a subset of $X_\omega$. 
	Moreover, we may assume that the $X'_k$'s and $X_\omega$ are isometrically embedded in $Z$ by replacing $Z$ by the metric space obtained from gluing $X_\omega$ and $X'_k$'s to 
	$Z$ along $\spt (\tau_\omega)$ and $\spt(\tau_k)$'s (\cite[Lemma I.5.24]{bridson_haefliger}). Also, we can assume that $Z$ has a convex geodesic bicombing by replacing $Z$ by $\ell^\infty(Z)$. 
	By construction, inside $X_\om\subset Z$ we know $\tau$ is supported in $X_\om\setminus N_\rho(Q_\omega)$ and $\si_\om=\partial\tau_\om$ is supported in $N_{2\rho}(Q_\om)$.

	Let  $\pi_{k}:Z\to Q'_k$ be a Lipschitz quasi-retraction as in Corollary~\ref{cor_quasiretraction}. 
	Then there exists $L_0$ depending only on $L$ and $A$, such that $\pi_{k}$ is $L_0$-Lipschitz. We use $\Fill_{U}$ to denote the filling volume inside a space $U$.
	By assumption we have $\Fill_{X_\omega}(\si)\le C' \rho\M(\si)$. Since $\si_k$ converges to $\si$ with respect to flat distance in $Z$, we conclude $\Fill_{Z}(\si_k)\leq 2C'\rho\M(\si_k)$ for all $k$
	large enough. Hence, for all such $k$ holds $\Fill_{Q_k'}(\pi_{k\#}\si_k)\leq 2L_0^n C'\rho\M(\si_k)$. By Corollary~\ref{cor_homotopy}, there exists a constant $C_1=C_1(L',L,A,n)$ such that 
	$\Fill_{X'_k}(\si_k-\pi_{k\#}\si_k)\le C_1\rho \M(\si_k)$ and the  proof is complete.
\end{proof}

\begin{cor}
	\label{cor_small_neck}
	Suppose $X$ has an $L'$-Lipschitz bicombing and $Q\subset X$ is an $(L,A)$ quasiflat. If for any asymptotic cone $X_\om$ of $X$ (with base points in $Q$) $Q_\om$ has the neck property in $X_\om$, then $Q$ has 
	the coarse neck property in $X$.
\end{cor}

\begin{thm}
	\label{thm_weak_equivalence}
	Let $Q:\R^n\to X$ be an $(L,A)$-quasiflat in a proper metric spaces $X$ with a Lipschitz bicombing. Then the following conditions are equivalent.
	\begin{enumerate}
		\item $Q$ is $(\mu,b)$-rigid (cf. Definition~\ref{def_rigid}).
		\item $Q$ has super-Euclidean divergence (cf. Definition~\ref{def_divergence2}).
		\item $Q$ has coarse neck property (cf. Definition~\ref{def_tcp}).
		\item $Q$ has coarse piece property (cf. Definition~\ref{def_neckp}).
	\end{enumerate}
\end{thm}

\begin{proof}
	$(2)\Rightarrow(3)$ follows from Proposition~\ref{prop_bridge}, Proposition~\ref{prop_cone_conditions} and Corollary~\ref{cor_small_neck}. $(1)\Rightarrow(2), (3)\Rightarrow(4)$ and $(4)\Rightarrow(1)$ is the same as the proof of Theorem~\ref{thm:equivalence}.
\end{proof}
Here we lose the dependence of the constants in various definitions as in Theorem~\ref{thm:equivalence}, as the argument go through asymptotic cones.

\section{Simpler proof for the Morse lemma}
\label{subsec:work with cones}

In this section we provide simpler proofs of Proposition~\ref{prop_divergence} and Proposition~\ref{prop_Morse_lemma_for_quasidisks} under the stronger assumption of Lipschitz bicombing. Using Proposition~\ref{prop_bridge} and Lemma~\ref{lem_rig_iff_supdiv}, these results will be a consequence of Proposition~\ref{prop_divergence1} and Proposition~\ref{prop_Morse_lemma_for_quasidisks1} below, where we take $\mathcal{X}$ to be the collection of all complete metric spaces with $L$-Lipschitz bicombing and $\mathcal{Q}$ to be the collection of all $n$-dimensional quasiflats in such metric spaces with $\delta$-super-Euclidean divergence for some given $\si$.

\begin{prop}
	\label{prop_divergence1}
	Let $\mathcal{X}$ be a collection of complete metric spaces with $L$-Lipschitz bicombing and let $\mathcal{Q}$ be a family of $n$-dimensional $L$-Lipschitz $(L,A)$-quasiflats in $\mathcal{X}$. Take $X\in\mathcal{X}$ and let $Q'\subset X$ be an $L$-Lipschitz $(L,A)$-quasiflats (not necessarily of dimension $n$). Suppose that $\mathcal{Q}$ has asymptotically full support property with respect to the singular homology.
	Then there exist $A_0>0$ and $\eps_0>0$ depending only on $\mathcal{X},L,A,n$ and $\mathcal{Q}$ such that for each $Q\in\mathcal{Q}$, either $d_H(Q,Q')\le A_0$, or 
	\begin{equation}
	\label{eq_linear_div1}
	\limsup_{r\to\infty} \frac{d_H(B_p(r)\cap Q,B_p(r)\cap Q')}{r}\ge \eps_0
	\end{equation}
	for some (hence any) $p\in Q$.	
\end{prop}

\begin{proof}
	This follows from Lemma~\ref{lem_bilip_uniqueness} and Lemma~\ref{lem_divergence} below.
\end{proof}

\begin{lemma}\label{lem_bilip_uniqueness}
	Let $X$ be a metric space with a Lipschitz bicombing. Let $F, F'\subset X$ be bilipschitz embeddings of closed convex subsets of $\R^m$ with $\dim F=n$. Suppose $F$ has full support in $X$.
	If $F$ is contained in a tubular neighborhood of 
	$F'$ and $\D F\subset \D F'$, then $F$ is equal to $F'$. In particular, their dimensions coincide.
\end{lemma}

\begin{proof}
	We only treat the case where $\D F$ is nontrivial and not compact, the case of empty boundary or compact boundary is either similar or easier.
	After rescaling we may assume $F\subset N_1(F')$. Then $n=\dim(F)\leq \dim(F')=n'$. 
	Let $\pi':X\to F'$ be an $L$-Lipschitz
	retraction. Choose a point  $p\in F\setminus\D F$. 
	For large $r>0$, choose a relative singular cycle $\si_r\in C_{n-1}(F\setminus\{p\},\D F)$ such that its image avoid $B_p(r)$ and $\si_r$ represents a non-trivial class in $\tilde H_{n-1}(F\setminus\{p\},\D F)$. Choose a chain $\beta_r\in C_{n-1}(\D F)$ filling $\D \si_r$, then $\si_r-\beta_r$ represents a non-trivial class in $\tilde H_{n-1}(F\setminus\{p\})$.
	Let $h_r$ denote the singular $n$-chain
	induced by the Lipschitz bicombing such that $\D h_r=\si_r-\si'_r$ where $\si'_r:=\pi'_*\si_r$. 
	If  $\tau'_r$ is an $n$-chain in $F'$ filling $\si'_r-\beta_r$, then $\tau'_r+h_r$
	fills $\si_r-\beta_r$. Since $F$ has full support, $p$ has to lie in the image of $\tau'_r+h_r$, and if $r$ is large enough, even in the image of $\tau'_r$. Hence $p\in F'$ and therefore $F\subset F'$.
	But then $F=F'$ because $F$ has full support.
\end{proof}

The next lemma is similar to \cite[Divergence Lemma 4.1]{kapovich1997quasi}.

\begin{lemma}\label{lem_divergence}
	Let $\mathcal{X},\mathcal{Q}, Q',Q,X$ be as in Proposition~\ref{prop_divergence1}. There exists a constant $\rho>0$ depending only on the constants $L, A, n, \mathcal{X}$ and $\mathcal{Q}$ with the following property. 
	If $p\in Q$ is such that 
	$d(p,Q')\geq \rho$, then there exists a point $q\in Q\cap B_p(\rho^2)$ with $d(q,Q')\geq d(p,Q')+1$. 
\end{lemma}

\begin{proof}
	Suppose for contradiction  that there are sequences $(Q_k)\in\mathcal{Q}$ and $(Q'_k)$ of $(L,A)$-quasiflats  in $X_k\in\mathcal{X}$ such that there are
	points $p_k\in Q_k$ with  $\lambda_k=d(p_k,Q'_k)\to\infty$ and such that $Q_k\cap B_{p_k}(\lambda_k^2)\subset N_{\lambda_k+1}(Q_k')$.
	We pass to an asymptotic cone $(X_\om,p_\om)=\om\lim(\frac{1}{\lambda_k}\cdot X_k,p_k)$. 
	Then $d(p_\om,Q'_\om)=1$ and $Q_\om\subset N_1(Q'_\om)$. As $Q_\om$ has full support in $X_\om$, by Lemma \ref{lem_bilip_uniqueness}, 
	$Q_\om=Q'_\om$. Contradiction.
\end{proof}

\begin{proposition}\label{prop_Morse_lemma_for_quasidisks1}
	Let $\mathcal{X}$ be a family of complete metric spaces with $L$-Lipschitz bicombing. Let $\mathcal{D}$ be a family of $n$-dimensional $L$-Lipschitz $(L,A)$-quasidisks in $\mathcal{X}$ such that $\mathcal{D}$ has asymptotic full support property with respect to the singular homology. Let $D'$ be some $n$-dimensional $L$-Lipschitz $(L,A)$-quasidisk.
	Then for every constant $c$ there exists a constant  $C$ depending only on $c, L, A, n, \mathcal{X}$ and $\mathcal{D}$
	such that for each $D\in\mathcal{D}$, $d_H(\D D,\D D')<c$ implies  $d_H(D, D')<C$.
\end{proposition}

\begin{proof}
	Suppose for contradiction that we can find sequences $D_k\in \mathcal{D}$ and $D'_k$ of $n$-dimensional $L$-Lipschitz $(L,A)$-quasidiscs such that $d_H(\D D_k,\D D'_k)<c$ but $d_H( D_k, D'_k)\to\infty$. Choose points $x_k\in D_k$ at maximal distance from $D'_k$ and set $\lambda_k=d(x_k,D'_k)$. Suppose $D_k\subset X_k\in\mathcal{X}$.
	Consider the ultralimit $(X_\om,x_\om)=\om \lim (\frac{1}{\lambda_k}\cdot X_k,x_k)$. The ultralimits $D_\om$ and $D_\om'$ of the quasidiscs are bilipschitz embeddings of closed convex subsets of $\R^n$
	which fulfill $D_\om\subset N_1(D'_\om)$. As $D_\om$ has full support, Lemma~\ref{lem_bilip_uniqueness} yields $D'_\om=D_\om$. This contradicts $d(x_\om,D'_\om)=1$.
\end{proof}

\section{Examples of Morse quasiflats}

\subsection{Examples in weakly special cube complexes}
We recall the following definition from \cite{huang_quasiflat}, which is a weak version of special cube complexes introduced in \cite{haglund_wise_special}.
\begin{definition}
	\label{weakly special}
	A cube complex $D$ is \textit{weakly special} if
	\begin{enumerate}
		\item $D$ is non-positively curved.
		\item No hyperplane of $D$ \textit{self-osculates} or \textit{self-intersects}.
	\end{enumerate}
\end{definition}
The notions of self-osculation and self-intersection were introduced in \cite[Definition 3.1]{haglund_wise_special}. The key property we need about weakly special cube complexes is the following, which was proved in \cite[Section 5.2]{huang_quasiflat}.

\begin{lem}
	\label{label}
	Let $D$ be a compact weakly special cube complex and let $G=\pi_1(D)$. Let $\tilde{D}$ be the universal cover of $D$. Then there exists a finite index subgroup $G'\le G$, and a labeling and orientation of edges in $\tilde{D}$ such that
	\begin{enumerate}
		\item The labeling and orientation of edges are invariant under the $G'$-action.
		\item If two edges of $\tilde{D}$ are parallel, then they have the same label, and their orientation is compatible with the parallelism.
		\item If two different edges of $\tilde{D}$ intersect at a vertex, then they have different labels.
	\end{enumerate}
\end{lem}

Pick a base vertex $x\in\tilde{D}$. For every edge path $\omega$ in $\tilde{D}$ starting at $x$, we define $L(\omega)$ to be the word $e^{\epsilon_{i_{1}}}_{i_{1}} e^{\epsilon_{i_{2}}}_{i_{2}} e^{\epsilon_{i_{3}}}_{i_{3}}\cdots$ where $e_{i_{j}}$ is the label of the $j$--th edge in $\omega$ and $\epsilon_{i_{j}}=\pm 1$ records the orientation of the $j$--th edge. Moreover, Lemma \ref{label} implies that if two edge paths starting at $x$ correspond to the same word, then they are equal.

\begin{lem}
	\label{parallel}
	Let $G',D$ and $\tilde{D}$ be as in Lemma \ref{label}. Suppose there is a convex subcomplex $W\subset\tilde{D}$ which splits as a product $W=W_1\times W_2$. Pick an element $g'\in G'$ such that there exist vertices $w_1\in W_1$ and $w_2\in W_2$ such that $g'$ maps vertex $(w_1,w_2)\in W$ to another vertex in $W$ of form $(w_1,w'_2)$. Then $g'$ maps $W_1\times \{w_2\}$ to $W_1\times \{w'_2\}$. Moreover, $g'$ restricted on $W_1\times \{w_2\}$ is exactly the parallelism map between $W_1\times \{w_2\}$ and $W_1\times \{w'_2\}$.
\end{lem}

\begin{proof}
	Let $w'_1\in W_1$ be a vertex and let $\omega\subset W_1$ be an edge path from $w_1$ to $w'_1$. By Lemma \ref{label} (2) and (3), $\omega\times\{w_2\}$ and $\omega\times\{w'_2\}$ correspond to the same word. Now Lemma \ref{label} (1) implies that $g'(\omega\times\{w_2\})$ and $\omega\times\{w'_2\}$ are two edge paths which start at the same point, and correspond to the same word. Thus $g'(\omega\times\{w_2\})=\omega\times\{w'_2\}$ and the lemma follows.
\end{proof}

\begin{thm}
	\label{thm_virtually_special}
	Suppose $G$ is virtually the fundamental group of a compact weakly special cube complex. Then each highest abelian subgroup of $G$ is a Morse quasiflat.
\end{thm}

\begin{proof}
	We assume without loss of generality that $G$ is the fundamental group of a compact weakly special cube complex $D$. Also by Lemma \ref{label}, we assume there is a $G$-invariant labeling and orientation of edges in $\tilde{D}$ satisfying the conclusion of Lemma \ref{label}. Let $A\le G$ be a highest abelian subgroup (note that $A$ is free), and let $k$ be the rank of $A$. Let $E\subset\tilde{D}$ be an $A$-invariant $k$-flat such that $A$ acts on $E$ properly and cocompactly. By sliding $E$ and possibly replacing $A$ by a finite index subgroup, we can assume $E$ is not contained in any hyperplane of $\tilde{D}$. It suffices to show $E$ is a Morse flat in $\tilde{D}$. By Corollary \ref{cor_Morse_crit}, it suffices to show $E$ does not bound a $(k+1)$-half-flat in $\tilde{D}$. We argue by contradiction and assume $E$ bounds a $(k+1)$-half-flat $F\subset\tilde{D}$. Note that $F$ is also not contained in a hyperplane.
	
	Let $C_F$ (resp. $C_E$) be the combinatorial convex hull of $F$ (resp. $E$). Let $\h(F)$ (resp. $\h(E)$) be the collection of hyperplanes of $\tilde{D}$ that intersect $F$ (resp. $E$). Then there is a 1-1 correspondence between hyperplanes in $C_F$ (resp. $C_E$) and hyperplanes in $\h(F)$ (resp. $\h(E)$). Since no hyperplane of $\tilde{D}$ contains $E$ or $F$, we have a decomposition $\h(F)=\h_1\sqcup \h_2$, where $\h_1$ is the set of hyperplanes that intersect $F$ in a $k$-half-flat, and $\h_2$ is the set of hyperplanes that intersect $F$ in a $k$-flat. Note that if $h\in\h_1$, then $h\cap F$ is transverse to $E$; if $h\in \h_2$, then $h\cap F$ is parallel to $E$. Thus $\h_1=\h(E)$. The decomposition $\h(F)=\h(E)\sqcup \h_2$ induces a product decomposition $C_F\cong C_E\times C^{\perp}_E$.
	
	Since $A$ is highest, by Theorem 2.1 and Theorem 3.6 of \cite{wise2017cubical}, $C_E$ and $E$ have finite Hausdorff distance. However, $F\subset C_E\times C^{\perp}_E$, thus $C^{\perp}_E$ can not be bounded. Pick a base vertex $w_1\in C_E$, then there exist a pair of vertices $w_2\neq w'_2$ of $C^{\perp}_E$ such that $(w_1,w_2)$ and $(w_1,w'_2)$ are mapped to the same vertex under the covering map $\tilde{D}\to D$. Thus there is an element $g\in G$ such that $g(w_1,w_2)=(w_1,w'_2)$.
	
	We claim $g$ commutes with any element in $A$. It suffices to show for any $a\in A$, $a\circ g (w_1,w_2)=g\circ a(w_1,w_2)$. Since $a$ leaves $C_E$ invariant, Lemma \ref{parallel} implies that $a(C_F)=C_F$, the action of $a$ respects the product decomposition $C_F\cong C_E\times C^{\perp}_E$ and $a$ acts by identity on the $C^{\perp}_E$ factor. Since $g(w_1,w_2)=(w_1,w'_2)$, Lemma \ref{parallel} implies $g|_{C_E\times \{w_2\}}$ is the parallelism map from $C_E\times \{w_2\}$ to $C_E\times\{w'_2\}$. Thus $a\circ g (w_1,w_2)=g\circ a(w_1,w_2)$.
	
	We claim the intersection of $A$ with the subgroup $\langle g\rangle$ generated by $g$ is the identity element. Let $P=C_E\times W$ be the combinatorial parallel set of $C_E$. Then we can identify $C^{\perp}_E$ as a convex subcomplex of $W$. Since $g(C_E\times \{w_2\})=C_E\times\{w'_2\}$, $g(P)=P$. By Lemma \ref{parallel}, $g$ respects the product decomposition $P=C_E\times W$ and acts as identity on the $C_E$ factor. If $g^{n}\in A$, then $g^{n}(C_E)=C_E$. However, $g^n$ must fix $C_E$ pointwise by previous discussion. Thus $n=0$ since $G$ is torsion free.
	
	The previous two claims imply that $A$ is contained in an abelian group of higher rank generated by $g$ and $A$, and this contradiction finishes the proof.
\end{proof}

\subsection{An example arising from branched covering}
\label{subsec:example}

A truncated hyperbolic space of dimension $\geq 3$ provides an example of a CAT(0) space which contains Morse-geodesics and flats of dimension $\geq 2$.
Taking products we obtain examples of CAT(0) spaces which contain Morse-flats which are not top-dimensional.

We are going to twist this example a bit to produce a smooth irreducible example.

Let $\bar M$ be a finite volume cusped hyperbolic manifold of dimension $n\geq 3$. Suppose that $\bar M$ contains a separating closed hypersurface $\bar N$ which is totally geodesic \cite{RaTsch}, \cite{LoRe}. 
We remove the cusps of $\bar M$ and deform the metric conformally near the boundary, leaving the metric unchanged in a neighborhood of $\bar N$, to obtain $\bar M'$, 
which is negatively curved in the interior and such that each component of its boundary is
a totally geodesic flat torus. The double $M$ of $\bar M'$ contains a finite family of totally geodesic flat hypersurfaces and a closed totally geodesic hyperbolic hypersurface $N$.
Then there exist finite coverings 
\[\beta:V\to M\times M\] 
of any degree, branched along $N\times N$. The pull-back metric on $V$ is locally CAT(0) \cite[Section 4.4]{Gromov_hyp} and can be smoothed near $\beta^{-1}(N\times N)$ to a metric of nonpositive sectional curvature \cite{FoSch}.

Denote by $\pi_V:\tilde V\to V$ and $\pi_{M}:\tilde M\to M$ 
the universal covers of $V$ respectively $M$. We obtain an induced covering map between universal covers 
\[\tilde\beta:\tilde V\to\tilde M\times\tilde M\] 
which branches along $(\pi_{M}\times\pi_{M})^{-1}(N\times N)$.

Since $\tilde M$ is a space with isolated $(n-1)$-flats, the product space $\tilde M\times\tilde M$ has isolated $(2n-2)$-flats. Further, a pair of geodesic in $\tilde M$ yields a $2$-flat in $\tilde M\times\tilde M$.
By the criterion in Proposition \ref{lem_coefficient}, we see that if none of these geodesics lies in one of the isolated flats, then the associated $2$-flat is Morse.

Now we turn to $\tilde V$. Since the top-dimensional tori in $M\times M$ are disjoint from $ N\times N$, the corresponding top-dimensional flats in $\tilde M\times\tilde M$ lift to flats in $\tilde V$. In particular, $\tilde V$
contains flats of dimension $\geq 4$.
Similarly, if we take two closed geodesics in $M$, both of them disjoint from $N$ and the flat hypersurfaces, then their product yields an immersed flat torus in $M\times M$. The inverse image of this torus under
$(\pi_{M}\times \pi_{ M})\circ\tilde\beta$ is a discrete family of $2$-flats, and each of these has to be Morse by Corollary \ref{cor_Morse_crit}.

\bibliography{morse_quasiflats}
\bibliographystyle{alpha}

\end{document}